\documentclass[]{amsart}
\usepackage{amsmath,amsthm,amsmath,amsrefs}

\usepackage{calc}
\usepackage{setspace}
\usepackage{amsbsy,amsmath,amsfonts,amsthm,amssymb,color, lineno}
\usepackage[colorlinks,citecolor=blue,urlcolor=pink]{hyperref} 
\usepackage[margin=1in]{geometry}
\numberwithin{equation}{section}
\theoremstyle{plain}
\newtheorem{theorem}{Theorem}[section] 
\newtheorem{lemma}[theorem]{Lemma}
\newtheorem{definition}[theorem]{Definition}
\newtheorem{corollary}[theorem]{Corollary}
\newtheorem{remark}[theorem]{Remark}
\newtheorem{proposition}[theorem]{Proposition}

\theoremstyle{definition}
\newtheorem{example}[theorem]{Example}
\usepackage{mathrsfs}

\newcommand{\ee}{\varepsilon}

\newcommand{\bN}{\mathbb{N}}
\newcommand{\cT}{\mathcal{T}}
\newcommand{\cN}{\mathcal{N}}

\newcommand{\cF}{\mathcal{F}}

\newcommand{\cC}{\mathcal{C}}

\newcommand{\cS}{\mathcal{S}}
\newcommand{\cB}{\mathcal{B}}
\newcommand{\bT}{\mathbb{T}}
\newcommand{\cK}{\mathcal{K}}

\newcommand{\la}{\langle}
\newcommand{\cL}{\mathcal{L}}
\newcommand{\ra}{\rangle}

\newcommand{\cH}{\mathcal{H}}
\newcommand{\cA}{\mathcal{A}}

\newcommand{\bofh}{\cB(\cH)}
\newcommand{\cP}{\mathcal{P}}
\newcommand{\cM}{\mathcal{M}}
\newcommand{\cU}{\mathcal{U}}
\newcommand{\bC}{\mathbb{C}}
\newcommand{\bR}{\mathbb{R}}

\newcommand{\cR}{\mathcal{R}}
\newcommand{\id}{\text{id}}

\newcommand{\bZ}{\mathbb{Z}}

\newcommand{\cJ}{\mathcal{J}}

\newcommand{\cE}{\mathcal{E}}

\newcommand{\tr}{\text{tr}}
\newcommand{\Tr}{\text{Tr}}
\newcommand{\spn}{\text{span }}
\newcommand{\mb}[1]{\mathbf{#1}}

\usepackage{tikz}
\usepackage{tikz-cd}

\title[The quantum-to-classical graph homomorphism game]{The quantum-to-classical graph homomorphism game}
\author{Michael Brannan, Priyanga Ganesan and Samuel J. Harris}

\address{Department of Mathematics\\ 
	Texas A\&M University \\
	College Station, TX 77840\\
	USA} 
\email[]{mbrannan@tamu.edu, priyanga.g@tamu.edu, sharris@tamu.edu}

\begin{document}

\begin{abstract}
Motivated by non-local games and quantum coloring problems, we introduce a graph homomorphism game between quantum graphs and classical graphs. This game is naturally cast as a ``quantum-classical game"--that is, a non-local game of two players involving quantum questions and classical answers. This game generalizes the graph homomorphism game between classical graphs. We show that winning strategies in the various quantum models for the game is an analogue of the notion of non-commutative graph homomorphisms due to D. Stahlke \cite{Sta16}. Moreover, we present a game algebra in this context that generalizes the game algebra for graph homomorphisms given by J.W. Helton, K. Meyer, V.I. Paulsen and M. Satriano \cite{HMPS19}. We also demonstrate explicit quantum colorings of all quantum complete graphs, yielding the surprising fact that the algebra of the $4$-coloring game for a quantum graph is always non-trivial, extending a result of \cite{HMPS19}.
\end{abstract}

\maketitle

In recent years, the theory of non-local games has risen to a level of great prominence in quantum information theory and related parts of physics and mathematics.  In quantum information theory, non-local games provide a convenient framework in which one can exhibit the advantages of using quantum entanglement as a resource to accomplish certain tasks.  In physics, non-local games are intimately tied to the study Tsirelson's correlation sets and Bell's work on local hidden variable models \cite{Ts93}.  Within mathematics, the theory of non-local games has led to some spectacular developments in the field of operator algebras.  Most notable here is the work of Junge-Navascues-Palazuelos-Perez-Garcia-Scholz-Werner \cite{J+}, T. Fritz \cite{Fr12} and N. Ozawa \cite{Oz13} connecting the Connes-Kirchberg conjecture to Tsirelson's correlation sets in quantum information.  Very recently, Ji-Natarajan-Vidick-Wright-Yuen \cite{J+20} used non-local games to provide a counterexample to the Connes-Kirchberg conjecture.  Another recent and quite remarkable application of non-local games in mathematics is the work of Man$\check{\text{c}}$inska-Roberson \cite{MR19} which uses a non-local game, called the graph isomorphism game, to provide a quantum interpretation of pairs of graphs that admit the same number of homomorphisms from planar graphs.  

The general setup of a (classical input, classical output) two player non-local game is given in terms of a tuple $\mathcal G = (I, O, V)$, where $I$ and $O$ are finite sets and $V:O \times O \times I \times I \to \{0,1\}$ is a predicate function which determines the rules of the game.  The game is played by two cooperating players, Alice and Bob, and a verifier (Referee).  Each round proceeds by the verifier (randomly) selecting a pair of questions $(x,y) \in I \times I$ and sending $x$ to Alice and $y$ to Bob.  Alice and Bob then respond with answers $(a,b) \in O \times O$.  The verifier declares the round won if $V(a,b,x,y) = 1$ and declares it lost if $V(a,b,x,y) = 0$.  The term {\it non-local} refers to the fact that during each round, Alice and Bob are spatially separated and are unable to communicate; neither Alice nor Bob knows which questions/answers the other received/returned.  This non-locality of $\mathcal G$ makes winning each round of the game (with high probability) generally very difficult.  It is in these scenarios that ``quantum strategies'' (which make use of some shared entangled resource between Alice and Bob) can allow the players to drastically improve their performance by better correlating their behaviors.

In this paper, we are mainly interested in a non-local game called the graph homomorphism game and certain extensions of it.
The graph homomorphism game is a well studied example of a non-local game \cite{MR14, TT20, PSSTW16,HMPS19}.  This game is described by a pair of finite simple graphs $G,H$, with input set $I = V(G)$ (the vertex set of $G$) and output set $O = V(H)$.  The goal of Alice and Bob in this game is to convince the referee that there exists a homomorphism $G \to H$.  In particular, the rules of the game are determined by the following two requirements: 
\begin{enumerate}
\item
Alice and Bob's answers must be {\it synchronous}, meaning that if they receive the same vertex $x \in V(G)$, then they must return the same vertex $a \in V(H)$.
\item
If the referee supplies an edge $(x,y) \in E(G)$ to Alice and Bob, then they must respond with an edge $(a,b) \in E(H)$.
\end{enumerate}
The graph homomorphism game (in particular, the special case of the graph coloring game) has led to many developments in the operator algebraic aspects of non-local games.  A particular notion of interest here is the notion of a {\it synchronous} non-local game and synchronous strategies for such games \cite{HMPS19}.  Winning strategies for synchronous games turn out to be completely described in terms of traces on a certain $\ast$-algebra associated to the game, bringing to bear many powerful operator algebraic techniques in the theory of non-local games.  

Within information theory (both quantum and classical) graph theory plays a central role, appearing quite naturally in the theory of zero-error communication in the form of confusability graphs of noisy communications channels.  If the channel at hand is classical, the confusability graph is a finite simple graph on the input alphabet whose edges indicate which letters can be confused after passing through the channel.  If the channel is genuinely quantum, it was shown in \cite{DSW13} that the role of the confusability graph in this case must be played by more general structure called a {\it quantum graph}. Quantum graphs are an operator space generalization of classical graphs, which have emerged in different disguises in
operator systems theory, non-commutative topology and quantum information theory. Traditionally, a quantum graph is viewed as an operator system that serves as a quantum generalization of the adjacency matrix. It was first introduced in \cite{DSW13} for studying a zero-error channel capacity problem and arose independently in the study of quantum relations \cite{weaver1, weaver 2} around the same time. An alternate approach was used in \cite{MRV18} to define a quantum graph using a quantum adjacency matrix acting on a finite-dimensional C$^\ast$-algebra which plays the role of functions on the vertex set. Both these perspectives are shown to be essentially equivalent \cite[Theorem 7.7]{MRV18} and offer different advantages and perspectives.

In the present work,  motivated by several recent works extending the notion of chromatic number from graphs to the setting of quantum graphs \cite{Sta16,PT15, PSSTW16, KM17}, our aim is to develop a non-local game that captures the coloring problem for quantum graphs.  To this end, we study homomorphisms from quantum graphs to  classical graphs, using a non-local game with quantum inputs and classical outputs. The inputs are quantum inputs, in the sense that the referee initializes the state space $\bC^n \otimes \bC^n$, where Alice has access to the left copy and Bob has access to the right copy of $\bC^n$. Alice and Bob are allowed to share a(n entanglement) resource space $\cH$ in some prepared state $\psi$. After receiving the input $\varphi$ on $\bC^n \otimes \bC^n$, they can perform measurements on the triple tensor product $\bC^n \otimes \cH \otimes \bC^n$, and respond to the referee with classical outputs based on their measurements.

The winning strategies for this game give rise to a notion of quantum graph homomorphism that is closely related to several notions of quantum graph homomorphism in the literature \cite{Sta16, MRV18, weaver2}.
We also construct a game $\ast$-algebra for this and show that this game algebra extends the game algebra for graph homomorphisms given in \cite{HMPS19}. Further, we consider the coloring game for quantum graphs and study the associated chromatic numbers.  We show interesting extensions of classical results in this framework.  In particular, we use unitary error basis tools to show that every quantum graph admits a finite chromatic number in the quantum model (but not necessarily the local model), and the fact that every quantum graph is 4-colorable in the algebraic model.

The organization of the paper is as follows: Section \ref{sec: general theory} develops the general theory of quantum input-classical output correlations and the various quantum models which give rise to such correlations.  Here we also introduce and study the universal operator system $\mathcal Q_{n,c}$ associated to such correlations and its C$^\ast$-envelope.  In section \ref{sec: sync cor}, we introduce a generalization of synchronous correlations to our quantum framework.  In particular, we establish in this section characterizations of synchronous correlations in terms of tracial states on $C^{\ast}$-algebras, and we also establish an extension of the well-known equality of the quantum and quantum-spatial correlation sets for synchronous correlations, extending a result of \cite{KPS18}.   In section \ref{sec: afd cor} we consider the structure of quantum approximate correlations in our context, extending the result of \cite{KPS18} by identifying  synchronous quantum approximate correlations with the closure of the synchronous quantum correlations.  In section \ref{sec: game} we define the homomorphism game from quantum graphs to classical graphs and study the corresponding winning strategies and game $\ast$-algebra. Finally, in section \ref{sec: coloring}, we study the coloring problem for quantum graphs, demonstrating explicit colorings of all quantum graphs in the $q$-model with the help of some quantum teleportation-like schemes, as well as extending classical results on algebraic colorings to this framework.

\subsection*{Acknowledgements}  The authors are grateful to Marius Junge, Vern Paulsen, Ivan Todorov, Lyudmila Turowska, Nik Weaver, and Andreas Winter for fruitful discussions related to the content of this paper.  We are especially grateful to Ivan Todorov and Lyudmila Turowska, who shared with us an early draft of their preprint \cite{TT20}, which independently obtains some of the results in this paper.   MB and PG were partially supported by NSF Grant DMS-2000331 and a T3 Grant from Texas A\&M University.  SH was partially supported by an NSERC postdoctoral fellowship.  

\section{Quantum Input, Classical Output Correlations} \label{sec: general theory}

In this section, we develop some general theory on non-local games with quantum questions and classical answers. These have already been used in the two-output context of quantum XOR games \cite{RV15,Ha17}.

To motivate things, first recall that in the classical setup of $n$ classical input, $c$ classical output two-player non-local games, the main objects of study are the bipartite correlation sets $C(n,c) \subset \bR^{n^2c^2}$ which model the players' behavior.  Namely, any element $P = (p(a,b|x,y))_{\substack{1 \le a,b \le c\\ 1 \le x,y  \le n}} \in C(n,c)$ specifies the probability $p(a,b|x,y)$  that the players Alice and Bob return answers $a$ and $b$ (respectively), given that they received questions $x$ and $y$ (respectively).  The correlations  (behaviors) $P \in C(n,c)$ that are physically relevant are the ones that can be realized by a {\it (quantum) strategy}, that is, by Alice and Bob performing joint measurements on a quantum mechanical system prepared in some  initial state.  Mathematically, a quantum strategy amounts to the data of two finite-dimensional Hilbert spaces $\cH_A$ and $\cH_B$, and families of positive operator-valued measure (POVMs) $\{P_1^x,...,P_c^x\}$ on $\cH_A$, $\{Q_1^y,...,Q_c^y\}$ on $\cH_B$, and a state $\chi \in \cH_A \otimes \cH_B$.  From this data, one obtains a correlation $P \in C(n,c)$ via the formula 
\[
p(a,b|x,y) = \langle \chi| P^x_a \otimes Q^y_b| \chi\rangle.
\]

The subset of all correlations obtainable from quantum strategies as above is dented by $C_q(n,c)$.  In a similar manner, one can define other classes of correlations (local, quantum spatial, quantum approximate, quantum commuting) that are built from of the corresponding classes of strategies.  (See, for example, \cite{KPS18} for a review of all of these models.)

Our goal now is to develop the analogous notion of the correlation set $C(n,c)$ and its various subclasses arising from quantum strategies.  The main idea is quite simple -- in order to allow for quantum questions, we replace the question set $[n] \times [n]$ with the set of quantum states on the bipartite system $\bC^n \otimes \bC^n$.  In the following, our approach is somewhat backwards, in that we first define the different strategies associated to a two-player scenario with quantum questions (on $\bC^n \otimes \bC^n$) and classical answers in $\{1,2,...,c\}$.  Afterwards, we consider the associated correlations.  For our purposes it is easiest to begin with the quantum (i.e., finite-dimensional tensor product) strategies.

A \textbf{quantum strategy}, or a {\bf $q$-strategy}, is given by two finite-dimensional Hilbert spaces $\cH_A$ and $\cH_B$, a POVM $\{P_1,...,P_c\}$ on $\bC^n \otimes \cH_A$, a POVM $\{Q_1,...,Q_c\}$ on $\cH_B \otimes \bC^n$, and a state $\chi \in \cH_A \otimes \cH_B$.

A \textbf{quantum spatial strategy}, or a {\bf $qs$-strategy}, is given in the same way as a $q$-strategy, except that we no longer assume that $\cH_A$ and $\cH_B$ are finite-dimensional.

A \textbf{quantum commuting strategy}, or a {\bf $qc$-strategy}, is given by a single  Hilbert space $\cH$, a POVM $\{P_1,...,P_c\}$ on $\bC^n \otimes \cH$, a POVM $\{Q_1,...,Q_c\}$ on $\cH \otimes \bC^n$, and a state $\chi \in \cH$, with the property that $(P_a \otimes I_n)(I_n \otimes Q_b)=(I_n \otimes Q_b)(P_a \otimes I_n)$ for all $a,b$.

\begin{remark}
	It is helpful to understand the above commutation condition in terms of block matrices. For $1 \leq a \leq c$, one may write $P_a=(P_{a,ij}) \in M_n(\bofh)$ with $P_{a,ij} \in \bofh$. Similarly, we may write $Q_b=(Q_{b,k\ell}) \in M_n(\bofh)$ with $Q_{b,k\ell} \in \bofh$. With this in mind, the above commutation relation is easily seen to be equivalent to the requirement that $[P_{a, ij}, Q_{b, kl}] = 0 \in \bofh$ for each $a,b, i,j,k,l$. (See, e.g., \cite{CLP17} and \cite{HP17}.)
\end{remark} 

Finally, in view of the above remark, we define a \textbf{local strategy}, or a \textbf{classical strategy}, as a quantum commuting strategy with the property that the set of operators  $P_{a,ij}$ and $Q_{b,k\ell}$ generate a commutative $C^*$-algebra.

Suppose now that the referee initializes $\bC^n \otimes \bC^n$ in the state $\varphi$. For a quantum strategy, the probability that Alice outputs $a$ and Bob outputs $b$ is given by
$$p(a,b|\varphi)=\la (P_a \otimes Q_b)(\varphi \odot  \chi),\varphi \odot \chi \ra,$$
where by $\varphi \odot \chi$ we mean the (permuted) state in $\bC^n \otimes (\cH_A \otimes \cH_B) \otimes \bC^n$ rather than on $\bC^n \otimes \bC^n \otimes (\cH_A \otimes \cH_B)$. For a quantum commuting strategy, we simply replace $\cH_A \otimes \cH_B$ with $\cH$ and $(P_a \otimes Q_b)$ with $(P_a \otimes I_n)(I_n \otimes Q_b)$. We note that this definition of the probability of outputs can easily be extended to other (e.g. mixed) states in $\bC^n \otimes \bC^n$ that may not be included in the definition of the game. This is because the probabilities corresponding to Alice and Bob's strategy are encoded entirely in the \textit{correlation} associated to their strategy. The \textbf{correlation} associated to the strategy $(P_1,...,P_c,Q_1,...,Q_c,\chi)$ with $n$-dimensional quantum inputs and $c$ classical outputs is given by the tuple
$$X:=(X^{(a,b)}_{(i,j),(k,\ell)})=\left( (\la(P_{a,ij} \otimes Q_{b,k\ell})\chi,\chi\ra)_{i,j,k,\ell} \right)_{a,b} \in (M_n \otimes M_n)^{c^2},$$
in the case when the entanglement resource space for Alice and Bob is of the form $\cH_A \otimes \cH_B$. In the case when their resource space is a single Hilbert space $\cH$, we replace $P_{a,ij} \otimes Q_{b,k\ell}$ with $P_{a,ij}Q_{b,k\ell}$.

We will let $\mathcal{Q}_q(n,c)$ be the set of all correlations of this form that arise from quantum strategies. In other words,
$$\mathcal{Q}_q(n,c)=\{ (\la (P_{a,ij} \otimes Q_{b,k\ell})\chi,\chi \ra)_{\substack{1 \leq i,j,k,\ell \leq n, \\ 1 \leq a,b \leq c}} \subseteq (M_n \otimes M_n)^{c^2},$$
where $\cH_A$ and $\cH_B$ are finite-dimensional Hilbert spaces; $P_{a,ij} \in \cB(\cH_A)$ are such that $P_a=(P_{a,ij}) \in M_n(\cB(\cH_A))$ is positive with $\sum_{a=1}^c P_a=I$; $Q_{b,k\ell} \in \cB(\cH_B)$ are such that $Q_b=(Q_{b,k\ell}) \in M_n(\cB(\cH_B))$ are positive with $\sum_{b=1}^c Q_b=I$, and $\chi \in \cH_A \otimes \cH_B$ is a state.

Similarly, we will let $\mathcal{Q}_{qs}(n,c)$ be the set of all quantum spatial correlations (where $\cH_A$ and $\cH_B$ may not be finite-dimensional), and we let $\mathcal{Q}_{qc}(n,c)$ be the set of all quantum commuting correlations of the above form (where we replace the tensor product space $\cH_A \otimes \cH_B$ with a single Hilbert space $\cH$, and $P_{a,ij} \otimes Q_{b,k\ell}$ with $P_{a,ij}Q_{b,k\ell}$). Keeping the analogy with the sets $C_t(n,k)$ corresponding to classical inputs, we will also define $\mathcal{Q}_{qa}(n,c)$ as the closure of $\mathcal{Q}_q(n,c)$ in the norm topology. Lastly, we define $\mathcal{Q}_{loc}(n,c)$ as the set of all quantum commuting correlations where $C^*(\{P_{a,ij},Q_{b,k\ell}: 1 \leq a,b \leq c, \, 1 \leq i,j,k,\ell \leq n\})$ is a commutative $C^*$-algebra.

Since each of the correlation sets above are defined in terms of POVMs, an argument involving direct sums shows that $\mathcal{Q}_t(n,c)$ is convex for all $t \in \{loc,q,qs,qa,qc\}$. Moreover, $\mathcal{Q}_{qa}(n,c)$ is closed (by definition) and an application of Theorem \ref{theorem: characterizing qc} shows that $\mathcal{Q}_{qc}(n,c)$ is closed. Similarly, Proposition \ref{proposition: loc is closed} shows that $\mathcal{Q}_{loc}(n,c)$ is closed.

Next, we define a universal operator system that encodes the above correlation sets. Define $\mathcal{Q}_{n,c}$ as the universal operator system generated by $c$ sets of $n^2$ entries $q_{a,ij}$ with the property that the matrix $Q_a=(q_{a,ij})$ is positive in $M_n(\mathcal{Q}_{n,c})$ for each $1 \leq a \leq c$ and $\sum_{a=1}^c Q_a=I_n$. The correlations above are directly related to states on certain operator system tensor products of $\mathcal{Q}_{n,c}$. For these results, we will use some facts about $\mathcal{Q}_{n,c}$ and its $C^*$-envelope. For convenience, we define $\mathcal{P}_{n,c}$ to be the universal unital $C^*$-algebra generated by $c$ sets of $n^2$ entries $p_{a,ij}$ such that $P_a=(p_{a,ij})$ is an orthogonal projection in $M_n(\mathcal{P}_{n,c})$ for each $1 \leq a \leq c$ and $\sum_{a=1}^c P_a=I_n$. Similarly, we define $\mathcal{B}_{n,c}$ as the universal unital $C^*$-algebra generated by $n^2$ entries $u_{ij}$ with the property that $U=(u_{ij}) \in M_n(\cB_{n,c})$ is a unitary of order $c$. The latter algebra is an obvious quotient of the Brown algebra $\cB_n$, which is the universal $C^*$-algebra generated by the entries of an $n \times n$ unitary. The algebra $\cB_n$ first appeared in \cite{Br81}.

Our goal is to show a quantum-classical version of the disambiguation theorems; that is, we will show that all correlations in $\mathcal{Q}_t(n,c)$ can be achieved using projection-valued measures (PVMs) instead of the more general notion of POVMs. First, we will show that POVMs in our context dilate to PVMs.

\begin{proposition}
\label{proposition: dilate POVM to PVM}
Let $\cH$ be a Hilbert space, and let $\{Q_a\}_{a=1}^c$ be a POVM in $\bofh$. Then there is a PVM $\{P_a\}_{a=1}^c$ in $M_{c+1}(\bofh)$ such that, if $E_{11}$ is the first diagonal matrix unit in $M_{c+1}$, then $(E_{11} \otimes I_{\cH})P_a(E_{11} \otimes I_{\cH})=Q_a$ for all $1 \leq a \leq c$.
\end{proposition}

\begin{proof}
We define $V=\begin{pmatrix} Q_1^{\frac{1}{2}} \\ \vdots \\ Q_c^{\frac{1}{2}} \end{pmatrix} \in M_{c,1}(\bofh)$. Then $V$ is an isometry, so
$$U=\begin{pmatrix} V & \sqrt{I-VV^*} \\ 0 & -V^* \end{pmatrix} \in M_{c+1}(\bofh)$$
is a unitary. Define $P_a=U^*(E_{aa} \otimes I_{\cH})U$ for $1 \leq a \leq c-1$, and define $P_c=U^*((E_{cc}+E_{c+1,c+1}) \otimes I_{\cH})U$. Then $\{P_a\}_{a=1}^c$ is a PVM in $M_{c+1}(\bofh)$. Write $U=(U_{k\ell})_{k,\ell=1}^{c+1}$ where each $U_{k\ell} \in \bofh$. The $(1,1)$ entry of $P_a$ is given by
$$(P_a)_{11}=U_{a1}^*U_{a1}=(Q_a^{\frac{1}{2}})(Q_a^{\frac{1}{2}})=Q_a,$$
as desired.
\end{proof}

As a result of Proposition \ref{proposition: dilate POVM to PVM}, we obtain the desired dilation property for POVMs over $M_n(\bofh)$.

\begin{proposition}
\label{proposition: dilate matrix-valued POVM to matrix-valued PVM}
Let $\cH$ be a Hilbert space, and let $q_{a,ij} \in \bofh$ for $1 \leq i,j \leq n$ and $1 \leq a \leq c$ be such that $\{Q_a\}_{a=1}^c$ is a POVM in $M_n(\bofh)$, where $Q_a=(q_{a,ij})$. Let $V:\cH \to \cH^{(c+1)}$ be the isometry sending $\cH$ to the first direct summand of $\cH^{(c+1)}$. Then there are operators $p_{a,ij} \in M_{c+1}(\bofh)$ such that $\{P_a\}_{a=1}^c$ is a PVM in $M_n(M_{c+1}(\bofh))$, where $P_a=(p_{a,ij})$, and $V^*p_{a,ij}V=q_{a,ij}$ for all $1 \leq i,j \leq n$ and $1 \leq a \leq c$.
\end{proposition}

\begin{proof}
We can regard $\{Q_a\}_{a=1}^c$ as a POVM in $M_n(\bofh)$. By Proposition \ref{proposition: dilate POVM to PVM}, there is a PVM $\{S_a\}_{a=1}^c$ in $M_{c+1}(M_n(\bofh))$ such that the $(1,1)$ entry of $S_a$ is $Q_a$. Performing a canonical shuffle $M_{c+1}(M_n(\bofh)) \simeq M_n(M_{c+1}(\bofh))$ \cite[p.~97]{Pbook} on each $S_a$, we obtain operators $p_{a,ij} \in M_{c+1}(\bofh)$ such that the $(1,1)$-entry of $p_{a,ij}$ is $q_{a,ij}$, and $P_a=(p_{a,ij}) \in M_n(M_{c+1}(\bofh))$ are projections with $\sum_{a=1}^c P_a=I$, completing the proof.
\end{proof}

\begin{remark}
In the case of classical inputs and outputs, one would consider $n$ POVMs in $\bofh$ with $c$ outputs each. It is a standard fact that such systems of POVMs can be dilated to a system of $n$ PVMs with $c$ outputs on a larger Hilbert space, which remains finite-dimensional whenever $\cH$ is finite-dimensional.

Alternatively, one can consider $n$ POVMs $\{P_{a,x}\}_{a=1}^c$ for $1 \leq x \leq n$ on $\cH$ as a single POVM on $\bC^n \otimes \cH$ by setting $Q_a=P_{a,1} \oplus \cdots \oplus P_{a,n}$. Then one applies Proposition \ref{proposition: dilate matrix-valued POVM to matrix-valued PVM} to obtain a single PVM in $M_n(\cH \otimes \bC^{c+1})$; however, the projections may no longer be block-diagonal, so they may not induce a family of $n$ PVMs in $\cB(\cH \otimes \bC^{c+1})$. In the case that $n=1$, one can dilate a POVM with $c$ outputs in $\bofh$ to a PVM with $c$ outputs in $\cB(\cH \otimes \bC^c)$, which is more optimal than Proposition \ref{proposition: dilate matrix-valued POVM to matrix-valued PVM}. On the other hand, as soon as $n \geq 2$, the dilation of Proposition \ref{proposition: dilate matrix-valued POVM to matrix-valued PVM} will be more optimal, since the general dilation of $n$ POVMs to $n$ PVMs requires an inductive argument.
\end{remark}

In the following, we let $C^*_{env}(\mathcal S)$ be the C$^\ast$-envelope of an operator system, first shown to exist by M. Hamana \cite{Ham79}.

\begin{proposition}
\label{proposition: c star envelope of block povms}
Let $n,c \in \bN$.
\begin{enumerate}
\item
$C_{env}^*(\mathcal{Q}_{n,c})$ is canonically $*$-isomorphic to the universal $C^*$-algebra $\cP_{n,c}$.
\item
There is a $\ast$-isomorphism $\cP_{n,c} \simeq \mathcal{B}_{n,c}$ given by the map $$\sum_{a=1}^c \omega^a P_a \leftarrow U,$$
where $\omega$ is a primitive $c$-th root of unity.
\end{enumerate}
\end{proposition}

\begin{proof}
We only prove the first claim; the second claim is analogous to the fact that $C^*(\bZ_c) \simeq \ell_{\infty}^c$ (see, for example, \cite{J+}, \cite{Fr12} or \cite{Oz13}). Let $p_{a,ij}$ be the canonical generators of $\cP_{n,c}$, for $1 \leq i,j \leq n$ and $1 \leq a \leq c$. Since $P_a=(p_{a,ij})$ is a projection in $M_n(\cP_{n,c})$, it is positive. Since $\sum_{a=1}^c P_a=I_n$, there is a ucp map $\varphi:\mathcal{Q}_{n,c} \to \cP_{n,c}$ such that $\varphi(q_{a,ij})=p_{a,ij}$. If we represent $\mathcal{Q}_{n,c} \subseteq \bofh$ for some Hilbert space $\cH$, then by Proposition \ref{proposition: dilate matrix-valued POVM to matrix-valued PVM}, there is a unital $*$-homomorphism $\pi:\cP_{n,c} \to M_{c+1}(\bofh)$ such that compressing to the first coordinate yields the map $p_{a,ij} \to q_{a,ij}$. Hence, $\varphi$ is a complete order isomorphism. This shows that $\cP_{n,c}$ is a $C^*$-cover for $\mathcal{Q}_{n,c}$, in the sense that there is a unital complete order embedding of $\mathcal{Q}_{n,c}$ into $\mathcal{P}_{n,c}$, whose range generates $\mathcal{P}_{n,c}$ as a $C^*$-algebra.

By the universal property of the $C^*$-envelope \cite{Ham79}, there is a unique, surjective unital $*$-homomorphism $\rho:\cP_{n,c} \to C_{env}^*(\mathcal{Q}_{n,c})$ such that $\rho(p_{a,ij})=q_{a,ij}$ for all $1 \leq i,j \leq n$ and $1 \leq a \leq c$. As each $P_a$ is a projection in $\cP_{n,c}$, the matrix $Q_a=(q_{a,ij}) \in M_n(C_{env}^*(\mathcal{Q}_{n,c}))$ is a projection as well. We will show that $\rho$ is injective by constructing an inverse. We assume that $\cP_{n,c}$ is faithfully represented as a $C^*$-algebra of operators on a Hilbert space $\cK$. Then the map $\varphi:\mathcal{Q}_{n,c} \to \cP_{n,c}$ above extends to a ucp map $\sigma:C_{env}^*(\mathcal{Q}_{n,c}) \to \cB(\cK)$ by Arveson's extension theorem \cite{Ar69}. We let $\sigma=V^*\beta(\cdot)V$ be a minimal Stinespring representation of $\sigma$, where $V:\cK \to \cL$ is an isometry and $\beta:C_{env}^*(\mathcal{Q}_{n,c}) \to \cB(\cL)$ is a unital $*$-homomorphism. With respect to the decomposition $\cL=\cK \oplus \cK^{\perp}$, one has
$$\beta(q_{a,ij})=\begin{pmatrix} \varphi(q_{a,ij}) & * \\ * & * \end{pmatrix}=\begin{pmatrix} p_{a,ij} & * \\ * & * \end{pmatrix}.$$
Thus, after a shuffle, one may write $\beta^{(n)}(Q_a)=(\beta(q_{a,ij}))$ as
$$\begin{pmatrix} \varphi^{(n)}(Q_a) & * \\ * & * \end{pmatrix}=\begin{pmatrix} P_a & * \\ * & * \end{pmatrix}.$$
As $Q_a$ is a projection in $M_n(C_{env}^*(\mathcal{Q}_{n,c}))$, so is $\beta^{(n)}(Q_a)$ in $M_n(\cB(\cL))$. But $P_a$ is a projection as well, so the off-diagonal blocks must be $0$. Therefore, reversing the shuffle yields
$$\beta(q_{a,ij})=\begin{pmatrix} p_{a,ij} & 0 \\ 0 & * \end{pmatrix}.$$
Considering $\beta(q_{a,ij}^*q_{a,ij})$ and $\beta(q_{a,ij}q_{a,ij}^*)$, it follows that the multiplicative domain of $\sigma$ contains $q_{a,ij}$ for each $1 \leq i,j \leq n$ and $1 \leq a \leq c$; as these elements generate $C_{env}^*(\mathcal{Q}_{n,c})$, $\sigma$ must be a $*$-homomorphism. Since $\rho$ and $\sigma$ are mutual inverses on the generators, they must be mutual inverses on the whole algebras. Hence, $\rho$ is injective, so that $C_{env}^*(\mathcal{Q}_{n,c}) \simeq \cP_{n,c}$.
\end{proof}

Combining part (2) of Proposition \ref{proposition: c star envelope of block povms} and Proposition \ref{proposition: dilate matrix-valued POVM to matrix-valued PVM}, one can obtain a similar dilation corresponding to a block unitary of order $c$. Indeed, if $T=(T_{ij}) \in M_n(\bofh)$ is a contraction that can be written as $T=\sum_{a=1}^c \omega^a Q_a$, where $\omega$ is a primitive $c$-th root of unity and $\{Q_a\}_{a=1}^c$ is a POVM in $M_n(\bofh)$, then one can dilate $T$ to a unitary $U=(U_{ij}) \in M_n(M_{c+1}(\bofh))$ of order $c$, such that the $(1,1)$ block of each $U_{ij}$ is $T_{ij}$. It is sometimes convenient to use this form of the dilation, rather than the dilation of the POVM to a PVM.

We now study some of the structure of $\cP_{n,c}$. First, we show that $\cP_{n,c}$ has the lifting property. Recall that a $C^*$-algebra $\cA$ has the \textbf{lifting property} if, whenever $\cB$ is a $C^*$-algebra, $\cJ$ is an ideal in $\cB$, and $\varphi:\cA \to \cB/\cJ$ is a contractive completely positive map, then there exists a contractive completely positive lift $\widetilde{\varphi}:\cA \to \cB$ of $\varphi$. As noted in \cite[Lemma~13.1.2]{BO08}, when $\cA$ is unital, one need only deal with the case when $\cB$ is unital and $\varphi,\widetilde{\varphi}$ are unital.

On the way to proving that $\cP_{n,c}$ has the lifting property, we will need the following fact. We include a proof for convenience.

\begin{proposition}
\label{proposition: POVMs lift to POVMs}
Let $\cB$ be a unital $C^*$-algebra, $\cJ$ be an ideal in $\cB$, and $p_1,...,p_c \in \cB/\cJ$ be projections with $\sum_{a=1}^c p_a=1_{\cB/\cJ}$. Let $q:\cB \to \cJ$ be the canonical quotient map. Then there are positive elements $\widetilde{p}_1,...,\widetilde{p}_c$ in $\cB$ such that $\sum_{a=1}^c \widetilde{p}_a=1_{\cB}$ and $q(\widetilde{p}_a)=p_a$ for all $1 \leq a \leq c$.
\end{proposition}
 
\begin{proof}
The assumption implies that there is a unital $*$-homomorphism $\pi:\ell^{\infty}_c \to \cB/\cJ$ such that $\pi(e_a)=p_a$ for each $a$. As $\ell^{\infty}_c$ is separable and nuclear, the Choi-Effros lifting theorem \cite{CE76} gives a ucp lift $\varphi:\ell^{\infty}_c \to \cB$ of $\pi$. Defining $\widetilde{p}_a=\varphi(e_a)$ concludes the proof.
\end{proof}
 
\begin{theorem}
\label{theorem: lifting property}
$\cP_{n,c}$ has the lifting property.
\end{theorem}

\begin{proof}
This proof is similar in nature to results from \cite{J+,BO08}. First, suppose that $\cB$ is a unital $C^*$-algebra, $\cJ$ is an ideal in $\cB$ and $\pi:\cP_{n,c} \to \cB/\cJ$ is a $*$-homomorphism. Then $\pi^{(n)}=\id_n \otimes \pi:M_n(\cP_{n,c}) \to M_n(\cB)/M_n(\cJ)$ is a $*$-homomorphism. Define $Q_a=\pi^{(n)}(P_a)$. By Proposition \ref{proposition: POVMs lift to POVMs}, one can find a POVM $\widetilde{Q}_1,...,\widetilde{Q}_c$ in $M_n(\cB)$ that is a lift of $Q_1,...,Q_c$. Next, we apply Proposition \ref{proposition: dilate matrix-valued POVM to matrix-valued PVM} and compress to the $(1,1)$ corner to obtain a ucp map $\gamma:\cP_{n,c} \to \cB$ given by $\gamma(p_{a,ij})=\widetilde{Q}_{a,ij}$ for all $a,i,j$. As $\gamma$ is a lift of $\pi$ on the generators, a multiplicative domain argument establishes that $\gamma$ is a lift of $\pi$.

Now we deal with the general case. Let $\varphi:\cP_{n,c} \to \cB/\cJ$ be a ucp map. Since $\cP_{n,c}$ is separable, one can restrict if necessary and assume without loss of generality that $\cB$ is separable. Then we apply Kasparov's dilation theorem \cite{Ka80} to $\varphi$: letting $\cK=\cK(\ell_2)$ denote the compact operators and $M(\cA)$ denote the multiplier algebra of a (separable) $C^*$-algebra $\cA$, there is a $*$-homomorphism $\rho:\cP_{n,c} \to M(\cK \otimes_{\min} (\cB/\cJ))$ satisfying $\rho(x)_{11}=\varphi(x)$ for all $x \in \cP_{n,c}$. (Here, $\rho(x)_{11}$ refers to the $(1,1)$ entry of $\rho(x)$.) If $q:\cB \to \cB/\cJ$ is the canonical quotient map, then $\id \otimes q:\cK \otimes_{\min} \cB \to \cK \otimes_{\min} (\cB/\cJ)$ extends to a surjective $*$-homomorphism $\sigma:M(\cK \otimes_{\min} \cB) \to M(\cK \otimes_{\min} (\cB/\cJ))$ by the non-commutative Tietze extension theorem \cite[Proposition~6.8]{La95}. Therefore, we can lift the $*$-homomorphism $\rho$ to a ucp map $\eta:\cP_{n,c} \to M(\cK \otimes_{\min} \cB)$. Defining $\widetilde{\varphi}(x)=\eta(x)_{11}$, we obtain a lift of $\varphi$.
\end{proof}

Next, we establish residual finite-dimensionality of $\cP_{n,c}$.  Recall that a $C^\ast$-algebra $\mathcal A$ is called {\bf residually finite-dimensional (RFD)} if, for any $x \in \cA \backslash \{0\}$, there exists  $k \in\mathbb N$ and a finite-dimensional representation $\pi: \cA \to M_k$ with $\pi(x) \ne 0$.   

\begin{theorem}
$\cP_{n,c}$ is RFD.
\end{theorem}

\begin{proof}
This proof is very similar to the proofs that $C^*(\mathbb{F}_2)$ and $\cB_n$ are RFD, respectively \cite{Ch80,Ha19}. As $\cP_{n,c}$ is separable, we may represent it faithfully as a subalgebra of $\bofh$, where $\cH$ is separable and infinite-dimensional. Let $(E_m)_{m=1}^{\infty}$ be an increasing sequence of projections in $\bofh$ such that $\text{rank}(E_m)=m$ and $\text{SOT}$-$\lim_{m \to \infty} E_m=I_{\cH}$. For each $1 \leq a \leq c$ and $1 \leq i,j \leq n$, we let $p_{a,ij}^{(m)}=E_m p_{a,ij}E_m$. Then the matrices $P_a^{(m)}=(p^{(m)}_{a,ij}) \in M_n(\cB(E_m \cH))$ define a POVM with $c$ outputs in $\cB(E_m \cH)$. Applying Proposition \ref{proposition: dilate matrix-valued POVM to matrix-valued PVM}, we obtain a unital $*$-homomorphism $\rho_m:\cP_{n,c} \to M_{c+1}(\cB(E_m \cH))$ which, after a shuffle of the form $M_n(M_{c+1}(\cB(E_m\cH))) \to M_{c+1}(M_n(\cB(E_m \cH)))$, can be written as
$$P_a \mapsto \begin{cases} U_m^*(E_{aa} \otimes I_{E_m\cH})U_m & 1 \leq a \leq c-1 \\ U_m^*((E_{cc}+E_{c+1,c+1}) \otimes I_{E_m\cH})U_m & a=c \end{cases},$$
where
$$U_m=\begin{pmatrix} \begin{pmatrix} (P_1^{(m)})^{\frac{1}{2}} \\ \vdots \\ (P_c^{(m)})^{\frac{1}{2}} \end{pmatrix} & \left( (\delta_{ab} I_{E_m \cH} - (P_a^{(m)})^{\frac{1}{2}}(P_b^{(m)})^{\frac{1}{2}})_{a,b=1}^c \right)^{\frac{1}{2}} \\ 0 & -\begin{pmatrix} (P_1^{(m)})^{\frac{1}{2}} & \cdots & (P_c^{(m)})^{\frac{1}{2}} \end{pmatrix}\end{pmatrix} \in M_{c+1}(M_n(\cB(E_m\cH))).$$ 
The key point is that, considering $\rho_m(p_{a,ij}) \in M_{c+1}(\cB(E_m \cH))$, each block from $\cB(E_m \cH)$ belongs to the $C^*$-subalgebra of $\cB(E_m \cH)$ generated by the set $\{ p^{(m)}_{a,ij}: 1 \leq a \leq c, \, 1 \leq i,j \leq n\}$. This set of blocks is closed under the adjoint since $(E_m p_{a,ij} E_m)^*=E_m p_{a,ij}^* E_m=E_m p_{a,ji} E_m$. Since $\text{SOT}$-$\lim_{m \to \infty} E_m=I_{\cH}$, we have $\text{SOT}^*$-$\lim_{m \to \infty} p_{a,ij}^{(m)}=p_{a,ij}$. By joint continuity of multiplication in the unit ball with respect to $\text{SOT}^*$, it follows that $\text{SOT}^*$-$\lim_{m \to \infty} P_a^{(m)}=\text{SOT}^*$-$\lim_{m \to \infty} (P_a^{(m)})^{\frac{1}{2}}=P_a$ for each $a$. One can check that
$$\text{SOT}^*\text{-}\lim_{m \to \infty} U_m=\begin{pmatrix} \begin{pmatrix} P_1 \\ \vdots \\ P_c \end{pmatrix} & \begin{pmatrix} I_{\cH}-P_1 \\ & \ddots \\ & & I_{\cH}-P_c \end{pmatrix} \\ 0 & -(\begin{pmatrix} P_1 & \cdots & P_c \end{pmatrix}) \end{pmatrix}.$$
Applying a shuffle, we see that, for each $a,i,j$, $\text{SOT}^*$-$\lim_{m \to \infty} \rho_m(p_{a,ij})$ exists in $M_{c+1}(\bofh)$; moreover, its $(1,1)$-entry is exactly $p_{a,ij}$. As $P_a=(p_{a,ij})$ is a projection, another shuffle argument shows that
$$\text{SOT}^*\text{-}\lim_{m \to \infty} \rho_m(p_{a,ij})=\begin{pmatrix} p_{a,ij} & 0 \\ 0 & * \end{pmatrix} \in M_{c+1}(\bofh).$$
Therefore, if $W$ is a linear combination of finite words in the generators of $\cP_{n,c}$, by considering the $(1,1)$ entry of $\rho_m(W)$, it follows that $\text{SOT}^*$-$\lim_{m \to \infty} \rho_m(W)=\begin{pmatrix} W & 0 \\ 0 & * \end{pmatrix}$. By passing to a subsequence if necessary, this forces $\lim_{m \to \infty} \|\rho_m(W)\|_{M_{c+1}(\cB(E_m\cH))}=\|W\|_{\bofh}$. Hence, $\bigoplus_{m=1}^{\infty} \rho_m:\cP_{n,c} \to \bigoplus_{m=1}^{\infty} M_{c+1}(\cB(E_m \cH))$ is a $*$-homomorphism that is isometric on the dense $*$-subalgebra spanned by finite words in the generators $p_{a,ij}$. It follows that $\bigoplus_{m=1}^{\infty} \rho_m$ is an isometry on the whole algebra. As each $E_m \cH$ is finite-dimensional, we conclude that $\cP_{n,c}$ is RFD.
\end{proof}

A standard fact is that minimal tensor products of RFD $C^*$-algebras remain RFD. Hence, $\cP_{n,c} \otimes_{\min} \cP_{n,c}$ is RFD. We can use this fact to relate $\mathcal{Q}_{qa}(n,c)$ to states on the minimal tensor product. First, we need the fact that quantum commuting correlations with a finite-dimensional entanglement space must belong to $\mathcal{Q}_q(n,c)$.

\begin{lemma}
\label{lemma: finite-dimensional realization of commuting}
Suppose that $X=(X_{(i,j),(k,\ell)}^{(a,b)}) \in \mathcal{Q}_{qc}(n,c)$ can be written as $X=(\la P_{a,ij}Q_{b,k\ell} \chi,\chi \ra)$, where $P_a=(P_{a,ij})$ and $Q_b=(Q_{b,k\ell})$ are positive in $M_n(\bofh)$, $\sum_{a=1}^c P_a=\sum_{a=1}^c Q_a=I_n$, $[P_{a,ij},Q_{b,k\ell}]=0$ for all $i,j,k,\ell,a,b$ and $\chi \in \cH$ is a unit vector. If $\cH$ is finite-dimensional, then $X \in \mathcal{Q}_q(n,c)$.
\end{lemma}

\begin{proof}
Let $\cA$ be the $C^*$-algebra generated by the set $\{P_{a,ij}: 1 \leq a \leq c, \, 1 \leq i,j \leq n\}$ and let $\cB$ be the $C^*$-algebra generated by the set $\{Q_{b,k\ell}: 1 \leq b \leq c, \, 1 \leq k,\ell \leq n\}$. Then $\cA$ and $\cB$ are unital $C^*$-subalgebras of $\bofh$, and every element of $\cA$ commutes with every element of $\cB$. By a theorem of Tsirelson \cite{Ts06}, there are finite-dimensional Hilbert spaces $\cH_A$ and $\cH_B$, an isometry $V:\cH \to \cH_A \otimes \cH_B$, and unital $*$-homomorphisms $\pi:\cA \to \cB(\cH_A)$ and $\rho:\cB \to \cB(\cH_B)$ such that $V^*(\pi(P_{a,ij}) \otimes \rho(Q_{b,k\ell}))V=P_{a,ij}Q_{b,k\ell}$ for all $a,b,i,j,k,\ell$. Defining the unit vector $\xi=V\chi \in \cH_A \otimes \cH_B$, we see that
\[ X_{(i,j),(k,\ell)}^{(a,b)}=\la (\pi(P_{a,ij}) \otimes \rho(Q_{b,k\ell}))\xi,\xi \ra. \]
Therefore, $X \in \mathcal{Q}_q(n,c)$.
\end{proof}

Now, we can prove the disambiguation theorems for $\mathcal{Q}_t(n,c)$. We note that, by Proposition \ref{proposition: dilate matrix-valued POVM to matrix-valued PVM}, any element of $\mathcal{Q}_q(n,c)$ can be represented using a finite-dimensional tensor product framework $\cH_A \otimes \cH_B$ and PVMs $\{P_a\}_{a=1}^c$ on $\cH_A$ and $\{Q_b\}_{b=1}^c$ on $\cH_B$, respectively. This fact holds because, given a POVM $\{Q_b\}_{b=1}^c$ in $\bofh$, the dilation in Proposition \ref{proposition: dilate matrix-valued POVM to matrix-valued PVM} is in $M_{c+1}(\bofh) \simeq \cB(\cH^{(c+1)})$; in particular, the Hilbert space remains finite-dimensional if $\cH$ is finite-dimensional. Similarly, it is easy to see that all elements of $\mathcal{Q}_{qs}(n,c)$ can be represented using PVMs.

Next, we show that every element $\mathcal{Q}_{qa}(n,c)$ can be represented by PVMs, which arise from the minimal tensor product of $\cP_{n,c}$.

\begin{theorem}
\label{theorem: characterizing qa}
Let $X=(X_{(i,j),(k,\ell)}^{(a,b)}) \in (M_n \otimes M_n)^{c^2}$. The following are equivalent:
\begin{enumerate}
\item
$X$ belongs to $\mathcal{Q}_{qa}(n,c)$.
\item
There is a state $s:\cP_{n,c} \otimes_{\min} \cP_{n,c} \to \bC$ satisfying $s(p_{a,ij} \otimes p_{b,k\ell})=X^{(a,b)}_{(i,j),(k,\ell)}$ for all $1 \leq a,b \leq c$ and $1 \leq i,j,k,\ell \leq n$.
\item
There is a state $s:\mathcal{Q}_{n,c} \otimes_{\min} \mathcal{Q}_{n,c} \to \bC$ satisfying $s(q_{a,ij} \otimes q_{b,k\ell})=X^{(a,b)}_{(i,j),(k,\ell)}$ for all $1 \leq a,b \leq c$ and $1 \leq i,j,k,\ell \leq n$.
\end{enumerate}
\end{theorem}

\begin{proof}
We recall that the minimal tensor product of operator spaces (in particular, operator systems) is injective (see, for example, \cite{KPTT11}). Since $\mathcal{Q}_{n,c} \subseteq \mathcal{P}_{n,c}$ via the mapping $q_{a,ij} \mapsto p_{a,ij}$, injectivity of the minimal tensor product shows that $\mathcal{Q}_{n,c} \otimes_{\min} \mathcal{Q}_{n,c} \subseteq \mathcal{P}_{n,c} \otimes_{\min} \mathcal{P}_{n,c}$ completely order isomorphically. Using the Hahn-Banach theorem, it then follows that (2) and (3) are equivalent.

If (1) holds, then $X$ is in $\mathcal{Q}_{qa}(n,c)$, so it is a pointwise limit of elements of $\mathcal{Q}_q(n,c)$. Since elements of $\mathcal{Q}_q(n,c)$ can be represented by PVMs, $X$ is a limit of elements which correspond to finite-dimensional tensor product representations of $\mathcal{P}_{n,c} \otimes_{\min} \mathcal{P}_{n,c}$, which are automatically continuous. Hence, (1) implies (2). Lastly, suppose that (2) holds. Since $\mathcal{P}_{n,c} \otimes_{\min} \mathcal{P}_{n,c}$ is RFD, a theorem of R. Exel and T.A. Loring \cite{EL92} shows that $s$ is a $w^*$-limit of states $s_{\lambda}$ on $\mathcal{P}_{n,c} \otimes_{\min} \mathcal{P}_{n,c}$ whose GNS representations are finite-dimensional. Applying Lemma \ref{lemma: finite-dimensional realization of commuting}, each $s_{\lambda}$ applied to the generators $p_{a,ij} \otimes p_{b,k\ell}$ of $\mathcal{P}_{n,c} \otimes_{\min} \mathcal{P}_{n,c}$ yields an element $X_{\lambda}$ of $\mathcal{Q}_q(n,c)$; moreover, $\lim_{\lambda} X_{\lambda}=X$ pointwise. This shows that $X \in \overline{\mathcal{Q}_q(n,c)}=\mathcal{Q}_{qa}(n,c)$, which shows that (2) implies (1).
\end{proof}

To establish the same disambiguation theorem for $qc$-correlations, we will show that the {\it commuting} tensor product $\mathcal{Q}_{n,c} \otimes_c \mathcal{Q}_{n,c}$ is completely order isomorphic to the copy of $\mathcal{Q}_{n,c} \otimes \mathcal{Q}_{n,c}$ inside of $\mathcal{P}_{n,c} \otimes_{\max} \mathcal{P}_{n,c}$. We recall that, if $\cS$ and $\cT$ are operator systems, then an element $Y$ in $M_n(\cS \otimes_c \cT)$ is defined as positive in the commuting tensor product provided that $Y=Y^*$ and, whenever $\varphi:\cS \to \bofh$ and $\psi:\cT \to \bofh$ are ucp maps with commuting ranges, then $(\varphi \cdot \psi)^{(n)}(Y)$ is positive in $M_n(\bofh)$, where $\varphi \cdot \psi:\cS \otimes \cT \to \bofh$ is the linear map defined by $(\varphi \cdot \psi)(x \otimes y)=\varphi(x)\psi(y)$ for all $x \in \cS$ and $y \in \cT$. 

The next lemma is an adaptation of \cite[Proposition~4.6]{Ha19}.

\begin{lemma}
\label{lemma: c to max complete order embedding}
Let $\cS$ be an operator system. Then the canonical map $\mathcal{Q}_{n,c} \otimes_c \cS \to \mathcal{P}_{n,c} \otimes_{\max} \cS$ is a complete order embedding.
\end{lemma}

\begin{proof}
Since $\mathcal{P}_{n,c}$ is a unital $C^*$-algebra, we have $\mathcal{P}_{n,c} \otimes_c \cS=\mathcal{P}_{n,c} \otimes_{\max} \cS$ \cite[Theorem~6.7]{KPTT11}. The canonical map $\mathcal{Q}_{n,c} \otimes_c \cS \to \mathcal{P}_{n,c} \otimes_c \cS$ is a tensor product of canonical inclusion maps, which are ucp. By functoriality of the commuting tensor product \cite{KPTT11}, the inclusion $\mathcal{Q}_{n,c} \otimes_c \cS \to \mathcal{P}_{n,c} \otimes_c \cS$ is ucp. Hence, it suffices to show that this map is a complete order embedding.

To this end, suppose that $Y=Y^* \in M_m(\mathcal{Q}_{n,c} \otimes \cS)$ is a positive element of $M_m(\mathcal{P}_{n,c} \otimes_c \cS)$. Let $\varphi:\mathcal{Q}_{n,c} \to \bofh$ and $\psi:\cS \to \bofh$ be ucp maps with commuting ranges; we will show that $(\varphi \cdot \psi)^{(m)}(Y)$ is positive in $M_m(\bofh)$. For convenience, we define $Q_{a,ij}=\varphi(q_{a,ij})$. By Proposition \ref{proposition: dilate matrix-valued POVM to matrix-valued PVM}, there is a unital $*$-homomorphism $\pi:\mathcal{P}_{n,c} \to M_{c+1}(\bofh)$ such that the $(1,1)$ corner of $\pi(p_{a,ij})$ is $Q_{a,ij}$ for all $1 \leq a \leq c$ and $1 \leq i,j \leq n$. Moreover, for each $x \in \mathcal{P}_{n,c}$, each block of $\pi(x)$ in $\bofh$ belongs to the $C^*$-algebra generated by the set $\{ Q_{a,ij}: 1 \leq a \leq c, \, 1 \leq i,j \leq n \}$. We extend $\varphi$ to a ucp map on $\mathcal{P}_{n,c}$ by defining $\varphi(\cdot)=(\pi(\cdot))_{11}$. Define $\widetilde{\psi}: \cS \to M_{c+1}(\bofh)$ by $\widetilde{\psi}(s)=I_{c+1} \otimes \psi(s)$. Since $\psi(s)$ commutes with the range of $\varphi$, $\psi(s)$ must commute with the $C^*$-algebra generated by the range of $\varphi$. Hence, $\psi(s)$ commutes with every block of $\pi(p_{a,ij})$, for all $a,i,j$. By the multiplicativity of $\pi$, $\psi(s)$ commutes with the range of $\pi$. By definition of the commuting tensor product, this means that $\pi \cdot \widetilde{\psi}: \mathcal{P}_{n,c} \otimes_c \cS \to M_{c+1}(\bofh)$ is ucp; moreover, the $(1,1)$ block of $\pi \cdot \widetilde{\psi}$ is $\varphi \cdot \psi$. This means that $\varphi \cdot \psi$ is ucp on $\mathcal{P}_{n,c} \otimes_c \cS$. Restricting to the copy of the algebraic tensor product $\mathcal{Q}_{n,c} \otimes \cS$, it follows that $(\varphi \cdot \psi)^{(m)}(Y)$ is positive, making the canonical map $\mathcal{Q}_{n,c} \otimes_c \cS \to \mathcal{P}_{n,c} \otimes_c \cS$ a complete order embedding.
\end{proof}

\begin{theorem}
\label{theorem: characterizing qc}
Let $X=(X^{(a,b)}_{(i,j),(k,\ell)}) \in (M_n \otimes M_n)^{c^2}$. The following are equivalent.
\begin{enumerate}
\item
$X$ belongs to $\mathcal{Q}_{qc}(n,c)$.
\item
There is a state $s:\cP_{n,c} \otimes_{\max} \cP_{n,c} \to \bC$ satisfying $s(p_{a,ij} \otimes p_{b,k\ell})=X^{(a,b)}_{(i,j),(k,\ell)}$ for all $1 \leq a,b \leq c$ and $1 \leq i,j,k,\ell \leq n$.
\item
There is a state $s:\mathcal{Q}_{n,c} \otimes_c \mathcal{Q}_{n,c} \to \bC$ satisfying $s(q_{a,ij} \otimes q_{b,k\ell})=X^{(a,b)}_{(i,j),(k,\ell)}$ for all $1 \leq a,b \leq c$ and $1 \leq i,j,k,\ell \leq n$.
\end{enumerate}
\end{theorem}

\begin{proof}
Since $\mathcal{Q}_{qc}(n,c)$ is defined in terms of POVMs where Alice's entries commute with Bob's, we see that (1) is equivalent to (3). Based on two applications of Lemma \ref{lemma: c to max complete order embedding}, we see that $\mathcal{Q}_{n,c} \otimes_c \mathcal{Q}_{n,c}$ is completely order isomorphic to the image of $\mathcal{Q}_{n,c} \otimes \mathcal{Q}_{n,c}$ in $\mathcal{P}_{n,c} \otimes_{\max} \mathcal{P}_{n,c}$. Hence, (2) and (3) are equivalent.
\end{proof}

When considering the quantum-to-classical graph homomorphism game, the local model will be of interest because of its link to the usual notion of a (classical) homomorphism from a quantum graph to a classical graph.  It is helpful to note  that all strategies in $\mathcal{Q}_{loc}(n,c)$ can be obtained using PVMs instead of just POVMs. We use a bit of a different direction for proving this fact. First, we show the following simple fact:

\begin{proposition}
\label{proposition: loc is closed}
$\mathcal{Q}_{loc}(n,c)$ is a closed set.
\end{proposition}

\begin{proof}
Let $X_m=(X_{m,(i,j),(k,\ell)}^{(a,b)}) \in \mathcal{Q}_{loc}(n,c)$ be a sequence of correlations such that $\displaystyle \lim_{m \to \infty} X_m=X$ pointwise in $(M_n \otimes M_n)^{c^2}$. For each $m$, there is a unital commutative $C^*$-algebra $\cA_m$, POVMs $P_1^{(m)},...,P_c^{(m)}$ and $Q_1^{(m)},...,Q_c^{(m)}$ in $M_n(\cA_m)$, and a state $s_m$ on $\cA_m$ such that
$$X_{m,(i,j),(k,\ell)}^{(a,b)}=s_m(P_{a,ij}^{(m)}Q_{b,k\ell}^{(m)}) \, \forall a,b,i,j,k,\ell.$$
Define $\cA=\bigoplus_{m=1}^{\infty} \cA_m$, $P_{a,ij}=\bigoplus_{m=1}^{\infty} P_{a,ij}^{(m)}$ and $Q_{b,k\ell}=\bigoplus_{m=1}^{\infty} Q_{b,k\ell}^{(m)}$. Then $P_a=(P_{a,ij})$ and $Q_b=(Q_{b,k\ell})$ define two POVMs in $M_n(\cA)$ with $c$ outputs. Using the canonical compression from $\cA$ onto $\cA_m$, we can extend $s_m$ to a state $\widetilde{s}_m$ on $\cA$. As the state space of $\cA$ is $w^*$-compact, we choose a $w^*$-limit point $s$ of the sequence $(\widetilde{s}_m)_{m=1}^{\infty}$. Then $X=(X^{(a,b)}_{(i,j),(k,\ell)})=(s(P_{a,ij}Q_{b,k\ell}))$, which shows that $X \in \mathcal{Q}_{loc}(n,c)$. 
\end{proof}

We note that the above proof works just as well for projection-valued measures. A standard argument shows that limits of convex combinations of elements of $\mathcal{Q}_{loc}(n,c)$ represented by PVMs from abelian algebras can still be represented by PVMs from abelian algebras. With this fact in hand, we can prove the disambiguation theorem for $\mathcal{Q}_{loc}(n,c)$.

\begin{theorem}
\label{theorem: characterizing loc}
Let $X=(X^{(a,b)}_{(i,j),(k,\ell)}) \in (M_n \otimes M_n)^{c^2}$. The following are equivalent:
\begin{enumerate}
\item
$X$ belongs to $\mathcal{Q}_{loc}(n,c)$;
\item
There is a commutative $C^*$-algebra $\cA$, a state $s$ on $\cA$ and POVMs $\{P_1,...,P_c\}$, $\{Q_1,...,Q_c\} \subseteq M_n(\cA)$ such that
$$X^{(a,b)}_{(i,j),(k,\ell)}=s(P_{a,ij}Q_{b,k\ell});$$
\item
There is a commutative $C^*$-algebra $\cA$, a state $s$ on $\cA$, and PVMs $\{P_1,...,P_c\},\{Q_1,...,Q_c\} \subseteq M_n(\cA)$ such that
$$X^{(a,b)}_{(i,j),(k,\ell)}=s(P_{a,ij}Q_{b,k\ell}).$$
\end{enumerate}
\end{theorem}

\begin{proof}
Clearly (1) and (2) are equivalent by the definition of $\mathcal{Q}_{loc}(n,c)$. Since every PVM is a POVM, (3) implies (2). Hence, we need only show that (2) implies (3). Suppose that
$$X_{(i,j),(k,\ell)}^{(a,b)}=s(P_{a,ij}Q_{b,k\ell})$$
for a state $s$ on a commutative $C^*$-algebra $\cA$ and a POVMs $P_1,...,P_c$ and $Q_1,...,Q_c$ in $M_n(\cA)$. Then $\cA \simeq C(Y)$ for a compact Hausdorff space $Y$. The extreme points of the state space of $Y$ are simply evaluation functionals $\delta_y$ for $y \in Y$, which are multiplicative. Hence, $\delta_y^{(n)}(Q_a) \in M_n(\bC)$ defines a POVM with $c$ outputs in $M_n(\bC)$, where $\delta_y^{(n)}=\id_n \otimes \delta_y$. Recall that the extreme points of the set of positive contractions in a von Neumann algebra are precisely the projections in the von Neumann algebra. An easy application of this argument shows that the extreme points of the set of POVMs with $c$ outputs in a von Neumann algebra are precisely the PVMs with $c$ outputs. Hence, $\{\delta_y^{(n)}(Q_1),...,\delta_y^{(n)}(Q_c)\}$ lies in the closed convex hull of the set of PVMs in $M_n(\bC)$ with $c$ outputs.  Applying a similar argument to $\{\delta_y^{(n)}(P_1),...,\delta_y^{(n)}(P_c)\}$, it follows that the correlation $(\delta_y(P_{a,ij}Q_{b,k\ell}))$ is a convex combination of elements of $\mathcal{Q}_{loc}(n,c)$ obtained by tensoring projections from $M_n(\bC)$. Taking the closed convex hull, we obtain the original correlation $X$. In this way, we can write $X$ using projection-valued measures, which shows that (2) implies (3).
\end{proof}

For $t \in \{loc,q,qs,qa,qc\}$, we let $C_t(n,c)$ denote the set of correlations with classical inputs and classical outputs in the $t$-model, where Alice and Bob now possess $n$ PVMs (equivalently, POVMs) with $c$ outputs each. These sets embed into $\mathcal{Q}_t(n,c)$ in a natural way.

\begin{proposition}
\label{proposition: classical inputs inside of quantum inputs}
Let $t \in \{loc,q,qs,qa,qc\}$. Then $C_t(n,c)$ is affinely isomorphic to
$$\{ X \in \mathcal{Q}_t(n,c): X^{(a,b)}_{(i,j),(k,\ell)}=0 \text{ if } i \neq j \text{ or } k \neq \ell\} \subseteq \mathcal{Q}_t(n,c).$$
Moreover, the compression map
$$X \mapsto (\delta_{ij}\delta_{k\ell} X^{(a,b)}_{(i,j),(k,\ell)}): \mathcal{Q}_t(n,c) \to C_t(n,c)$$
is a continuous affine map.
\end{proposition}

\begin{proof}
All of the claims follow from the following observations: if $\{E_{a,x}\}$ is a collection of positive operators such that $\{E_{a,x}\}_{a=1}^c$ is a POVM in $\bofh$ for each $1 \leq x \leq n$, then the operators $P_a:=\bigoplus_{x=1}^n E_{a,x}$ define a POVM in $M_n(\bofh)$. Similarly, if $\{Q_a\}_{a=1}^c$ is a POVM in $M_n(\bofh)$, then setting $F_{a,x}=Q_{a,xx} \in \bofh$, we see that $\{F_{a,x}\}_{a=1}^c$ is a POVM in $\bofh$ for each $1 \leq x \leq n$. We leave the verification of the claims above to the reader.
\end{proof}

Using what we have established so far, we see that the sets $\mathcal{Q}_t(n,c)$ satisfy
$$\mathcal{Q}_{loc}(n,c) \subseteq \mathcal{Q}_q(n,c) \subseteq \mathcal{Q}_{qs}(n,c) \subseteq \mathcal{Q}_{qa}(n,c) \subseteq \mathcal{Q}_{qc}(n,c).$$
The sets $\mathcal{Q}_{loc}(n,c)$, $\mathcal{Q}_{qa}(n,c)$ and $\mathcal{Q}_{qc}(n,c)$ are all closed, and $\mathcal{Q}_{qa}(n,c)=\overline{\mathcal{Q}_{qs}(n,c)}=\overline{\mathcal{Q}_q(n,c)}$. Using the previous proposition, all of the containments above are strict in general. Indeed, $\mathcal{Q}_{loc}(2,2) \neq \mathcal{Q}_q(2,2)$ by the CHSH game \cite[Chapter 3]{wilde}. By a theorem of A. Coladangelo and J. Stark \cite{CS18}, $\mathcal{Q}_q(5,3) \neq \mathcal{Q}_{qs}(5,3)$. A theorem of K. Dykema, V.I. Paulsen and J. Prakash \cite{DPP19} shows $\mathcal{Q}_{qs}(5,2) \neq \mathcal{Q}_{qa}(5,2)$. In fact, using the $T_2$ quantum XOR game and a result of R. Cleve, L. Liu and V.I. Paulsen \cite{CLP17}, one can show that $\mathcal{Q}_{qs}(3,2) \neq \mathcal{Q}_{qa}(3,2)$. (The analogous problem for $C_t(3,2)$, perhaps surprisingly, is still open, although it has been shown that the \textit{synchronous} versions are equal; in fact, $C_q^s(3,2)=C_{qc}^s(3,2)$ \cite{Ru19}.) Lastly, due to the negative resolution to Connes' embedding problem \cite{J+20}, it follows that $\mathcal{Q}_{qa}(n,c) \neq \mathcal{Q}_{qc}(n,c)$ for some (likely very large) values of $n$ and $c$.

We close this section with the following isomorphism between $\cP_{n,c}$ and its opposite algebra. This isomorphism will be used in our discussion of synchronous correlations in the next few sections. Recall that, if $\cA$ is a $C^*$-algebra, then $\cA^{op}$ is the $C^*$-algebra with the same norm structure as $\cA$, but with multiplication given by $a^{op}b^{op}=(ba)^{op}$.

\begin{lemma}
\label{lemma: opposite algebra}
The map $p_{a,ij} \mapsto p_{a,ji}^{op}$ extends to a unital $*$-isomorphism $\pi:\cP_{n,c} \to \cP_{n,c}^{op}$.
\end{lemma}

\begin{proof}
In $\cP_{n,c}^{op}$, one has
\begin{align*}
\sum_{k=1}^n p_{a,kj}^{op}p_{a,ik}^{op}&=\sum_{k=1}^n (p_{a,ik}p_{a,kj})^{op} \\
&=\left( \sum_{k=1}^n p_{a,ik}p_{a,kj} \right)^{op} \\
&=p_{a,ij}^{op},
\end{align*}
where we have used the fact that $P_a=(p_{a,ij})$ is a projection in $\cP_{n,c}$. Evidently $(p_{a,ij}^{op})^*=(p_{a,ij}^*)^{op}=p_{a,ji}^{op}$, so the above calculations show that $P_a^{op}:=(p_{a,ji}^{op})_{i,j=1}^n$ is a projection in $M_n(\cP_{n,c}^{op})$. Moreover, $\sum_{a=1}^c P_a^{op}=I_n$. By the universal property of $\cP_{n,c}$, there is a unital $*$-homomorphism $\pi:\cP_{n,c} \to \cP_{n,c}^{op}$ such that $\pi(p_{a,ij})=p_{a,ji}^{op}$.

One can show that $\cP_{n,c}^{op}$ is the universal $C^*$-algebra generated by entries $p_{a,ij}^{op}$ with the property that $P_a^{op}=(p_{a,ji}^{op})_{i,j=1}^n$ is a projection in $M_n(\cP_{n,c}^{op})$ with $\sum_{a=1}^c P_a^{op}=I_n$. By a similar argument to the above, the map $p_{a,ji}^{op} \mapsto p_{a,ij}$ extends to a $*$-homomorphism $\rho:\cP_{n,c}^{op} \to \cP_{n,c}$. Since $\pi$ and $\rho$ are mutual inverses on the generators of the respective algebras, they both extend to isomorphisms, yielding the desired result.
\end{proof}

\section{Synchronous Quantum Input--Classical Output Correlations} \label{sec: sync cor}

We now generalize the notion of synchronous correlations.  Recall that a correlation $P = (p(a,b|x,y)) \in C(n,k)$ is called {\it synchronous} if $p(a,b|x,x) = 0$ whenever $a \ne b$ \cite{HMPS19}.   

In the following considerations, we fix once and for all an orthonormal basis $\{e_1,...,e_n\}$ for $\bC^n$.  

\begin{definition}
Let $S \subseteq [n]$. We define the \textbf{maximally entangled Bell state corresponding to} $S$ as the vector
$$\varphi_{S}=\frac{1}{\sqrt{|S|}} \sum_{j \in S} e_j \otimes e_j.$$
\end{definition}

\begin{definition}
Let $X \in \mathcal{Q}_t(n,c)$ be a correlation in $n$-dimensional quantum inputs and $c$ classical outputs, where $t \in \{ loc,q,qs,qa,qc\}$. We say that $X$ is \textbf{synchronous} provided that there is a partition $S_1 \dot{\cup} \cdots \dot{\cup} S_{\ell}$ of $[n]$ with the property that, if $a \neq b$, then
$$p(a,b|\varphi_{S_r})=0 \text{ for all } 1 \leq r \leq \ell.$$
We define the subset
$$\mathcal{Q}_t^s(n,c)=\{ X \in \mathcal{Q}_t(n,c): X \text{ is synchronous}\}.$$
\end{definition}

The following proposition gives a very useful description of synchronicity in terms of the entries of the matrices involved in the correlation.

\begin{proposition}
\label{proposition: synchronicity in terms of the correlation}
Let $X=(X_{(i,j),(k,\ell)}^{(a,b)}) \in \mathcal{Q}_t(n,c)$ for $t \in \{loc,q,qs,qa,qc\}$. The following are equivalent:
\begin{enumerate}
\item
$X$ is synchronous.
\item
$X$ satisfies the equation
\begin{equation}
\frac{1}{n} \sum_{a=1}^c \sum_{i,j=1}^n X_{(i,j),(i,j)}^{(a,a)}=1. \label{synchronous in terms of sum of ij entries on each a}
\end{equation}
\item
If $a \neq b$, then
\begin{equation}\sum_{i,j=1}^n X^{(a,b)}_{(i,j),(i,j)}=0. \label{synchronous in terms of sum of ij entries on a neq b}
\end{equation}
\end{enumerate}
\end{proposition}

\begin{proof}
Suppose that $X$ can be represented using the PVMs $\{P_a\}_{a=1}^c$ in $\cB(\bC^n \otimes \cH)$ and $\{Q_b\}_{b=1}^c$ in $\cB(\cH \otimes \bC^n)$ and the state $\chi \in \cH$. We observe that, if $S \subseteq [n]$, then
\begin{align*}
p(a,b|\varphi_{S})&=\frac{1}{|S|} \sum_{i,j \in S} \la (P_a \otimes I_n)(I_n \otimes Q_b)(e_j \otimes \chi \otimes e_j),e_i \otimes \chi \otimes  e_i \ra \\
&=\frac{1}{|S|} \sum_{i,j \in S} \la P_{a,ij} Q_{b,ij} \chi,\chi \ra \\
&=\frac{1}{|S|} \sum_{i,j \in S} X^{(a,b)}_{(i,j),(i,j)}.
\end{align*}
Suppose that $X$ is synchronous, and let $S_1,...,S_{\ell}$ be a partition of $[n]$ for which $p(a,b|\varphi_{S_r})=0$ whenever $a \neq b$ and $1 \leq r \leq \ell$. Then the above calculation shows that $\sum_{i,j \in S_r} X^{(a,b)}_{(i,j),(i,j)}=0$ for all $r$. Summing over all $r$, it follows that $\sum_{i,j=1}^n X^{(a,b)}_{(i,j),(i,j)}=0$ whenever $a \neq b$. Hence, (1) implies (3).

Next, we show that (3) implies (2). Notice that, for any $X \in \mathcal{Q}_{qc}(n,c)$,
\begin{align*}
\sum_{a,b=1}^c \sum_{i,j=1}^n X^{(a,b)}_{(i,j),(k,\ell)}&=\sum_{a,b=1}^c \sum_{i,j=1}^n \la P_{a,ij}Q_{b,ij} \chi,\chi \ra \\
&=\sum_{i,j=1}^n \left\la \left(\sum_{a=1}^c P_{a,ij}\right)\left(\sum_{b=1}^c Q_{b,ij} \right)\chi,\chi \right\ra \\
&=\sum_{i=1}^n \la \chi,\chi \ra=n,
\end{align*}
where we have used the fact that $\sum_{a=1}^c P_a=\sum_{b=1}^c Q_b=I_n$ implies that $\sum_{a=1}^c P_{a,ij}=\sum_{b=1}^c Q_{b,ij}$ is $I$ when $i=j$ and $0$ otherwise. Therefore,
$$\frac{1}{n} \sum_{i,j=1}^n X^{(a,b)}_{(i,j),(i,j)}=1,$$
which shows that (2) holds.

Lastly, if (2) holds, then (1) immediately follows using the single-set partition $S=[n]$.
\end{proof}

\begin{remark}
In the case of a correlation $p(a,b|x,y) \in C_t(n,c)$ with $n$ classical inputs and $c$ classical outputs, using the $[n]=\{1\} \cup \{2\} \cup \cdots \cup \{n\}$, we see that any synchronous correlation in $C_t(n,c)$ is a synchronous correlation in the sense of the definition above. In this way, we see that $C_t^s(n,c) \subseteq \mathcal{Q}_t^s(n,c)$.
\end{remark}

We wish to show the analogue of \cite[Theorem~5.5]{PSSTW16}; that is, synchronous correlations with $n$-dimensional inputs and $c$ outputs arise from tracial states on the algebra generated by Alice's operators (respectively, Bob's operators). We will also see that, in any realization of a synchronous correlation, Bob's operators can be described naturally in terms of Alice's operators. By a \textbf{realization} of $X \in \mathcal{Q}_{qc}(n,c)$, we simply mean a $4$-tuple $(\{P_a\}_{a=1}^c,\{Q_b\}_{b=1}^c,\cH,\psi)$, where $\{P_a\}_{a=1}^c$ is a PVM on $\bC^n \otimes \cH$, $\{Q_b\}_{b=1}^c$ is a PVM on $\cH \otimes \bC^n$, $\psi$ is a state in $\cH$, and $[P_a \otimes I_n,I_n \otimes Q_b]=0$ for all $a,b$.

\begin{theorem}
\label{theorem: synchronicity condition}
Let $X=(X_{(i,j),(k,\ell)}^{(a,b)}) \in \mathcal{Q}_{qc}^s(n,c)$. Let $(\{P_a\}_{a=1}^c,\{Q_b\}_{b=1}^c,\cH,\psi)$ be a realization of $X$. Then:
\begin{enumerate}
\item
$Q_{a,ij}\psi=P_{a,ij}^*\psi \text{ for all } 1 \leq a \leq c \text{ and } 1 \leq i,j \leq n$.
\item
The state $\rho=\la (\cdot) \psi,\psi \ra$ is a tracial state on the $C^*$-algebra $\cA$ generated by $\{P_{a,ij}: 1 \leq a \leq c, \, 1 \leq i,j \leq n\}$, and on the $C^*$-algebra $\cB$ generated by $\{Q_{b,k\ell}: 1 \leq b \leq c, \, 1 \leq k,\ell \leq n\}$.
\end{enumerate}
Conversely, if $P_{a,ij}$ are operators in a tracial $C^*$-algebra $\cA$ with a trace $\tau$, such that the operators $P_a=(P_{a,ij}) \in M_n(\cA)$ form a PVM with $c$ outputs, then $(\tau(P_{a,ij}P_{b,k\ell}^*))$ defines an element of $\mathcal{Q}_{qc}^s(n,c)$.
\end{theorem}

\begin{proof}
Suppose $X \in \mathcal{Q}_{qc}^s(n,c)$, with realization $(\{P_a\}_{a=1}^c,\{Q_b\}_{b=1}^c,\cH,\psi)$. By Proposition \ref{proposition: synchronicity in terms of the correlation}, we have 
\begin{align}
1&=\frac{1}{n} \sum_{a=1}^c \sum_{i,j=1}^n X^{(a,a)}_{(i,j),(i,j)} \\
&=\frac{1}{n} \sum_{a=1}^c \sum_{i,j=1}^n \la P_{a,ij}Q_{a,ij} \psi,\psi \ra \nonumber \\
&\leq \frac{1}{n} \sum_{a=1}^c \sum_{i,j=1}^n |\la P_{a,ij}Q_{a,ij}\psi,\psi \ra| \label{absvalueineq} \\
&=\frac{1}{n} \sum_{a=1}^c \sum_{i,j=1}^n |\la Q_{a,ij}\psi,P_{a,ij}^*\psi \ra | \nonumber \\
&\leq \frac{1}{n} \sum_{a=1}^c \sum_{i,j=1}^n \|Q_{a,ij}\psi\| \|P_{a,ij}^*\psi\| \label{firstCS} \\
&\leq \frac{1}{n}  \left( \sum_{a=1}^c\sum_{i,j=1}^n \|Q_{a,ij}\psi \|^2 \right)^{\frac{1}{2}} \left( \sum_{a=1}^c\sum_{i,j=1}^n \|P_{a,ij}^*\psi \|^2 \right)^{\frac{1}{2}} \\
&=\frac{1}{n} \left( \sum_{a=1}^c\sum_{i,j=1}^n \la Q_{a,ij}^*Q_{a,ij}\psi,\psi \ra \right)^{\frac{1}{2}} \left( \sum_{a=1}^c\sum_{i,j=1}^n \la P_{a,ij}P_{a,ij}^* \psi, \psi \ra \right)^{\frac{1}{2}} \nonumber \\
&= \frac{1}{n}\left( \sum_{a=1}^c\sum_{i,j=1}^n \la Q_{a,ji}Q_{a,ij}\psi,\psi \ra \right)^{\frac{1}{2}} \left(\sum_{a=1}^c \sum_{i,j=1}^n \la P_{a,ij}P_{a,ji} \psi,\psi \ra \right)^{\frac{1}{2}} \nonumber
\end{align}
Since $P_a$ and $Q_a$ are projections, the last line is equal to
\begin{align*}
\frac{1}{n} \left( \sum_{a=1}^c\sum_{j=1}^n \la Q_{a,jj}\psi,\psi \ra \right)^{\frac{1}{2}} \left(\sum_{a=1}^c \sum_{i=1}^n \la P_{a,ii} \psi,\psi \ra \right)^{\frac{1}{2}} &=\frac{1}{n} \left( \sum_{j=1}^n \la I_{\cH}\psi,\psi \ra \right)^{\frac{1}{2}} \left( \sum_{i=1}^n \la I_{\cH}\psi,\psi \ra \right)^{\frac{1}{2}} \\
&=\frac{1}{n}\cdot \sqrt{n} \cdot \sqrt{n} \\
&=1.
\end{align*}
Therefore, all of these inequalities are equalities. Then (\ref{absvalueineq}) implies that
$$X^{(a,a)}_{(i,j),(i,j)} \geq 0 \text{ for all } 1 \leq a \leq c, \, 1\leq i,j \leq n.$$
The equality case of (\ref{firstCS}) shows that
\begin{equation}
Q_{a,ij}\psi=\alpha_{a,ij} P_{a,ij}^*\psi \text{ for some } \alpha_{a,ij} \in \bT. \label{firstCSconsequence}
\end{equation}
Then equation (\ref{firstCSconsequence}) yields
$$X^{(a,a)}_{(i,j),(i,j)}=\alpha_{a,ij} \la P_{a,ij}P_{a,ij}^*\psi,\psi \ra=\alpha_{a,ij}\|P_{a,ij}^*\psi\|^2.$$
Since $X_{(i,j),(i,j)}^{(a,a)} \geq 0$ and $\|P_{a,ij}^*\psi \|^2 \geq 0$, we either have $P_{a,ij}^*\psi=0$ or $\alpha_{a,ij}=1$. In either case, we obtain
$$Q_{a,ij}\psi=P_{a,ij}^*\psi,$$
as desired.

To prove (2), it suffices to show that it holds for $\cA=C^*(\{P_{a,ij}\}_{a,i,j})$; a similar argument works for $\cB=C^*(\{Q_{b,k\ell}\}_{b,k,\ell})$.  Let $\rho:\cA \to \bC$ be the state given by $\rho(X)=\la X \psi,\psi \ra$.  Let $W=P_{a_1,i_1j_1}^{m_1} \cdots P_{a_k,i_kj_k}^{m_k}$ be a word in $\{P_{a,ij},P_{a,ij}^*\}_{a,i,j}$, where we denote by $P_{a_{\ell},i_{\ell}j_{\ell}}^{-1}$ the operator $P_{a_{\ell},i_{\ell}j_{\ell}}^*$ and let $m_{\ell} \in \{-1,1\}$.  We will first show that $W\psi=Q_{a_k,i_kj_k}^{-m_k} \cdots Q_{a_1,i_1j_1}^{-m_1} \psi$, where $Q_{a_{\ell},i_{\ell}j_{\ell}}^{-1}:=Q_{a_{\ell},i_{\ell}j_{\ell}}^*$.  Using the fact that $P_{a,ij}$ and $Q_{b,k\ell}$ $*$-commute for each $a,b,i,j,k,\ell$, we obtain
\begin{align*}
W\psi&=P_{a_1,i_1j_1}^{m_1} \cdots P_{a_k,i_kj_k}^{m_k}\psi \\
&=P_{a_1,i_1j_1}^{m_1} \cdots P_{a_{k-1},i_{k-1}j_{k-1}}^{m_{k-1}} Q_{a_k,i_kj_k}^{-m_k}\psi \\
&=Q_{a_k,i_kj_k}^{-m_k} (P_{a_1,i_1j_1}^{m_1} \cdots P_{a_{k-1},i_{k-1}j_{k-1}}^{m_{k-1}})\psi.
\end{align*}
and the desired equality easily follows by induction on $k$.  For $1 \leq a \leq c$ and $1 \leq i,j \leq n$,
\begin{align*}
\rho(P_{a,ij}W)&=\la P_{a,ij}W\psi,\psi \ra \\
&=\la P_{a,ij}(Q_{a_1,i_1j_1}^{m_1} \cdots Q_{a_k,i_kj_k}^{m_k})^*\psi,\psi \ra \\
&=\la (Q_{a_1,i_1j_1}^{m_1} \cdots Q_{a_k,i_kj_k}^{m_k})^*P_{a,ij}\psi,\psi \ra \\
&=\la P_{a,ij}\psi,Q_{a_1,i_1j_1}^{m_1} \cdots Q_{a_k,i_kj_k}^{m_k}\psi \ra \\
&=\la P_{a,ij}\psi,(P_{a_1,i_1j_1}^{m_1} \cdots P_{a_k,i_kj_k}^{m_k})^*\psi \ra \\
&=\la P_{a,ij}\psi,W^*\psi \ra \\
&=\la WP_{a,ij}\psi,\psi \ra=\rho(WP_{a,ij}).
\end{align*}
In the same way, $\rho(P_{a,ij}P_{b,k\ell}W)=\rho(WP_{a,ij}P_{b,k\ell})$.  It follows by induction, linearity and continuity that $\rho$ is tracial on $\cA$, as desired.

For the converse direction, we recall the standard fact that, if $\cA$ is a unital $C^*$-algebra and $\tau$ is a trace on $\cA$, then there is a state $s:\cA \otimes_{\max} \cA^{op} \to \bC$ satisfying $s(x \otimes y^{op})=\tau(xy)$ for all $x,y \in \cA$. Thus, if $P_1,...,P_c \in M_n(\cA)$ is a projection-valued measure, then
$$s(P_{a,ij} \otimes P_{b,k\ell}^{op})=\tau(P_{a,ij}P_{b,k\ell}) \, \forall 1 \leq a,b \leq c, \, 1 \leq i,j,k,\ell \leq n.$$
Applying the universal property of $\cP_{n,c}$, we obtain a state $\gamma:\cP_{n,c} \otimes_{\max} \cP_{n,c}^{op} \to \bC$ satisfying
$$\gamma(p_{a,ij} \otimes p_{b,k\ell}^{op})=\tau(P_{a,ij}P_{b,k\ell}).$$
By Lemma \ref{lemma: opposite algebra}, the map $p_{a,ij} \otimes p_{b,k\ell} \mapsto \tau(P_{a,ij}P_{b,\ell k})=\tau(P_{a,ij}P_{b,k \ell}^*)$ defines a state on $\cP_{n,c} \otimes_{\max} \cP_{n,c}$. Then Theorem \ref{theorem: characterizing qc} shows that
$$X:=\tau(P_{a,ij}P_{b,k\ell}^*)$$
defines an element of $\mathcal{Q}_{qc}(n,c)$. If $a \neq b$, then
\begin{align*}
\sum_{i,j=1}^n X^{(a,b)}_{(i,j),(i,j)}&=\sum_{i,j=1}^n \tau(P_{a,ij}P_{b,ij}^*) \\
&=\Tr \otimes \tau(P_aP_b)=0,
\end{align*}
since $P_aP_b=0$. By Proposition \ref{proposition: synchronicity in terms of the correlation}, $X=(X^{(a,b)}_{(i,j),(k,\ell)}) \in \mathcal{Q}_{qc}^s(n,c)$.
\end{proof}

In light of Theorem \ref{theorem: synchronicity condition}, we may refer to a \textbf{synchronous} $t$\textbf{-strategy} $(\{P_a\}_{a=1}^c,\chi)$ when referring to a $t$-strategy $(\{P_a\}_{a=1}^c,\{Q_b\}_{b=1}^c,\chi)$ where the associated correlation is synchronous.

\begin{corollary}
Let $(X_{(i,j),(k,\ell)}^{(a,b)}) \in \mathcal{Q}_t^s(n,c)$ where $t \in \{loc,q,qs,qa,qc\}$. Then:
\begin{enumerate}
\item
$X_{(i,i),(j,j)}^{(a,b)} \geq 0$ for all $1 \leq a,b \leq c$ and $1 \leq i,j \leq n$.
\item
$X^{(a,b)}_{(i,j),(k,\ell)}=\overline{X^{(a,b)}_{(j,i),(\ell,k)}}$.
\item
For any $1 \leq a \neq b \leq c$ and $1 \leq i,j \leq n$, we have
$$\sum_{k=1}^n X^{(a,b)}_{(i,k),(j,k)}=\sum_{k=1}^n X^{(a,b)}_{(k,i),(k,j)}=0.$$
\item
For any $1 \leq i,j \leq n$, we have
$$\sum_{a=1}^c\sum_{k=1}^n X^{(a,a)}_{(i,k),(j,k)}=\sum_{a=1}^c\sum_{k=1}^n X^{(a,a)}_{(k,i),(k,j)}=\delta_{ij}.$$
\end{enumerate}
\end{corollary}

\begin{proof}
By Theorem \ref{theorem: synchronicity condition}, we may choose projections $P_1,...,P_c \in M_n(\cA)$, for a unital $C^*$-algebra $\cA$, along with a tracial state $\tau$ on $\cA$ such that
$$X^{(a,b)}_{(i,j),(k,\ell)}=\tau(P_{a,ij}P_{b,k\ell}^*) \text{ for all } 1 \leq a,b \leq c, \, 1 \leq i,j,k,\ell \leq n.$$
Since $P_a$ is a projection, it defines a positive element of $M_n(\cA)$. Compressing to any diagonal block preserves positivity, which implies that $P_{a,ii} \in \cA^+$ for any $i$. Since $\tau$ is a trace, it follows that $\tau(P_{a,ii}P_{b,jj}) \geq 0$ for any $i,j,a,b$. Hence, (1) follows.

We note that (2) follows easily from the fact that $\tau$ is a trace and that, since $\tau$ is a state, one has $\tau(Y^*)=\overline{\tau(Y)}$ for all $Y \in \cA$.

To show (3), we observe that
\begin{align*}
\sum_{k=1}^n X_{(i,k),(j,k)}^{(a,b)}&=\sum_{k=1}^n \tau(P_{a,ik}P_{b,jk}^*) \\
&=\sum_{k=1}^n \tau(P_{a,ik}P_{b,kj}) \\
&=\tau \left( \sum_{k=1}^n P_{a,ik}P_{b,kj} \right) \\
&=\tau( (P_aP_b)_{ij})=0,
\end{align*}
since $P_aP_b=0$. Similarly, $\sum_{k=1}^n X_{(k,i),(k,j)}^{(a,b)}=0$ when $a \neq b$.

A similar argument establishes (4). Indeed, we have
$$\sum_{a=1}^c \sum_{k=1}^n X_{(i,k),(j,k)}^{(a,a)}=\sum_{a=1}^c \sum_{k=1}^n \tau(P_{a,ik}P_{a,kj})=\tau \left( \sum_{a=1}^c P_{a,ij} \right),$$
and this latter sum is $\delta_{ij}$, since $\sum_{a=1}^c P_a=I$. The other equation in (4) follows similarly.
\end{proof}

\begin{remark}
It makes sense (and we will have occasion) to discuss synchronicity of a strategy with respect to a different orthonormal basis 
$\mb{v} = \{v_1,...,v_n\}$ of $\bC^n$. In this case, a $qc$-strategy $(\{P_a\}_{a=1}^c,\{Q_b\}_{b=1}^c,\chi)$ is said to be \textbf{synchronous with respect to} $\{v_1,...,v_n\}$ if there is a partition $S_1 \cup \cdots \cup S_s$ of $[n]$ such that for each $r$ and $\varphi_{S_r,\mb{v}} :=\frac{1}{\sqrt{|S_r|}} \sum_{j \in S_r} v_j \otimes v_j$, we have
$$p(a,b|\varphi_{S_r,\mb{v}})=0 \text{ if } a \neq b.$$
One can then write down an analogue of Theorem \ref{theorem: synchronicity condition} in this context. Alternatively, one can simply let $\widetilde{P}_a=U^*P_aU$ and $\widetilde{Q}_b=U^*Q_bU$, where $U$ is the unitary satisfying $Ue_i=v_i$ for all $i$. Then applying Theorem \ref{theorem: synchronicity condition} relates the entries of $\widetilde{Q}_a$ to the entries of $\widetilde{P}_a$, while showing that the state $\la (\cdot)\chi,\chi \ra$ is a trace on the algebra generated by the entries of the operators $\widetilde{Q}_a$ (respectively, $\widetilde{P}_a$). Since $P_a=U\widetilde{P}_a U^*$ and $Q_b=U\widetilde{Q}_bU^*$, the entries of $P_a$ (respectively, $Q_b$) are linear combinations of the entries of $\widetilde{P}_a$ (respectively, $\widetilde{Q}_b$), so it follows that the algebra generated by the entries of the operators $P_a$ (respectively, $Q_b$) is the same as the algebra generated by the entries of $\widetilde{P}_a$ (respectively, $\widetilde{Q}_b$).
\end{remark}

It is helpful to describe the simplest ways to realize synchronous correlations. To that end, we spend the rest of this section describing the simplest realizations for $t\in \{loc,q,qs\}$. We start with the case of $\mathcal{Q}_{loc}^s(n,c)$.

\begin{corollary}
\label{corollary: characterizing synchronous loc}
Let $X \in (M_n \otimes M_n)^{c^2}$. Then $X$ belongs to $\mathcal{Q}_{loc}^s(n,c)$ if and only if there is a unital, commutative $C^*$-algebra $\cA$, a projection-valued measure $\{P_a\}_{a=1}^c \subseteq M_n(\cA)$ for $1 \leq a \leq c$, and a faithful state $\psi \in \cS(\cA)$ such that, for all $1 \leq a,b \leq c$ and $1 \leq i,j,k,\ell \leq n$,
\[ X_{(i,j),(k,\ell)}^{(a,b)}=\psi(P_{a,ij}P_{b,k\ell}^*).\]
Moreover, if $X$ is an extreme point in $\mathcal{Q}_{loc}^s(n,c)$, then we may take $\cA=\bC$.
\end{corollary}

\begin{proof}
If $X \in \mathcal{Q}_{loc}^s(n,c)$, then by definition of loc-correlations, $X$ can be written using projection-valued measures $\{P_a\}_{a=1}^c$ and $\{Q_b\}_{b=1}^c$ in $M_n(\bofh)$, along with a state $\chi \in \cH$, such that $X_{(i,j),(k,\ell)}^{(a,b)}=\la P_{a,ij}Q_{b,k\ell}\chi,\chi \ra$ and the $C^*$-algebra $\cA$ generated by the set of all entries $P_{a,ij}$ and $Q_{b,\ell}$ is a commutative $C^*$-algebra. Applying Theorem \ref{theorem: synchronicity condition}, we can write $X_{(i,j),(k,\ell)}^{(a,b)}=\psi(P_{a,ij}P_{b,k\ell}^*)$, where $\psi(\cdot)=\la (\cdot)\chi,\chi \ra$. As this state is tracial, by replacing $\cA$ with its quotient by the kernel of the GNS representation of $\psi$ if necessary, we may assume without loss of generality that $\psi$ is faithful, which establishes the forward direction. The converse follows by the converse of Theorem \ref{theorem: synchronicity condition} and the definition of $\mathcal{Q}_{loc}(n,c)$.

To establish the claim about extreme points, we note that the proof of Proposition \ref{proposition: loc is closed} shows that every element of $\mathcal{Q}_{loc}(n,c)$ is a limit of convex combinations of correlations arising from PVMs in $M_n(\bC)$. Evidently the set of elements of $\mathcal{Q}_{loc}(n,c)$ that have realizations using PVMs in $M_n(\bC)$ is a closed set. As $\mathcal{Q}_{loc}(n,c)$ is compact and convex, the converse of the Krein-Milman theorem shows that extreme points in $\mathcal{Q}_{loc}(n,c)$ must have a realization using PVMs in $M_n(\bC)$. Now, the proof of the forward direction of Theorem \ref{theorem: synchronicity condition} shows that $\frac{1}{n} \sum_{a=1}^c \sum_{i,j=1}^n Y_{(i,j),(i,j)}^{(a,b)} \leq 1$ for any $Y \in \mathcal{Q}_{qc}(n,c)$. Moreover, this inequality is an equality if and only if $Y$ is synchronous, by Proposition \ref{proposition: synchronicity in terms of the correlation}. Hence, $\mathcal{Q}_{loc}^s(n,c)$ is a face in $\mathcal{Q}_{loc}(n,c)$, so extreme points in $\mathcal{Q}_{loc}^s(n,c)$ are also extreme points in $\mathcal{Q}_{loc}(n,c)$. This shows that $X$ has a realization using the algebra $\cA=\bC$.
\end{proof}

\begin{corollary}
\label{corollary: characterizing synchronous q}
Let $X \in (M_n \otimes M_n)^{c^2}$. Then $X$ belongs to $\mathcal{Q}_q^s(n,c)$ if and only if there is a finite-dimensional $C^*$-algebra $\cA$, a projection-valued measure $\{P_a\}_{a=1}^c \subseteq M_n(\cA)$ for $1 \leq a \leq c$, and a faithful tracial state $\psi \in \cS(\cA)$ such that, for all $1 \leq a,b \leq c$ and $1 \leq i,j,k,\ell \leq n$,
\[ X_{(i,j),(k,\ell)}^{(a,b)}=\psi(P_{a,ij}P_{b,k\ell}^*).\]
Moreover, if $X$ is an extreme point in $\mathcal{Q}_q^s(n,c)$, then we may take $\cA=M_d$ for some $d$, and hence $\psi=\tr_d$, where $\tr_d$ is the normalized trace on $M_d$.
\end{corollary}

\begin{proof}
If $X$ belongs to $\mathcal{Q}_q^s(n,c)$, then one can write $X=(\la (P_{a,ij} \otimes Q_{b,k\ell})\chi,\chi \ra)$ for projection-valued measures $\{P_a\}_{a=1}^c \subseteq M_n(\cB(\cH_A))$ and $\{Q_b\}_{b=1}^c \subseteq M_n(\cB(\cH_B))$ on finite-dimensional Hilbert spaces $\cH_A$ and $\cH_B$, along with a unit vector $\chi \in \cH_A \otimes \cH_B$. By Theorem \ref{theorem: synchronicity condition}, we may write $X=\psi(P_{a,ij}P_{b,k\ell}^*)$ where $\psi$ is the (necessarily faithful) tracial state on the finite-dimensional $C^*$-algebra $\cA$ generated by the set $\{ P_{a,ij}: 1 \leq a \leq c, \, 1 \leq i,j \leq n \}$.

Conversely, if $X$ can be written as $X=(\psi(P_{a,ij}P_{b,k\ell}^*))$ for a projection-valued measure $\{P_a\}_{a=1}^c \subseteq M_n(\cA)$, where $\cA$ is a finite-dimensional $C^*$-algebra with  a faithful trace $\psi$ on $\cA$, then the proof of Theorem \ref{theorem: synchronicity condition} yields a finite-dimensional realization of $X$ as an element of $\mathcal{Q}_{qc}^s(n,c)$. By Lemma \ref{lemma: finite-dimensional realization of commuting}, we must have $X \in \mathcal{Q}_q^s(n,c)$.

Now, assume that $X$ is extreme in $\mathcal{Q}_q^s(n,c)$. Since $\cA$ is finite-dimensional, it is $*$-isomorphic to $\bigoplus_{r=1}^m M_{k_r}$ for some $r$ and numbers $k_1,...,k_r \in \bN$. Since $\psi$ is a trace on $\cA$, there must be $t_1,...,t_m \geq 0$ such that $\sum_{r=1}^m t_r=1$ and $\psi(\cdot)=\sum_{r=1}^m t_r \tr_{k_r}(\cdot)$, where $\tr_{k_r}$ is the normalized trace on $M_{k_r}$. Writing $P_{a,ij}=\bigoplus_{r=1}^m P_{a,ij}^{(r)}$ for each $1 \leq a \leq c$ and $1 \leq i,j \leq n$, where $P_{a,ij}^{(r)} \in M_{k_r}$, we have
\[ X_{(i,j),(k,\ell)}^{(a,b)}=\sum_{r=1}^m t_r \tr_{k_r}(P_{a,ij}^{(r)} (P_{b,k\ell}^{(r)})^*). \]
Since $P_a^{(r)}=(P_{a,ij}^{(r)}) \in M_n(M_{k_r})$ must define an orthogonal projection and $\sum_{a=1}^c P_a^{(r)}=I_n \otimes I_{k_r}$, it follows that
$X_{r,(i,j),(k,\ell)}^{(a,b)}=\tr_{k_r}(P_{a,ij}^{(r)} (P_{b,k\ell}^{(r)})^*) \in \mathcal{Q}_q^s(n,c)$, and
$\sum_{r=1}^m t_r X_{r,(i,j),(k,\ell)}^{(a,b)}=X^{(a,b)}_{(i,j),(k,\ell)}$. Therefore, $X_{r,(i,j),(k,\ell)}^{(a,b)}=X^{(a,b)}_{(i,j),(k,\ell)}$ for each $r$. This shows that we may take $\cA$ to be a matrix algebra, completing the proof.
\end{proof}

We will end this section by showing that $\mathcal{Q}_{qs}^s(n,c)=\mathcal{Q}_q^s(n,c)$. To prove this fact, we use a similar approach to \cite{KPS18}. In fact, by an application of Proposition \ref{proposition: classical inputs inside of quantum inputs}, the following theorem is a direct generalization of the analogous result in \cite{KPS18}.

\begin{theorem}
For each $n,c \in \bN$, we have $\mathcal{Q}_{qs}^s(n,c)=\mathcal{Q}_q^s(n,c)$.
\end{theorem}

\begin{proof}
Let $X=(X^{(a,b)}_{(i,j),(k,\ell)}) \in \mathcal{Q}_{qs}^s(n,c)$, and write
$$X^{(a,b)}_{(i,j),(k,\ell)}=\la (P_{a,ij} \otimes Q_{b,k\ell})\psi,\psi \ra,$$
where $P_a=(P_{a,ij})$ is a projection in $\bC^n \otimes \cH_A$ for each $1 \leq a \leq c$, $Q_b=(Q_{b,ij})$ is a projection in $\cH_B \otimes \bC^n$ for each $1 \leq b \leq c$, $\sum_{a=1}^c P_a=I_{\bC^n \otimes \cH_A}$, $\sum_{b=1}^c Q_b=I_{\cH_B \otimes \bC^n}$, and $\psi \in \cH_A \otimes \cH_B$ is a state. We can arrange to have $\dim(\cH_A)=\dim(\cH_B)$. For example, if $\dim(\cH_A)<\dim(\cH_B)$, then we choose a Hilbert space $\cH_C$ with $\dim(\cH_A \oplus \cH_C)=\dim(\cH_B)$, and extend $P_a$ by defining $\widetilde{P}_{a,ij}=P_{a,ij} \oplus \delta_{ij} I_{\cH_C}$. Then
$$\la (\widetilde{P}_{a,ij} \otimes Q_{b,k\ell})\psi,\psi \ra=\la (P_{a,ij} \otimes Q_{b,k\ell})\psi,\psi \ra=X^{(a,b)}_{(i,j),(k,\ell)}.$$
In this way, we may assume without loss of generality that $\dim(\cH_A)=\dim(\cH_B)$.

We write down a Schmidt decomposition
$$\psi=\sum_{p=1}^{\infty} \alpha_p e_p \otimes f_p,$$
where $\{e_p\}_{p=1}^{\infty} \subseteq \cH_A$ and $\{f_p\}_{p=1}^{\infty} \subseteq \cH_B$ are orthonormal, and $\alpha_1 \geq \alpha_2 \geq ... \geq 0$ are such that $\sum_{p=1}^{\infty} \alpha_p^2=1$. If one extends these orthonormal sets to orthonormal bases for $\cH_A$ and $\cH_B$ respectively, and defines additional $\alpha_p$'s to be $0$, then after direct summing a Hilbert space on one side if necessary, we may assume that $\dim(\cH_A)=\dim(\cH_B)$ and that $\{e_r\}_{r \in I}$ is an orthonormal basis for $\cH_A$, and $\{f_s\}_{s \in I}$ is an orthonormal basis for $\cH_B$.

We rewrite the (at most) countable set $\{\alpha_q: \alpha_q \neq 0\}=\{\beta_v: v \in V\}$, where $V=\{1,2,...\}$ and $\beta_v>\beta_{v+1}$ for all $v \in V$. We define $K_v=\{ e_q: \alpha_q=\beta_v\}$ and $L_v=\{ f_q: \alpha_q=\beta_v\}$, and define subspaces $\cK_v=\text{span}(K_v)$ and $\cL_v=\text{span}(L_v)$ of $\cH_A$ and $\cH_B$, respectively. Since $\sum_{q=1}^{\infty} |\alpha_q|^2=1$, it follows that each $K_v$ and $L_v$ must be finite, so that $\cK_v$ and $\cL_v$ are finite-dimensional. We will show that each $\cK_v$ is invariant for the operators $\{P_{a,ij}: 1 \leq a \leq c, \, 1 \leq i,j \leq n\}$, and that each $\cL_v$ is invariant for the operators $\{Q_{b,k\ell}: 1 \leq b \leq c, \, 1 \leq k,\ell \leq n\}$. To this end, let $\omega$ be a primitive $c$-th root of unity, and define order $c$ unitaries $U=\sum_{a=1}^c \omega^a P_a \in \cB(\bC^n \otimes \cH_A)$ and $V=\sum_{b=1}^c \omega^{-b} Q_b \in \cB(\cH_A \otimes \bC^n)$. Since $X$ is synchronous, by Theorem \ref{theorem: synchronicity condition}, we know that
$$(I_{\cH_A} \otimes Q_{a,ij}^*)\psi=(P_{a,ij} \otimes I_{\cH_B})\psi \text{ and } (I_{\cH_A} \otimes Q_{a,ij}Q_{b,ij}^*)\psi=(P_{b,ij}P_{a,ij}^* \otimes I_{\cH_B})\psi.$$
Since $U_{ij}U_{ij}^*=\sum_{a,b=1}^c \omega^{a-b} P_{a,ij}P_{b,ij}^*$ and $V_{ij}V_{ij}^*=\sum_{a,b=1}^c \omega^{b-a} Q_{a,ij}Q_{b,ij}^*$, it follows that
$$(I_{\cH_A} \otimes V_{ij}^*)\psi=(U_{ij} \otimes I_{\cH_B})\psi \text{ and } (I_{\cH_A} \otimes V_{ij}V_{ij}^*)\psi=(U_{ij}U_{ij}^* \otimes I_{\cH_B})\psi.$$
Using this fact and the decomposition of $\psi$,
$$\alpha_q \la U_{ij} e_q,e_p \ra=\la (U_{ij} \otimes I_{\cH_B})\psi,e_p \otimes f_q \ra=\la (I_{\cH_A} \otimes V_{ij}^*)\psi,e_p \otimes f_q \ra=\alpha_p \la V_{ij}^*f_p,f_q \ra.$$
Since $U$ and $V$ are unitary, it follows that, for all $p$,
$$\sum_{i,j=1}^n \|U_{ij}^*e_p\|^2=\sum_{i,j=1}^n \la U_{ij}U_{ij}^* e_p,e_p \ra=n \text{ and } \sum_{i,j=1}^n \|U_{ij}e_p\|^2=\sum_{i,j=1}^n \la U_{ij}^*U_{ij}e_p,e_p \ra=n.$$
Similarly, $\sum_{i,j=1}^n \|V_{ij}^*f_q\|^2=\sum_{i,j=1}^n \|V_{ij}f_q\|^2=n$. Suppose that $q$ is such that $e_q \in K_1$. Then using the fact that $\alpha_q=\alpha_1$ and that $\alpha_p \leq \alpha_1$ for all $p$ yields
\begin{align*}
n|\alpha_1|^2 &= \sum_{i,j=1}^n |\alpha_1|^2 \|V_{ij}^* f_q\|^2 \\
&\geq \sum_{i,j=1}^n \sum_{p=1}^{\infty} |\alpha_p|^2 |\la V_{ij}^* f_p,f_q \ra|^2 \\
&=\sum_{i,j=1}^n \sum_{p=1}^{\infty} |\alpha_q|^2 |\la U_{ij}e_q,e_p \ra|^2 \\
&=|\alpha_1|^2 \sum_{i,j=1}^n \sum_{p=1}^{\infty} |\la U_{ij}e_q,e_p \ra|^2 \\
&=|\alpha_1|^2 \sum_{i,j=1}^n \sum_{p=1}^{\infty} |\la U_{ij}^*e_p,e_q \ra|^2 \\
&=|\alpha_1|^2 \sum_{i,j=1}^n \|U_{ij}^*e_q\|^2 \\
&=n|\alpha_1|^2.
\end{align*}
Therefore, we must have equality at all lines. If $p$ is such that $e_p \not\in K_1$, then since $\alpha_p<\alpha_1$, we must have $0=\sum_{i,j=1}^n |\alpha_p|^2|\la V_{ij}^*f_p,f_q \ra|^2=\sum_{i,j=1}^n |\alpha_q|^2 |\la U_{ij}e_q,e_p\ra|^2$. Therefore, $\la U_{ij}e_q,e_p \ra=0$ for each such $p$, which shows that $U_{ij}e_q \perp e_p$ for all $p$ with $e_p \not\in K_1$. Since this happens whenever $\alpha_q=\alpha_1$, the subspace $\cK_1$ must be invariant for every $U_{ij}$. By the same argument as above with the quantity $\sum_{i=1}^n |\alpha_1|^2 \|V_{ij}f_q\|^2$, it follows that $\cK_1$ is invariant for every $U_{ij}^*$. Therefore, $\cK_1$ is reducing for the operators $U_{ij}$, for all $1 \leq i,j \leq n$. A similar argument proves that $\cL_1$ is reducing for the operators $V_{k\ell}$, for all $1 \leq k, \ell \leq n$.

Now, choose $q$ such that $e_q \in K_2$ and $f_q \in L_2$. If $\alpha_p>\alpha_q$, then $\alpha_p=\alpha_1$, so that $e_p \in K_1$ and $f_p \in K_1$. The above shows that $\la U_{ij}e_q,e_p \ra=0$ and $\la U_{ij}^*e_q,e_p \ra=0$, so that $U_{ij}e_q \perp \cK_1$ and $U_{ij}^*e_q \perp \cK_2$. Similarly, $V_{k\ell}f_q \perp \cL_1$ and $V_{k\ell}^*f_q \perp \cL_1$. Then using a similar string of inequalities as before, one obtains $U_{ij}e_q \perp e_p$ whenever $p$ is such that $e_p \not\in K_2$ and $q$ is such that $e_q \in K_2$. Therefore, one finds that $\cK_2$ is invariant for each $U_{ij}$. A similar argument shows that $\cK_2$ is invariant for $U_{ij}^*$, so that $\cK_2$ is reducing for $\{U_{ij}: 1 \leq i,j \leq n\}$. The same argument shows that $\{V_{k\ell}: 1 \leq k,\ell \leq n\}$ reduces $\cK_2$.

It follows by induction that $\cK_v$ is reducing for $\{U_{ij}: 1 \leq i,j \leq n\}$ for all $v$ and that $\cL_v$ is reducing for $\{V_{k\ell}:1  \leq k,\ell \leq n \}$ for all $v$. By construction of the unitaries $U$ and $V$, we know that
$$P_a=\frac{1}{c} \sum_{d=1}^c \omega^{-ad} U^d \text{ and } Q_b=\frac{1}{c} \sum_{d=1}^c \omega^{bd} V^d.$$
Therefore, $\cK_v$ is reducing for each $P_{a,ij}$, and $\cL_v$ is reducing for each $Q_{b,k\ell}$, as desired.

Finally, we will exhibit $X=(X^{(a,b)}_{(i,j),(k,\ell)})$ as a countable convex combination of elements of $\mathcal{Q}_q^s(n,c)$. One can regard elements of $\mathcal{Q}_q^s(n,c)$ as elements of $\bC^{n^4c^2}$, or as elements of $\bR^{2(n^4c^2)}$. Then by a countably infinite version of Carath\'{e}odory's Theorem \cite{CW72}, this will show that $X$ belongs to $\mathcal{Q}_q^s(n,c)$, which will complete the proof. (As mentioned in \cite{KPS18}, this result from \cite{CW72} holds even with non-closed convex sets, of which $\mathcal{Q}_q^s(n,c)$ is an example.)

For each $v \in V$, we let $d_v=\dim(\cK_v)=\dim(\cL_v)=|K_v|=|L_v|$, which is finite. Define the state
$$\psi_v=\frac{1}{\sqrt{d_v}} \sum_{p: e_p \in K_v} e_p \otimes f_p,$$
 and define 
$$P_{v,a,ij}=P_{a,ij}|_{\cK_v} \text{ and } Q_{v,b,k\ell}=Q_{b,k\ell}|_{\cL_v}.$$
Since $\cK_v$ is reducing for $P_{a,ij}$, and since $P_a$ is a projection, the operator $P_{v,a}=(P_{v,a,ij})_{i,j=1}^n$ is a projection on $\bC^n \otimes \cK_v$. Similarly, $Q_{v,b}=(Q_{v,b,k\ell})_{k,\ell=1}^n$ is a projection on $\cL_v \otimes \bC^n$. Moreover, $\sum_{a=1}^c P_{v,a}=I_{\bC^n} \otimes I_{\cK_v}$ and $\sum_{b=1}^c Q_{v,b}=I_{\cL_v} \otimes I_{\bC^n}$. Therefore, the correlation
$$X_v=(X_{v,(i,j),(k,\ell)}^{(a,b)})=(\la( P_{v,a,ij} \otimes Q_{v,b,k\ell})\psi_v,\psi_v \ra)$$
belongs to $\mathcal{Q}_q(n,c)$ for each $v$. Set $t_v=\beta_v^2 d_v$. Then $t_v \geq 0$ and $\sum_{v \geq 1} t_v=\sum_{p=1}^{\infty} |\alpha_p|^2=1$. Finally, for each $1 \leq a,b \leq c$ and $1 \leq i,j \leq n$, we compute
\begin{align*}
X_{(i,j),(k,\ell)}^{(a,b)}&=\la (P_{a,ij} \otimes Q_{b,k\ell})\psi,\psi \ra \\
&=\sum_v \sum_{p,q: e_p,e_q \in K_v} \beta_v^2 \la (P_{a,ij} \otimes Q_{b,k\ell}) (e_p \otimes f_p),e_q \otimes f_q \ra \\
&=\sum_v \beta_v^2 d_v \la (P_{v,a,ij} \otimes Q_{v,b,k\ell})\psi_v,\psi_v \ra \\
&=\sum_v t_v X_{v,(i,j),(k,\ell)}^{(a,b)}.
\end{align*}
It follows that $X=\sum_v t_v X_v$. Since each $X_v \in \mathcal{Q}_q(n,c)$, it follows that $X \in \mathcal{Q}_q(n,c)$. Since $X$ is also synchronous, we obtain $X \in \mathcal{Q}_q^s(n,c)$, completing the proof.
\end{proof}

\section{Approximately finite-dimensional correlations} \label{sec: afd cor}

In this section, we will show that elements of $\mathcal{Q}_{qa}^s(n,c)$ arise from amenable traces. Equivalently, elements of $\mathcal{Q}_{qa}^s(n,c)$ can be represented using the trace on $\cR^{\cU}$ and projection-valued measures with $c$ outputs in $M_n(\cR^{\cU})$, where $\cR^{\cU}$ denotes an ultrapower of the hyperfinite $II_1$-factor $\cR$ by a free ultrafilter $\cU$ on $\bN$. The proof is similar to \cite[Section~3]{KPS18}, and the main result is a generalization of \cite[Theorem~3.6]{KPS18}. Relevant details on $\cR^{\cU}$ can be found in \cite{BO08}.

For amenable traces, we will use the following result of Kirchberg \cite[Proposition~3.2]{Ki94}, which can also be found in \cite[Theorem~6.2.7]{BO08}.

\begin{theorem}
\label{theorem: amenable traces}
Let $\cA$ be a separable $C^*$-algebra and let $\tau$ be a tracial state on $\cA$. The following statements are equivalent:
\begin{enumerate}
\item
The trace $\tau$ is \textbf{amenable}; i.e., whenever $\cA \subseteq \bofh$ is a faithful representation, then there is a state $\rho$ on $\bofh$ such that $\rho_{|\cA}=\tau$ and $\rho(u^*Tu)=\rho(T)$ for all $T \in \bofh$ and unitaries $u \in \cA$;
\item
There is a $*$-homomorphism $\pi:\cA \to \cR^{\cU}$ along with a completely positive contractive lift $\varphi:\cA \to \ell^{\infty}(\cR)$ such that $\tr_{\cR^{\cU}} \circ \pi=\tau$;
\item
There is a sequence of natural numbers $N(k)$ and completely positive contractive maps $\varphi_k:\cA \to M_{N(k)}$ satisfying
$$\lim_{k \to \infty} \|\varphi_k(ab)-\varphi_k(a)\varphi_k(b)\|_2=0 \text{ and } \lim_{k \to \infty} |\tr_{N(k)}(\varphi_k(a))-\tau(a)|=0$$
for all $a,b \in \cA$;
\item
The linear functional $\gamma:\cA \otimes \cA^{op} \to \bC$ given by $\gamma(a \otimes b^{op})=\tau(ab)$ extends to a continuous linear map on $\cA \otimes_{\min} \cA^{op}$. 
\end{enumerate}
\end{theorem} 

As pointed out in \cite{KPS18}, as soon as $\gamma$ is continuous on $\cA \otimes_{\min} \cA^{op}$ in condition (4) above, it is automatically a state on the minimal tensor product.

In what follows, we will let $\| \cdot \|_2$ denote the $2$-norm with respect to the trace. For the convenience of the reader, we recall the following perturbation result.

\begin{lemma}
\emph{(Kim-Paulsen-Schafhauser, \cite{KPS18})}
\label{lemma: the KPS lemma}
Let $\ee>0$ and $c \in \bN$. Then there exists a $\delta>0$ such that, if $n,N \in \bN$ and $P_1,...,P_c \in M_n(M_N)$ are positive contractions with $\|P_aP_b\|_2<\delta$ for all $a \neq b$ and $\|P_a^2-P_a\|_2<\delta$ for all $a$, then there exist orthogonal projections $Q_1,...,Q_c \in M_n(M_N)$ with $Q_aQ_b=0$ for $a \neq b$ and $\|P_a-Q_a\|_2<\ee$. Moreover, if $\left\| \sum_{a=1}^c P_a-I_n \otimes I_N \right\|_2<\delta$, then we may arrange for the projections $Q_1,...,Q_c$ to satisfy $\sum_{a=1}^c Q_a=I_n \otimes I_N$.
\end{lemma}

Note that this lemma is stated slightly differently than in \cite{KPS18}; however, it is easy to see that their result is equivalent to the above result. Notice that, in the above lemma, if we write $P_a=(P_{a,ij}) \in M_n(M_N)$ and $Q_a=(Q_{a,ij}) \in M_n(M_N)$, then one has $\|Q_{a,ij}-P_{a,ij}\|_2 \leq \|Q_a-P_a\|_2<\ee$, where the first $2$-norm is in $M_N$ and the second $2$-norm is in $M_n(M_N)$.

\begin{theorem}
\label{theorem: characterizing synchronous qa}
Let $X=(X^{(a,b)}_{(i,j),(k,\ell)})$ be an element of $(M_n \otimes M_n)^{c^2}$. The following statements are equivalent:
\begin{enumerate}
\item
$X$ belongs to $\mathcal{Q}_{qa}^s(n,c)$;
\item
$X$ belongs to $\overline{\mathcal{Q}_q^s(n,c)}$;
\item
There is a separable unital $C^*$-algebra $\cA$, a PVM $\{P_1,...,P_c\}$ in $M_n(\cA)$, and an amenable trace $\tau$ on $\cA$ such that, for all $1 \leq i,j,k,\ell \leq n$ and $1 \leq a,b \leq c$,
$$X^{(a,b)}_{(i,j),(k,\ell)}=\tau(P_{a,ij}P_{b,k\ell}^*);$$
\item
There are elements $q_{a,ij}$ in $\cR^{\cU}$ such that $q_a=(q_{a,ij})$ are projections in $M_n(\cR^{\cU})$ with $\sum_{a=1}^c q_a=I_n$ and
$$X^{(a,b)}_{(i,j),(k,\ell)}=\tr_{\cR^{\cU}}(q_{a,ij}q_{b,k\ell}^*).$$
\end{enumerate}
\end{theorem}

\begin{proof}
First, we show that (1) implies (3). Since $\mathcal{Q}_{qa}(n,c)$ is the closure of $\mathcal{Q}_q(n,c)$, this means that there are correlations $X_r=(X^{(a,b)}_{r,(i,j),(k,\ell)}) \in \mathcal{Q}_q(n,c)$ with $\displaystyle\lim_{r \to \infty} X_r=X$ pointwise. We may choose natural numbers $N(r)$ and $M(r)$ along with projection-valued measures $\{P_1^{(r)},...,P_c^{(r)}\} \subseteq M_n(M_{N(r)})$ and $\{Q_1^{(r)},...,Q_c^{(r)}\} \subseteq M_n(M_{M(r)})$ and unit vectors $\chi_r \in \bC^{N(r)} \otimes \bC^{M(r)}$ such that, for all $r \in \bN$ and for all $1 \leq a,b \leq c$ and $1 \leq i,j,k,\ell \leq n$,
$$X^{(a,b)}_{r,(i,j),(k,\ell)}=\la (P_{a,ij}^{(r)} \otimes Q_{b,k\ell}^{(r)})\chi_r,\chi_r \ra.$$
Then for each $r$, by Theorem \ref{theorem: characterizing qa} and Theorem \ref{theorem: synchronicity condition}, there is a state $\varphi_r$ on $\cP_{n,c} \otimes_{\min} \cP_{n,c}^{op}$ satisfying
$$\varphi_r(p_{a,ij} \otimes p_{b,k\ell}^{op})=\la (P^{(r)}_{a,ij} \otimes Q^{(r)}_{b,k \ell}) \chi_r,\chi_r \ra.$$
As the state space of any unital $C^*$-algebra is $w^*$-compact, we may take a $w^*$-limit point $\varphi$ of the sequence $(\varphi_r)_{r=1}^{\infty}$. By construction, we note that $\varphi(p_{a,ij} \otimes p_{b,k\ell}^{op})=X^{(a,b)}_{(i,j),(k,\ell)}$ for all $1 \leq a,b \leq c$ and $1 \leq i,j,k,\ell \leq n$. We write $\varphi=\la \pi(\cdot)\chi,\chi \ra$ in its GNS representation, where $\pi:\cP_{n,c} \otimes_{\min} \cP_{n,c}^{op} \to \bofh$ is a unital $*$-homomorphism and $\chi \in \cH$ is a unit vector. Applying Theorem \ref{theorem: synchronicity condition} and restricting to $\cP_{n,c}$, we see that $\tau(y):=\la \pi(y \otimes 1)\chi,\chi \ra$ defines a trace on $\cP_{n,c}$. To establish (3), we need to show that $\tau(x \otimes y^{op})=\tau(xy)$ for all $x,y \in \cP_{n,c}$. Notice that for each $a$ and $i,j$, by Theorem \ref{theorem: synchronicity condition} we have
$$\pi(x \otimes p_{a,ij}^{op})\chi= \pi(x \otimes 1^{op})(\pi(1 \otimes p_{a,ij}^{op}))\chi= \pi(xp_{a,ij} \otimes 1^{op})\chi.$$
Then for each $b,k,\ell$, we have
\begin{align*}
\pi(x \otimes p_{a,ij}^{op}p_{b,k\ell}^{op})\chi&=\pi(x \otimes p_{a,ij}^{op})\pi(1 \otimes p_{b,k\ell}^{op})\chi \\
&=\pi(x \otimes p_{a,ij}^{op})\pi(p_{b,k\ell} \otimes 1)\chi \\
&=\pi(xp_{b,k\ell} \otimes p_{a,ij}^{op})\chi \\
&=\pi(xp_{b,k\ell}p_{a,ij} \otimes 1)\chi.
\end{align*}
Since $p_{a,ij}^{op}p_{b,k\ell}^{op}=(p_{b,k\ell}p_{a,ij})^{op}$ and the elements $p_{a,ij}$ generate $\cP_{n,c}$, one can see that $\pi(x \otimes y^{op})\chi=\pi(xy \otimes 1)\chi$. Therefore, $\varphi(x \otimes y^{op})=\tau(xy)$ for all $x,y \in \cP_{n,c}$, showing that $\tau$ is an amenable trace. Setting $P_{a,ij}=\pi(p_{a,ij} \otimes 1)$, we obtain a PVM $\{P_1,...,P_c\}$ in $M_n(\cA)$, where $\cA=\pi(\cP_{n,c} \otimes 1)$ and $P_a=(P_{a,ij})$. Moreover, $X_{(i,j),(k,\ell)}^{(a,b)}=\tau(P_{a,ij}P_{b,k\ell}^*)$, so (1) implies (3).

Next, we show that (3) implies (2). Let $\{P_1,...,P_c\}$ be a PVM in $M_n(\cA)$ for a separable unital $C^*$-algebra, and let $\tau$ be an amenable trace on $\cA$ such that $X_{(i,j),(k,\ell)}^{(a,b)}=\tau(P_{a,ij}P_{b,k\ell}^*)$ for all $a,b,i,j,k,\ell$. By Theorem \ref{theorem: amenable traces}, we may choose a sequence of completely positive contractive maps $\varphi_r:\cA \to M_{N(r)}$ with $\lim_{r \to \infty} \|\varphi_r(xy)-\varphi_r(x)\varphi_r(y)\|_2=0$ and $\lim_{r \to \infty} |\tr_{N(r)}(\varphi_r(x))-\tau(x)|=0$ for all $x,y \in \cA$. Define $p^{(r)}_{a,ij}=\varphi_r(P_{a,ij})$ and set $p^{(r)}_a=(p^{(r)}_{a,ij}) \in M_n(M_{N(r)})$. Using the $2$-norm on $M_n(M_{N(r)})$ and the fact that $P_a^2=P_a$ implies $p^{(r)}_{a,ij}=\sum_{k=1}^n \varphi_r(P_{a,ik}P_{a,kj})$, one sees that
\begin{align*}
\|(p_a^{(r)})^2-p_a^{(r)}\|_2&=\left\| \left( \sum_{k=1}^n p^{(r)}_{a,ik}p^{(r)}_{a,kj} - p^{(r)}_{a,ij}\right)_{i,j}\right\|_2 \\
&=\left\| \left(\sum_{k=1}^n( \varphi_r(P_{a,ik})\varphi_r(P_{a,kj})-\varphi_r(P_{a,ik}P_{a,kj}))\right)\right\|_2 \\
&\leq \sum_{i,j,k=1}^n \|\varphi_r(P_{a,ik})\varphi_r(P_{a,kj})-\varphi_r(P_{a,ik}P_{a,kj})\|_2 \xrightarrow{r \to \infty} 0.
\end{align*}
Similarly, one can show that $\lim_{r \to \infty} \|p_a^{(r)}p_b^{(r)}\|_2=0$ whenever $a \neq b$ and $\left\|\sum_{a=1}^c p_a^{(r)}-1_n \right\|_2 \to 0$. Applying Lemma \ref{lemma: the KPS lemma} and dropping to a subsequence if necessary, we obtain a sequence of PVMs $\{q_1^{(r)},...,q_c^{(r)}\} \subseteq M_n(M_{N(r)})$ with $c$ outputs with $\|p_{a,ij}^{(r)}-q_{a,ij}^{(r)}\|_2 \xrightarrow{r \to \infty} 0$ for all $a,i,j$.
There is a unital $*$-homomorphism $\psi_r:\cP_{n,c} \to M_{N(r)}$ with $\psi_r(p_{a,ij})=q_{a,ij}^{(r)}$. As $\cP_{n,c}$ is generated by $\{p_{a,ij}\}_{a,i,j}$, a standard argument shows that $\displaystyle\lim_{r \to \infty} \|\varphi_r(x)-\psi_r(x)\|_2=0$ for every $x \in \cP_{n,c}$. This implies that $|\tr_{N(r)}(\varphi_r(x))-\tr_{N(r)}(\psi_r(x))| \to 0$, so that $\displaystyle\lim_{r \to \infty} |\tr_{N(r)}(\psi_r(x))-\tau(x)|=0$. Hence,
$\displaystyle\lim_{r \to \infty} \tr_{N(r)}(q_{a,ij}^{(r)} (q_{b,k\ell}^{(r)})^*)=\tau(P_{a,ij}P_{b,k\ell}^*)$ for all $a,b,i,j,k,\ell$. As each correlation $X^{(a,b)}_{r,(i,j),(k,\ell)}=\tr_{N(r)}(q_{a,ij}^{(r)}(q_{b,k\ell}^{(r)})^*)$ defines an element of $\mathcal{Q}_q^s(n,c)$, we see that $X=(\tau(P_{a,ij}P_{b,k\ell}^*))$ belongs to $\overline{\mathcal{Q}_q^s(n,c)}$. Hence, (3) implies (2).

Since $\mathcal{Q}_{qa}(n,c)$ is the closure of $\mathcal{Q}_q(n,c)$, it is easy to see that (2) implies (1).

To show that (3) implies (4), we use Theorem \ref{theorem: amenable traces}. If $\{P_1,...,P_c\}$ is a PVM in $M_n(\cA)$ and $\tau$ is an amenable trace on $\cA$ satisfying $X^{(a,b)}_{(i,j),(k,\ell)}=\tau(P_{a,ij}P_{b,k\ell}^*)$, then there is a unital $*$-homomorphism $\rho:\cA \to \cR^{\cU}$ that preserves $\tau$. Setting $q_{a,ij}=\rho(P_{a,ij})$, we obtain projections $q_a=(q_{a,ij}) \in M_n(\cR^{\cU})$ summing to the identity and satisfying
$$X^{(a,b)}_{(i,j),(k,\ell)}=\tr_{\cR^{\cU}}(q_{a,ij}q_{b,k\ell}^*),$$
which establishes (4).

Lastly, we prove that (4) implies (3). Given the elements $q_{a,ij}$ in (4), there is a unital $*$-homomorphism $\sigma:\cP_{n,c} \to \cR^{\cU}$ satisfying $\sigma(p_{a,ij})=q_{a,ij}$. By Theorem \ref{theorem: lifting property}, $\cP_{n,c}$ has the lifting property, so there is a ucp map $\zeta:\cP_{n,c} \to \ell^{\infty}(\cR)$ that is a lift of $\sigma$. Then Theorem \ref{theorem: amenable traces} shows that $\tau:=\tr_{\cR^{\cU}} \circ \sigma$ is an amenable trace on $\cP_{n,c}$. Since $\tau(p_{a,ij}p_{b,k\ell}^*)=\tr_{\cR^{\cU}}(q_{a,ij}q_{b,k\ell}^*)=X^{(a,b)}_{(i,j),(k,\ell)}$, we see that (3) holds.
\end{proof}

\section{The game for quantum-to-classical graph homomorphisms} \label{sec: game}

In this section, we define the quantum-to-classical game for quantum-classical graph homomorphisms. Throughout our discussion, we use the bimodule perspective of quantum graphs considered by N. Weaver \cite{weaver1,weaver2} (which is a direct generalization of the non-commutative graphs considered by R. Duan, S. Severini and A. Winter in \cite{DSW13}, and D. Stahlke in \cite{Sta16}). In addition, we will see later how our framework also relates nicely to other perspectives as well (e.g., the quantum adjacency matrix formalism of quantum graphs introduced by B. Musto, S. Reuter and D. Verdon in \cite{MRV18}).

For our purposes, we refer to a \textbf{quantum graph} as a triple $(\cS,\cM,M_n)$, where
\begin{itemize}
\item
$\cM$ is a (non-degenerate) von Neumann algebra and $\cM \subseteq M_n$;
\item
$\cS \subseteq M_n$ is an operator system; and
\item
$\cS$ is an $\cM'$-$\cM'$-bimodule with respect to matrix multiplication.
\end{itemize}

In our discussion below, one can just as well use the ``traceless" version of quantum graphs along the lines of D. Stahlke \cite{Sta16}; i.e., one replaces the second condition with the condition that $\cS$ is a self-adjoint subspace of $M_n$ with $\Tr(X)=0$ for every $X \in \cS$. This condition, combined with the bimodule property, would force $\cS \subseteq (\cM')^{\perp}$. Our use of the operator system approach is generally cosmetic: one can easily adapt the quantum-classical game to traceless self-adjoint operator spaces that are $\cM'$-$\cM'$-bimodules with respect to matrix multiplication.

We begin by exhibiting a certain orthonormal basis for $\cS$ with respect to the (unnormalized) trace on $M_n$. It is from this (preferred) basis for $\cS$ that we will extract our input states for the homomorphism game.

\begin{proposition}
\label{proposition: quantum edge basis}
Let $\cK_1,...,\cK_m$ be non-zero subspaces of $\bC^n$ with $\cK_1 \oplus \cdots \oplus \cK_m=\bC^n$, such that $\cM$ acts irreducibly on each $\cK_r$. Let $E_r$ be the orthogonal projection of $\bC^n$ onto $\cK_r$, for each $1 \leq r \leq m$. Then there exists an orthonormal basis $\cF$ of $\cS \subseteq M_n$ with respect to the unnormalized trace, such that
\begin{itemize}
\item
$\frac{1}{\sqrt{\dim(\cK_r)}} E_r \in \cF$ for each $1 \leq r \leq m$;
\item
$\cF$ contains an orthonormal basis for $\cM'$; and
\item
For each $Y \in \cF$, there are unique $r,s$ with $E_r Y E_s=Y$.
\end{itemize}
\end{proposition}

\begin{proof}
Since $\cM$ acts irreducibly on $\cK_r$, it follows that $E_r \in \cM'$. Let $X$ be an element of $\cS$. As $\cS$ is an $\cM'$-$\cM'$-bimodule, it follows that $E_rXE_s \in \cS$ for all $1 \leq r,s \leq m$. Moreover, since $\sum_{r=1}^m E_r=1$, we have $X=\sum_{r,s=1}^m E_rXE_s$. Given $X,Y \in \cS$, we have
$\la E_rXE_s,E_pYE_q \ra=0$ whenever $r \neq p$ or $s \neq q$, where $\la \cdot,\cdot \ra$ is the inner product with respect to the unnormalized trace on $M_n$. We choose an orthonormal basis $\cF_{r,s}$ for $E_r\cS E_s$ with respect to this inner product as follows. We start with an orthonormal basis for $E_r \cM' E_s$; if $r=s$, then we arrange for this orthonormal basis  to contain $\frac{1}{\sqrt{\dim(\cK_r)}}E_r$. Then we extend the orthonormal basis for $E_r \cM' E_s$ to an orthonormal basis for $E_r \cS E_s$. We may do this since, if $X \in \cS \cap (\cM')^{\perp}$ and $Y \in \cM'$, then
$$\la E_rXE_s,Y \ra=\Tr(Y^*E_rXE_s)=\Tr(X E_sY^*E_r)=\la X,E_rYE_s \ra=0,$$
which shows that $E_r(\cS \cap (\cM')^{\perp})E_s \perp \cM'$. Then $\cF=\bigcup_{r,s} \cF_{r,s}$ is an orthonormal basis for $\cS$, which evidently satisfies all three properties.
\end{proof}

\begin{definition}
We call a basis for $\cS$ satisfying Proposition \ref{proposition: quantum edge basis} as a \textbf{quantum edge basis} for $(\cS,\cM,M_n)$.
\end{definition}

Alternatively, one could arrange for a quantum edge basis for $\cS$ to also contain a normalized system of matrix units for $\cM'$, since a quantum edge basis must already contain the normalized diagonal matrix units. We will see in Theorem \ref{theorem: qgraphhom winning conditions} that the game is independent of the quantum edge basis chosen.

Once an orthonormal basis for $\bC^n$ has been fixed, one can define the inputs for the game using the following well-known correspondence between vectors in $\bC^n \otimes \bC^n$ and matrices in $M_n$. With respect to a basis $\{v_1,...,v_n\}$, this correspondence is given by the assignment $v_i \otimes v_j \mapsto v_iv_j^*$, where $v_iv_j^*$ is the rank-one operator in $M_n$ such that $v_iv_j^*(x)=\la x,v_j \ra v_i$ for all $x \in \bC^n$.

\begin{proposition}
Let $(\cS,\cM, M_n)$ be a quantum graph with quantum edge basis $\cF$. Let $\{v_1,...,v_n\}$ be an orthonormal basis for $\bC^n$ that can be partitioned into bases for the subspaces $\cK_1,...,\cK_m$. For each $Y_{\alpha} \in \cF$, write $Y_{\alpha}=\sum_{p,q} y_{\alpha,pq} v_p v_q^*$ for $y_{\alpha,pq} \in \bC$. Then the set 
$$\left\{ \sum_{p,q} y_{\alpha,pq} v_p \otimes v_q\right\}_\alpha \subset \bC^n \otimes \bC^n$$
is orthonormal.
\end{proposition}

\begin{proof}
This result immediately follows from the fact that the correspondence $v_i \otimes v_j \mapsto v_iv_j^*$ preserves inner products, when using the canonical inner product on $\bC^n \otimes \bC^n$ and the (unnormalized) Hilbert-Schmidt inner product on $M_n$. 
\end{proof}

With the notion of quantum edge bases in hand, we now define the homomorphism game for the quantum graph $(\cS,\cM,M_n)$ and the classical graph $G$.

\begin{definition}
\label{definition: qtoc graphhom game}
Let $(\cS,\cM,M_n)$ be a quantum graph, and let $\{v_1,...,v_n\}$ be a basis for $\bC^n$ that can be partitioned into bases for the subspaces $\cK_1,...,\cK_r$. Let $G$ be a classical (undirected) graph on $c$ vertices, with no multiple edges and no loops. The \textbf{quantum-to-classical graph homomorphism game} for $((\cS,\cM,M_n),G)$, with respect to the basis $\{v_1,...,v_n\}$ and the quantum edge basis $\cF$, is defined as follows:
\begin{itemize}
\item
The inputs are of the form $\sum_{p,q} y_{\alpha,pq} v_p \otimes v_q$, where $Y_{\alpha}:=\sum_{p,q} y_{\alpha,pq} v_pv_q^*$ is an element of $\cF$. The outputs are vertices $a,b \in \{1,...,c\}=V(G)$. There are two rules to the game:
\item
(Adjacency rule) If $Y_{\alpha} \perp \cM'$, then Alice and Bob must respond with an edge in $G$; i.e., $a \sim b$.
\item
(Same ``vertex" rule) If $Y_{\alpha} \in \cM'$, then Alice and Bob must respond with the same vertex; i.e., $a=b$.
\end{itemize}
\end{definition}

Notice that the second rule will include a synchronicity condition: the inputs corresponding to $\frac{1}{\sqrt{\dim(\cK_r)}}E_r$ will arise in the second rule. We will see that the rule applied to these inputs will force Bob's projections to arise from Alice's projections; the rule applied to the other basis elements of $\cM'$ will be what forces the projections to live in $\cM \otimes \bofh$, rather than $M_n \otimes \bofh$.

 While the above definition of the game seems heavily basis-dependent, we will see that the existence of winning strategies in the various models is independent of the basis $\{v_1,...,v_n\}$, and independent of the quantum edge basis $\cF$ chosen for $(\cS,\cM,M_n)$. This will be a direct consequence of Theorem \ref{theorem: qgraphhom winning conditions}.

We would also like to relate winning strategies for the homomorphism game to the non-commutative graph homomorphisms in the sense of D. Stahlke \cite{Sta16}. For this, we first review Kraus operators in the infinite-dimensional case. Recall that a von Neumann algebra $\cN$ is \textbf{finite} if every isometry in $\cN$ is a unitary; i.e., $v^*v=1$ implies $vv^*=1$ in $\cN$. In this case, it is well-known that $\cN$ is equipped with a normal tracial state. We will be dealing with the case when $\cN$ is a finite von Neumann algebra equipped with a faithful normal trace $\tau$. One may always choose a faithful normal representation $\cN \subseteq \bofh$ such that $\tau(\cdot)=\la (\cdot)\chi,\chi \ra$ for some unit vector $\chi \in \cH$.

Suppose that $\cL \subseteq \cB(\cK)$ is another von Neumann algebra with faithful normal trace $\rho$. If $\Phi:\cL \to \cN$ is a normal ucp map, then $\Phi_*:\cN_* \to \cL_*$ is a CPTP map. In our context, $\cL$ will be a finite-dimensional von Neumann algebra, so a ucp map $\Phi:\cL \to \cN$ is automatically normal. One may choose $\cK$ to be finite-dimensional and extend $\Phi$ to a ucp map from $\cB(\cK)$ to $\bofh$, which is still (automatically) normal. Then one may choose Kraus operators $F_i$ such that $\Phi(\cdot)=\sum_{i=1}^m F_i^*(\cdot)F_i$, where $m$ is either finite or countably infinite. In the latter case, the sum converges in the $\text{SOT}^*$-topology. Then $\Phi_*:\cN_* \to \cL_*=\cL$ can be written as
$$\Phi_*(\cdot)=\sum_{i=1}^m F_i(\cdot)F_i^*.$$
The interested reader can consult \cite{CN13} (and the references therein) for more information on these topics.

Now, we address some of the basis dependence of the game before the main theorem. The next lemma shows that, up to a unitary conjugation, the basis for $\bC^n$ in Definition \ref{definition: qtoc graphhom game} does not matter.

\begin{lemma} Let $(\cS,\cM,M_n)$ be a quantum graph, and write $\bC^n=\cK_1 \oplus \cdots \oplus \cK_m$, where $\cM$ acts irreducibly on each $\cK_r$. let $G$ be a classical graph on $c$ vertices, and let $\{v_1,...,v_n\}$ be an orthonormal basis for $\bC^n$ that can be partitioned into bases for the subspaces $\cK_1,...,\cK_m$. Define $U \in M_n$ to be the unitary such that $Ue_i=v_i$ for all $i$, where $\{e_1,...,e_n\}$ is another orthonormal basis for $\bC^n$. Suppose that $X \in \mathcal{Q}_{qc}(n,c)$, and let $\{Y_{\alpha}\}_{\alpha}$ be a quantum edge basis for $(\cS,\cM,M_n)$. Then $X$ is a winning strategy for the homomorphism game for $((\cS,\cM,M_n),G)$ with respect to $\{Y_{\alpha}\}_{\alpha}$ if and only if $Z:=(U \otimes U)^*X(U \otimes U)$ is a winning strategy for the homomorphism game for $((U^*\cS U,U^* \cM U,M_n),G)$ with respect to the quantum edge basis $\{U^*Y_{\alpha}U\}_{\alpha}$.
\end{lemma}

\begin{proof}
Suppose that we can write $X=(\la (P_a \otimes I_n)(I_n \otimes Q_b)(e_j \otimes \chi \otimes e_{\ell}),e_i \otimes \chi \otimes e_k \ra)$, where $(\{P_a\}_{a=1}^c,\{Q_b\}_{b=1}^c,\chi)$ is a $qc$-strategy on a Hilbert space $\cH$. Then
$$\la (P_a \otimes I_n)(I_n \otimes Q_b)(v_j \otimes \chi \otimes v_{\ell}),v_i \otimes \chi \otimes v_k \ra=\la (U^*P_aU \otimes I_n)(I_n \otimes U^*Q_bU)(e_j \otimes \chi \otimes e_{\ell}),e_i \otimes \chi \otimes e_k\ra.$$
In other words, the element $Z=(Z^{(a,b)}):=((U \otimes U)^*X^{(a,b)}(U \otimes U))$ is a $qc$-correlation with respect to the basis $\{v_1,...,v_n\}$. It is not hard to see that, if $\cF$ is a quantum edge basis for $(\cS,\cM,M_n)$, then $U^* \cF U$ is a quantum edge basis for $(U^*\cS U,U^* \cM U,M_n)$, since $U^*\cM'U=(U^*\cM U)'$ and the Hilbert-Schmidt inner product is invariant under unitary conjugation.
Therefore, if $Y_{\alpha}=\sum_{p,q} y_{\alpha,pq} v_pv_q^*$ belongs to $\cF$, then its associated input vector is $\sum_{p,q} y_{\alpha,pq} v_p \otimes v_q$. Then $U^*Y_{\alpha}U=\sum_{p,q} y_{\alpha,pq} U^*v_pv_q^*U$ has associated input vector $\sum_{p,q} y_{\alpha,pq} U^*v_p \otimes U^*v_q=\sum_{p,q} y_{\alpha,pq} e_p \otimes e_q$.

Therefore, the probability of Alice and Bob outputting $(a,b)$ given the input vector $\sum_{p,q} y_{\alpha,pq} v_p \otimes v_q$, with respect to the correlation $X$, is the same as the probability of outputting $(a,b)$ given the input vector $\sum_{p,q} y_{\alpha,pq} e_p \otimes e_q$, with respect to the correlation $Z$. As this equality occurs for any element of the quantum edge basis $\cF$, the desired result follows. 
\end{proof}

\begin{remark}
The previous remark, along with the adjacency rule, forces any winning strategy to be synchronous with respect to the basis $\{v_1,...,v_n\}$. Thus, in our main theorem, we may assume that we are dealing with a synchronous $t$-strategy $(\{P_a\}_{a=1}^c,\chi)$, where $\{P_a\}_{a=1}^c$ is a PVM and $\chi$ is a faithful normal tracial state on the von Neumann algebra generated by the entries of $\{P_a\}_{a=1}^c$. Note that conjugating $\{P_a\}_{a=1}^c$ by a unitary in $M_n$ does not change the von Neumann algebra generated by the entries of the operators $P_a$.
\end{remark}

\begin{theorem}
\label{theorem: qgraphhom winning conditions}
Let $(\cS,\cM,M_n)$ be a quantum graph, let $G$ be a classical graph on $c$ vertices, and let $t \in \{loc,q,qa,qc\}$. Let $\cN \subseteq \bofh$ be a (non-degenerate) finite von Neumann algebra, and $\chi \in \cH$ be a unit vector such that $\tau=\la (\cdot) \chi,\chi \ra$ is a faithful (normal) trace on $\cN$. The following are equivalent:
\begin{enumerate}
\item
There is a winning strategy $(\{P_a\}_{a=1}^c,\chi)$ from $\cN$ for the homomorphism game for $((\cS,\cM,M_n),G)$ with respect to any quantum edge basis.
\item
There is a winning strategy $(\{P_a\}_{a=1}^c,\chi)$ from $\cN$ for the homomorphism game for $((\cS,\cM,M_n),G)$ with respect to some quantum edge basis.
\item
There is a PVM $\{P_a\}_{a=1}^c$ in $\cM \otimes \cN$ satisfying the following: if $1 \leq a,b \leq c$ and $a \not\sim b$ in $G$, then
$$P_a((\cS \cap (\cM')^{\perp}) \otimes 1)P_b=0.$$
\item
There is a CPTP map $\Phi:\cM \otimes \cN_* \to D_c$ of the form $\Phi(\cdot)=\sum_{i=1}^m F_i(\cdot)F_i^*$ such that 
$$F_i((\cS \cap (\cM')^{\perp}) \otimes 1_{\cN})F_j^* \subseteq \cS_G \cap (D_c)^{\perp} \text{ for all i,j},$$
and
$$F_i(\cM' \otimes 1_{\cN})F_j^* \subseteq D_c \text{ for all } i,j.$$
\end{enumerate}
\end{theorem}

\begin{proof}
Clearly (1) implies (2). We will show that $(2) \implies (3) \implies (4) \implies (1)$. Let $\{v_1,...,v_n\}$ be an orthonormal basis for $\bC^n$. Let $U$ be the unitary such that $Ue_i=v_i$ for all $i$. Suppose that we can establish (3) for the PVM $\{(U \otimes 1_{\cN})^*P_a(U \otimes 1_{\cN})\}_{a=1}^c$ and the quantum graph $(U^*\cS U,U^*\cM U,M_n)$. Using the fact that $(U^*\cM U)'=U^*\cM' U$, the condition in (3) can be written as
$$(U \otimes 1_{\cN})^*P_a (U \otimes 1_{\cN})((U^* \cS U) \cap (U^* \cM' U)^{\perp} \otimes 1_{\cN})(U \otimes 1_{\cN})^*P_b(U\otimes 1_{\cN})=0 \text{ if } a \not\sim b.$$
It is not hard to see that $(U^* \cM' U)^{\perp}=U^*(\cM')^{\perp} U$, so that the above reduces to
$$(U \otimes 1_{\cN})^*P_a((\cS \cap (\cM')^{\perp}) \otimes 1_{\cN})P_b(U \otimes 1_{\cN})=0.$$
Since $U$ is a unitary, we obtain the desired condition for $\{P_a\}_{a=1}^c$ with respect to the quantum graph $(\cS,\cM,M_n)$. Hence, we may assume without loss of generality that $v_i=e_i$ for all $i$.

Then, given a matrix $Y=\sum_{p,q} y_{pq} v_pv_q^*$ with associated unit vector $y=\sum_{p,q} y_{pq} v_p \otimes v_q$, the probability of Alice and Bob outputting $a$ and $b$ respectively, given $y$ and using the synchronous strategy $(\{P_a\}_{a=1}^c,\chi)$, is
\begin{align}
p(a,b|y)&=\left\la (P_{a,ij}P_{b,k\ell}^*)_{(i,j),(k,\ell)}\left( \sum_{p,q} y_{pq} v_p \otimes \chi \otimes v_q\right),\sum_{r,s} y_{rs} v_r \otimes \chi \otimes v_s \right\ra \nonumber \\
&=\left\la \sum_{i,j,k,\ell} v_i \otimes P_{a,ij}y_{j\ell} P_{b,k\ell}^*\chi \otimes v_k,\sum_{r,s} y_{rs} v_r \otimes \chi \otimes v_s \right\ra \nonumber \\
&=\sum_{i,j,k,\ell} \left\la P_{a,ij}y_{j\ell}P_{b,k\ell}^* \overline{y}_{ik} \chi,\chi \right\ra \nonumber \\
&=\sum_{i,j,k,\ell} \tau( P_{a,ij}y_{j\ell}P_{b,\ell k} \overline{y}_{ik} ) \nonumber \\
&=\Tr \otimes \tau\left(\left(\sum_{j,\ell} P_{a,ij} y_{j\ell} P_{b,\ell k}\right)_{i,k}(Y^* \otimes 1_{\cN}) \right) \nonumber \\
&=\Tr \otimes \tau(P_a(Y \otimes 1_{\cN})P_b(Y^* \otimes 1_{\cN})) \nonumber \\
&=\Tr \otimes \tau(P_a(Y \otimes 1_{\cN})P_b(Y^* \otimes 1_{\cN})P_a), \label{equation: qgraphhom winning condition}
\end{align}
where we have used the fact that $P_a$ is an orthogonal projection. Now, suppose that $\cF=\{Y_{\alpha}\}_{\alpha}$ is a quantum edge basis for $(\cS,\cM,M_n)$, and suppose that $(\{P_a\}_{a=1}^c,\chi)$ is a winning strategy with respect to this quantum edge basis. If $Y_{\alpha} \in \cM'$, then Equation \ref{equation: qgraphhom winning condition} and faithfulness of the trace gives $P_a(Y_{\alpha} \otimes 1_{\cN})P_b=0$ whenever $a \neq b$. Then 
$$P_a(Y_{\alpha} \otimes 1_{\cN})P_a=\sum_{b=1}^c P_a(Y_{\alpha} \otimes 1_{\cN})P_b=P_a(Y_{\alpha} \otimes 1_{\cN})\left( \sum_{b=1}^c P_b \right)=P_a(Y_{\alpha} \otimes 1_{\cN}).$$
Similarly, $P_a(Y_{\alpha} \otimes 1_{\cN})P_a=(Y_{\alpha} \otimes 1_{\cN})P_a$. Hence, $P_a$ commutes with $Y_{\alpha} \otimes 1_{\cN}$ whenever $Y_{\alpha} \in \cM'$. This shows that $P_a \in (\cM' \otimes 1_{\cN})' \cap (M_n \otimes \cN)=\cM \otimes \bofh \cap (M_n \otimes \cN)=\cM \otimes \cN$.

Similarly, if $Y_{\alpha} \perp \cM'$, then the rules of the game and the faithfulness of the trace force $P_a(Y_{\alpha} \otimes 1_{\cN})P_b=0$ whenever $a \not\sim b$, which shows that (3) holds.

Now we show that (3) implies (4). If (3) holds, then there is a projection-valued measure $\{P_a\}_{a=1}^c$ in $\cM \otimes \cN$ such that $P_a(Y \otimes 1_{\cN})P_b=0$ for all $Y \in \cS \cap (\cM')^{\perp}$ and $a \not\sim b$. Then the map $\Psi:D_c \to \cM \otimes \cN$ given by $\Psi(E_{kk})=P_k$ is a unital $*$-homomorphism. Since $D_c$ is finite-dimensional, $\Psi$ is normal. Hence, we may find Kraus operators $F_1,F_2,...$ in $\cB(\bC^n \otimes \cH,\bC^c)$ such that
$$\Psi(\cdot)=\sum_{i=1}^m F_i^* (\cdot) F_i,$$
where $m$ is either finite or $\aleph_0$. In the infinite case, these sums converge in the $\text{SOT}^*$-topology. Then $\Psi=\Theta^*$ for a CPTP map $\Theta:\cM_* \otimes \cN_*=\cM \otimes \cN_* \to D_c$ given by
$$\Theta(\cdot)=\sum_{i=1}^m F_i (\cdot) F_i^*.$$ Given $Y \in \cS$, we set $Z_{a,b,i,j}=E_{aa}F_i(Y \otimes 1_{\cN})F_j^*E_{bb}$. 
Notice that
\[ Z_{a,b,i,j}Z_{a,b,i,j}^*=E_{aa}F_i(Y \otimes 1_{\cN})F_j^*E_{bb}F_j(Y^* \otimes 1_{\cN})F_i^*E_{aa},\]
so summing over $j$ (for fixed $i$, this sum will converge in the SOT*-topology) and using the fact that $\sum_{i=1}^m F_j^*E_{bb}F_j=\Psi(E_{bb})=P_b$,
\[ \sum_{j=1}^m Z_{a,b,i,j}Z_{a,b,i,j}^*=E_{aa}F_i(Y \otimes 1_{\cN})P_b(Y^* \otimes 1_{\cN})F_i^*E_{aa}.\]
Now set $W_{a,b,i}=E_{aa}F_i(Y \otimes 1_{\cN})P_b$. Then $\sum_{j=1}^m Z_{a,b,i,j}Z_{a,b,i,j}^*=W_{a,b,i}W_{a,b,i}^*$, since $P_b$ is a projection. On the other hand,
\[ W_{a,b,i}^*W_{a,b,i}=P_b(Y^* \otimes 1_{\cN})F_i^*E_{aa}F_i(Y \otimes 1_{\cN})P_b,\]
so summing over $i$ gives
\[ \sum_{i=1}^m W_{a,b,i}^*W_{a,b,i}=P_b(Y^* \otimes 1_{\cN})P_a(Y \otimes 1_{\cN})P_b=(P_a (Y \otimes 1_{\cN})P_b)^*(P_a(Y \otimes 1_{\cN})P_b).\]
It follows that, if the latter quantity is zero, then $Z_{a,b,i,j}=0$ for all $i,j$. By condition (3), if $a \not\sim b$ in $G$ and $Y \in \cS \cap (\cM')^{\perp}$, then $P_a(Y \otimes 1_{\cN})P_b=0$. This immediately implies that $E_{aa}F_i(Y \otimes 1_{\cN})F_j^*E_{bb}=0$. Then
\[ F_i(Y \otimes 1_{\cN})F_j^*=\sum_{a,b} E_{aa}F_i(Y \otimes 1_{\cN})F_j^*E_{bb}=\sum_{a \sim b} E_{aa}F_i(Y \otimes 1_{\cN})F_j^*E_{bb} \in \cS_G \cap D_c^{\perp}.\]
Since each $P_a$ belongs to $\cM \otimes \cN$, $P_a$ commutes with $\cM' \otimes 1_{\cN}$. Therefore, $P_a(Y \otimes 1_{\cN})P_b=0$ for all $a \neq b$ and $Y \in \cM'$. A consideration of the above equations, yields $E_{aa}F_i(Y \otimes 1_{\cN})F_j^*E_{bb}=0$ whenever $Y \in \cM'$ and $a \neq b$. In that case, we have
\[ F_i(Y \otimes 1_{\cN})F_j^*=\sum_{a,b} E_{aa}F_i(Y \otimes 1_{\cN})F_j^*E_{bb}=\sum_a E_{aa}F_i(Y \otimes 1_{\cN})F_j^*E_{aa} \in D_c,\]
which yields the second part of condition (4). Hence, (3) implies (4).

Lastly, suppose that (4) holds; we will obtain a winning strategy for the game. Suppose that $\Phi:\cM_* \otimes \cN_* \to D_c$ is a CPTP map of the form $\Phi(\cdot)=\sum_{i=1}^m R_i(\cdot)R_i^*$, such that $R_i(Y \otimes 1_{\cN})R_j^* \in \cS_G \cap D_c^{\perp}$ for all $1 \leq i,j \leq m$ and $Y \in \cS \cap (\cM')^{\perp}$, and $R_i(Y \otimes 1_{\cN})R_j^* \in D_c$ for all $Y \in \cM'$. Then $\Phi^*(\cdot)=\sum_{i=1}^m R_i^*(\cdot)R_i$ defines a normal ucp map from $D_c$ to $\cM \otimes \cN$. Let $P_a=\Phi^*(E_{aa})=\sum_{i=1}^m R_i^*E_{aa}R_i$ for each $1 \leq a \leq c$. Since $\Phi^*$ is UCP, $\{P_a\}_{a=1}^c$ is a POVM in $\cM \otimes \cN$. By considering the unitary $U$ sending $e_i$ to $v_i$ for each $i$, along with the POVM $\{U^*P_aU\}_{a=1}^c$, the quantum graph $(U^*\cS U,U^*\cM U,M_n)$ and the operators $R_iU$ if necessary, we may assume without loss of generality that $v_i=e_i$ for all $i$. We will show that $X^{(a,b)}_{(i,j),(k,\ell)}=(\tau(P_{a,ij}P_{b,k\ell}^*))$ defines a winning $t$-strategy for the quantum graph homomorphism game for $((\cS,\cM,M_n),G)$. 

For $1 \leq a,b \leq c$, $1 \leq i,j \leq m$ and $Y \in \cS$, we define $V_{a,b,i,j}=E_{aa}R_i(Y \otimes 1_{\cN})R_j^*E_{bb}$. Then
\[ \sum_{j=1}^m V_{a,b,i,j}V_{a,b,i,j}^*=\sum_{j=1}^m E_{aa}R_i(Y \otimes 1_{\cN})R_j^*E_{bb}R_j(Y^* \otimes 1_{\cN})R_i^*E_{aa}=E_{aa}R_i(Y \otimes 1_{\cN})P_b(Y^* \otimes 1_{\cN})R_i^*E_{aa},\]
since $P_b=\Phi^*(E_{bb})$. Therefore, $\sum_{j=1}^m V_{a,b,i,j}V_{a,b,i,j}^*=T_{a,b,i}^*T_{a,b,i}$ where $T_{a,b,i}=P_b^{\frac{1}{2}}(Y^* \otimes 1_{\cN})R_i^*E_{aa}$. Next, we examine the sum
\[\sum_{i=1}^m T_{a,b,i}T_{a,b,i}^*=\sum_{i=1}^m P_b^{\frac{1}{2}}(Y^* \otimes 1_{\cN})R_i^*E_{aa}R_i(Y \otimes 1_{\cN})P_b^{\frac{1}{2}}=P_b^{\frac{1}{2}}(Y^* \otimes 1_{\cN})P_a(Y \otimes 1_{\cN})P_b^{\frac{1}{2}}.\]
In the case when $Y \in \cM'$, we have $V_{a,b,i,j}=0$ whenever $a \neq b$, which implies that $P_b^{\frac{1}{2}}(Y^* \otimes 1_{\cN})P_a(Y \otimes 1_{\cN})P_b^{\frac{1}{2}}=0$. It follows that $P_a^{\frac{1}{2}}(Y \otimes 1_{\cN})P_b^{\frac{1}{2}}=0$. Multiplying on the left by $P_a^{\frac{1}{2}}$ and on the right by $P_b^{\frac{1}{2}}$, we obtain $P_a(Y \otimes 1_{\cN})P_b=0$ whenever $Y \in \cM'$ and $a \neq b$. The case when $Y=1_{\cM}$ shows that $P_aP_b=0$ for $a \neq b$. Combining this orthogonality with the fact that $\{P_a\}_{a=1}^c$ is a POVM, we conclude that $\{P_a\}_{a=1}^c$ is a PVM.
Similarly, if $Y \in \cS \cap (\cM')^{\perp}$, then by condition (4), $V_{a,b,i,j}=0$ whenever $a \not\sim b$ in $G$. The same calculation shows that $P_a(Y \otimes 1_{\cN})P_b=0$ in this case as well.

Therefore, using Equation (\ref{equation: qgraphhom winning condition}), if $\{Y_{\alpha}\}_{\alpha}$ is a quantum edge basis for $(\cS,\cM,M_n)$, $Y_{\alpha}$ has associated unit vector $y_{\alpha}$ and $Y_{\alpha} \perp \cM'$, then by equation (\ref{equation: qgraphhom winning condition}),
$$p(a,b| y_{\alpha} )=\la P_a(Y_{\alpha} \otimes 1_{\cN})P_b,Y_{\alpha} \ra=0. \text{ if } a \not\sim b.$$
If $Y_{\alpha}$ belongs to $\cM$ with associated unit vector $y_{\alpha}$, then $p(a,b|y_{\alpha})=\la P_a(Y \otimes 1_{\cN})P_b,Y \otimes 1_{\cN}\ra=0$ as well. This shows that $(\{P_a\}_{a=1}^c,\chi)$ defines a winning strategy  for the homomorphism game for $(\cS,\cM,M_n)$ and $G$ with respect to any quantum edge basis, completing the proof.
\end{proof}

The next theorem will show that, in the loc model, condition (4) of Theorem \ref{theorem: qgraphhom winning conditions} is an analogue of Stahlke's notion of graph homomorphism from a non-commutative graph (i.e. a quantum graph with $\cM=M_n$) into a classical graph \cite{Sta16}, with an added assumption on the commutant of $\cM$. A similar analogue holds in the $q$-model, with natural generalizations to the $qa$ and $qc$ models. 

We observe that, if we start with a projection-valued measure $\{P_a\}_{a=1}^c$ whose block entries are in a tracial von Neumann algebra $(\cN,\tau)$, where $\tau$ is faithful and normal, then either all four conditions of Theorem \ref{theorem: qgraphhom winning conditions} are satisfied by the PVM, or none of the four conditions are satisfied. Notice that we needed to start with a PVM and a faithful trace for this to happen.

In the following discussion, we write $(\cS,\cM,M_n) \xrightarrow{t} G$ to mean that there is a winning $t$-strategy for the graph homomorphism game for $((\cS,\cM,M_n),G)$.We will also write $(\cS,\cM,M_n) \to G$ if $(\cS,\cM,M_n) \xrightarrow{loc} G$. Our choice of notation is since, from a loc-homomorphism, one can always obtain a graph homomorphism.

Using Theorem \ref{theorem: qgraphhom winning conditions} and the characterizations of synchronous correlations, we obtain the following theorem:

\begin{theorem}
\label{theorem: qgraphhom in terms of channel}
Let $(\cS,\cM,M_n)$ be a quantum graph and let $G$ be a classical graph on $c$ vertices.
\begin{enumerate}
\item
$(\cS,\cM,M_n) \xrightarrow{loc} G$ if and only if there is a CPTP map $\Phi:\cM \to D_c$ of the form $\Phi(\cdot)=\sum_{i=1}^m F_i(\cdot)F_i^*$ such that
$$F_i(\cS \cap (\cM')^{\perp})F_j^* \subseteq \cS_G \cap (D_c)^{\perp} \text{ for all } i,j,$$
and
$$F_i\cM'F_j^* \subseteq D_c \text{ for all } i,j.$$
\item
$(\cS,\cM,M_n) \xrightarrow{q} G$ if and only if there is $d \in \bN$ and a CPTP map $\Phi:\cM \otimes M_d \to D_c$ of the form $\Phi(\cdot)=\sum_{i=1}^m F_i(\cdot)F_i^*$ such that
$$F_i((\cS \cap (\cM')^{\perp}) \otimes I_d)F_j^* \subseteq \cS_G \cap (D_c)^{\perp} \text{ for all } i,j,$$
and
$$F_i(\cM' \otimes I_d)F_j^* \subseteq D_c \text{ for all } i,j.$$
\item
$(\cS,\cM,M_n) \xrightarrow{qa} G$ if and only if there is a CPTP map $\Phi:\cM \otimes (\cR^{\cU})_* \to D_c$ of the form $\Phi(\cdot)=\sum_{i=1}^m F_i(\cdot)F_i^*$ such that
$$F_i((\cS \cap (\cM')^{\perp}) \otimes 1_{\cR^{\cU}})F_j^* \subseteq \cS_G \cap (D_c)^{\perp} \text{ for all } i,j,$$
and
$$F_i(\cM' \otimes 1_{\cR^{\cU}})F_j^* \subseteq D_c \text{ for all } i,j.$$
\item
$(\cS,\cM,M_n) \xrightarrow{qc} G$ if and only if there is a von Neumann algebra $\cN$, a faithful normal trace $\tau$ on $\cN$, and a CPTP map $\Phi:\cM \otimes \cN_* \to D_c$ of the form $\Phi(\cdot)=\sum_{i=1}^m F_i(\cdot)F_i^*$ such that
$$F_i((\cS \cap (\cM')^{\perp}) \otimes 1_{\cN})F_j^* \subseteq \cS_G \cap (D_c)^{\perp} \text{ for all } i,j,$$
and
$$F_i(\cM' \otimes 1_{\cN})F_j^* \subseteq D_c \text{ for all } i,j.$$
\end{enumerate}
\end{theorem}

\begin{proof}

We consider the case $t=loc$ first. If $(\cS,\cM,M_n) \xrightarrow{loc} G$, then there is a winning $loc$-strategy for the homomorphism game from $(\cS,\cM,M_n)$ into $G$. Since $\mathcal{Q}_{loc}^s(n,c)$ is  convex and non-empty, one may obtain an extreme point in $\mathcal{Q}^s_{loc}(n,c)$ that wins the game with probability $1$. Applying Corollary \ref{corollary: characterizing synchronous loc}, there is a realization of this correlation using a PVM $\{P_a\}_{a=1}^c$ in $\mathcal{M} = \mathcal{M} \otimes \bC$. Then the result follows by condition (4) of Theorem \ref{theorem: qgraphhom winning conditions} with $\cN=\bC$. The converse of (1) holds by condition (3) of Theorem \ref{theorem: qgraphhom winning conditions}.

The argument is similar for $t=q$. Indeed, if there is a winning strategy for the homomorphism game in the $q$-model, then an application of Corollary \ref{corollary: characterizing synchronous q} shows that there is a winning $q$-strategy using an extreme point in $\mathcal{Q}_q^s(n,c)$, which can be realized using projections whose entries are in $M_d$, for some $d$. Then condition (4) of Theorem \ref{theorem: qgraphhom winning conditions} with $\cN=M_d$ yields the desired CPTP map. The converse, as before, holds by condition (3) of Theorem \ref{theorem: qgraphhom winning conditions}.

We note that (3) holds because of Theorem \ref{theorem: characterizing synchronous qa}. Condition (4) is achieved using the following well-known trick: if $\cA$ is a unital, separable $C^*$-algebra with tracial state $\tau$, and if $\pi_{\tau}:\cA \to \cB(\cH_{\tau})$ is the GNS representation of $\tau$ with cyclic vector $\xi$, then $\pi_{\tau}(\cA)''$ is a finite von Neumann algebra and $\la (\cdot)\xi,\xi \ra$ is a faithful normal trace on $\pi_{\tau}(\cA)''$. We leave the details to the reader.
\end{proof}

For synchronous games with classical inputs and classical outputs, J.W. Helton, K.P. Meyer, V.I. Paulsen and M. Satriano constructed a universal $*$-algebra for the game, generated by self-adjoint idempotents whose products were $0$ when the related pair of outputs was not allowed \cite{HMPS19}. One can define a game $*$-algebra in our context as follows.

\begin{definition}
Let $(\cS,\cM,M_n)$ be a quantum graph and let $G$ be a classical graph on $c$ vertices. The \textbf{game $*$-algebra} for the homomorphism game for $((\cS,\cM,M_n),G)$, denoted $\cA(\text{Hom}((\cS,\cM,M_n),G))$, is the universal $*$-algebra generated by entries $\{ p_{a,ij}: 1 \leq a \leq c, \, 1 \leq i,j \leq n\}$ subject to the relations
\begin{itemize}
\item
$p_a=(p_{a,ij})_{i,j}$ satisfies $p_a^2=p_a=p_a^*$ and $\sum_{a=1}^c p_a=I_n$, where $I_n$ is the $n \times n$ identity matrix;
\item
$p_a((\cS \cap (\cM')^{\perp}) \otimes 1)p_b=0$ for each $a \not\sim b$; and
\item
$p_a(\cM' \otimes 1)p_b=0$ for each $a \neq b$.
\end{itemize}
We say that the algebra \textbf{exists} if $1 \neq 0$ in the algebra.
\end{definition}

As one might expect, we obtain the following characterizations of the various flavors of winning strategies for the homomorphism game in terms of $\ast$-homomorphisms of the game algebra.

\begin{theorem}
Let $(\cS,\cM,M_n)$ be a quantum graph and let $G$ be a classical graph.
\begin{enumerate}
\item
$(\cS,\cM,M_n) \xrightarrow{loc} G$ $\iff$ there is a unital $*$-homomorphism $\cA(\text{Hom}((\cS,\cM,M_n),G)) \to \bC$.
\item
$(\cS,\cM,M_n) \xrightarrow{q} G$ if and only if there is a unital $*$-homomorphism $\cA(\text{Hom}((\cS,\cM,M_n),G)) \to M_d$ for some $d \in \bN$.
\item
$(\cS,\cM,M_n) \xrightarrow{qa} G$ if and only if there is a unital $*$-homomorphism $\cA(\text{Hom}((\cS,\cM,M_n),G)) \to \cR^{\cU}$.
\item
$(\cS,\cM,M_n) \xrightarrow{qc} G$ if and only if there is a unital $*$-homomorphism $\cA(\text{Hom}((\cS,\cM,M_n),G)) \to \cC$, where $\cC$ is a tracial $C^*$-algebra.
\item
$(\cS,\cM,M_n) \xrightarrow{alg} G$ if and only if $\cA(\text{Hom}((\cS,\cM,M_n),G)) \neq 0$.
\end{enumerate}
\end{theorem}

One can also define $C^*$-homomorphisms and hereditary homomorphisms of quantum graphs into classical graphs. We write $(\cS,\cM,M_n) \xrightarrow{C^*} G$ provided that there is a unital $*$-homomorphism \[\pi:\cA(\text{Hom}((\cS,\cM,M_n),G)) \to \bofh,\] for some Hilbert space $\cH$. (Equivalently, by the Gelfand-Naimark theorem, one may simply require that the game algebra has a representation into some unital $C^*$-algebra.)

For hereditary homomorphisms, we recall the concept of a hereditary (unital) $*$-algebra.  Recall that a unital $*$-algebra $\cA$ is said to be \textbf{hereditary} if, whenever $x_1,...,x_n \in \cA$ are such that $x_1^*x_1+\cdots+x_n^*x_n=0$, then $x_1=x_2=\cdots=x_n=0$. If one defines $\cA_+$ as the cone generated by all elements of the form $x^*x$ for $x \in \cA$, then $\cA$ being hereditary is equivalent to having $\cA_+ \cap (-\cA_+)=\{0\}$. Every unital $C^*$-algebra is hereditary as a unital $*$-algebra.

With this background in hand, we will write $(\cS,\cM,M_n) \xrightarrow{hered} G$ provided that there is a unital $*$-homomorphism from $\cA(\text{Hom}((\cS,\cM,M_n),G))$ into a (non-zero) hereditary unital $*$-algebra. One has the following sequence of implications for these types of homomorphism: 
\begin{align}
(\cS,\cM,M_n) \to G &\implies (\cS,\cM,M_n)\xrightarrow{q} G\implies (\cS,\cM,M_n)\xrightarrow{qa} G\implies (\cS,\cM,M_n)\xrightarrow{qc} G \nonumber\\
&\implies (\cS,\cM,M_n)\xrightarrow{C^*}G \implies (\cS,\cM,M_n)\xrightarrow{hered}G \implies (\cS,\cM,M_n)\xrightarrow{alg} G. \label{qgraphhom implications}
\end{align}

Our notions of homomorphisms reduce to the analogous types of homomorphisms for classical graphs in the case when $(\cS,\cM,M_n)$ is a classical graph. Recall \cite{dsw2013} that, for a classical graph $G$ on $n$ vertices, the graph operator system $\cS_G$ (or classical quantum graph) is defined as
\[ \cS_G=\spn (\{E_{ii}: 1 \leq i \leq n\} \cup \{E_{ij}: i \sim j \text{ in } G\}). \]
Note that $\cS_G$ is naturally a quantum graph in the previous sense if we regard it as a bimodule over the diagonal algebra $D_n = D_n' \subseteq M_n$.  In \cite{weaver1} it is shown that quantum graphs of the form $(\mathcal S, D_n, M_n)$ are in one-to-one correspondence with classical graphs on $n$ vertices. 

\begin{corollary}
\label{corollary: qgraphhom for classical graphs}
Let $G$ and $H$ be classical graphs on $n$ and $c$ vertices, respectively. Suppose that $t \in \{loc,q,qa,qc,C^*,hered,alg\}$. Then $G \xrightarrow{t} H$ if and only if $(\cS_G,D_n,M_n) \xrightarrow{t} H$.
\end{corollary}

\begin{proof}
We will show that the algebra $\cA(\text{Hom}(G,H))$ from \cite{HMPS19} is isomorphic to  $\cA(\text{Hom}((\cS_G,D_n,M_n),H))$. The former algebra is the universal unital $*$-algebra generated by self-adjoint idempotents $e_{x,a}$ such that $\sum_{x=1}^n e_{x,a}=1$, $e_{x,a}e_{x,b}=0$ if $a \neq b$, and $e_{x,a}e_{y,b}=0$ if $x \sim y$ in $G$ but $a \not\sim b$ in $H$. Since $D_n=D_n'$, the latter algebra is the universal unital $*$-algebra generated by elements $p_{a,ij}$ such that $p_a=(p_{a,ij}) \in M_n(\cA)$ is a self-adjoint idempotent with $\sum_{a=1}^c p_a=I_n$, $p_a((\cS_G \cap (D_n)^{\perp}) \otimes 1)p_b=0$ whenever $a \not\sim b$ in $H$, and $p_a(D_n \otimes 1)p_b=0$ whenever $a \neq b$.
Since $E_{ii} \in D_n$, using the fact that $p_a(D_n \otimes 1)p_b=0$ for $a \neq b$, we obtain
\[ p_a(E_{ii} \otimes 1)p_a=\sum_{b=1}^c p_a(E_{ii} \otimes 1)p_b=p_a(E_{ii} \otimes 1).\]
Similarly, $p_a(E_{ii} \otimes 1)p_a=(E_{ii} \otimes 1)p_a$, so that $E_{ii} \otimes 1$ commutes with $p_a$. It follows that $p_{a,ij}=0$ whenever $i \neq j$. Since $p_a^2=p_a=p_a^*$, we see that $p_{a,ii}^2=p_{a,ii}=p_{a,ii}^*$. For $1 \leq a \leq c$ and $1 \leq x \leq n$, we define $q_{x,a}=p_{a,xx}$. Then $q_{x,a}$ is a self-adjoint idempotent and $\sum_{a=1}^c q_{x,a}=1$ for all $1 \leq x \leq n$. Note that, if $x \sim y$ in $G$ but $a \not\sim b$ in $H$, then
\[ q_{x,a}q_{y,b}=p_{a,xx}p_{b,yy}=p_a(E_{xy} \otimes 1)p_b=0,\]
since $E_{xy} \in \cS_G \cap (D_n)^{\perp}$ and $a \not\sim b$ in $H$. Similarly, if $a \neq b$, then $q_{x,a}q_{x,b}=p_a(E_{xx} \otimes 1)p_b=0$ since $E_{xx} \in D_n$. By the universal property of $\cA(\text{Hom}(G,H))$, there is a unital $*$-homomorphism $\pi:\cA(\text{Hom}(G,H)) \to \cA(\text{Hom}((\cS_G,D_n,M_n),H))$ such that $\pi(e_{x,a})=q_{x,a}$ for all $x,a$.

Conversely, in $\cA(\text{Hom}(G,H))$, one can construct the $n \times n$ matrices $f_a=(f_{a,ij})$ with $f_{a,ij}=0$ for $i \neq j$ and $f_{a,ii}=e_{a,i}$. Then evidently $f_a^2=f_a=f_a^*$ and $\sum_{a=1}^c f_a=I_n$. Since $e_{x,a}e_{x,b}=0$ for $a \neq b$, we see that $f_a(E_{xx} \otimes 1)f_b=0$ if $a \neq b$. Since $D_n=\spn \{E_{xx}: 1 \leq x \leq n\}$, it follows that $f_a(D_n \otimes 1)f_b=0$ for $a \neq b$. Similarly, it is not hard to see that $f_a(E_{xy} \otimes 1)f_b=0$ whenever $x \sim y$ in $G$ but $a \not\sim b$ in $H$. By the universal property, there is a unital $*$-homomorphism $\rho:\cA(\text{Hom}((\cS_G,D_n,M_n),H)) \to \cA(\text{Hom}(G,H))$ such that $\rho(p_{a,ij})=f_{a,ij}$. Evidently $\rho$ and $\pi$ are mutual inverses on the generators, so we conclude that the algebras are $*$-isomorphic. The result follows.
\end{proof}

It is known that some of the implications in (\ref{qgraphhom implications}) cannot be reversed. While there are many examples of classical graphs $G$ and $H$ with $G \xrightarrow{q} H$ but $G \not\to H$, Theorem \ref{theorem: only abelian complete graphs can be colored} will show that $(M_n,\cM,M_n) \xrightarrow{q} K_{\dim(\cM)}$ but $(M_n,\cM,M_n) \not\to K_{\dim(\cM)}$ whenever $\cM$ is non-abelian. Here, $K_{\dim(\cM)}$ denotes the (classical) complete graph on $\dim(\cM)$ vertices. Using the work of S.-J. Kim, V.I. Paulsen and C. Schafhauser on synchronous binary constraint (syncBCS) games, there is a graph $G$ and a number $m$ such that $K_m \xrightarrow{qa} \overline{G}$ holds, but $K_m \xrightarrow{q} \overline{G}$ does not hold, where $\overline{G}$ denotes the graph complement of $G$ \cite[Corollary~5.5]{KPS18}. The other known separation is that $\xrightarrow{alg}$ does not imply $\xrightarrow{hered}$. For example, $K_5 \xrightarrow{alg} K_4$ holds, but $K_5 \xrightarrow{hered} K_4$ does not hold \cite{HMPS19}. This result will be generalized to quantum graphs.

\section{Coloring quantum graphs}
\label{sec: coloring}

A special case of the homomorphism game is when the target graph is the classical complete graph $K_c$ on $c$ vertices. In this case, the resulting game is a generalization of the coloring game for classical graphs.

\begin{definition}
Let $t \in \{loc,q,qs,qa,qc,C^*,hered,alg\}$. Let $(\cS,\cM,M_n)$ be a quantum graph. We define $$\chi_t((\cS,\cM,M_n))=\min\{c \in \bN: (\cS,\cM,M_n) \xrightarrow{t} K_c\},$$
and we define $\chi_t((\cS,\cM,M_n))=\infty$ if $(\cS,\cM,M_n) \not\xrightarrow{t} K_c$ for all $c \in \bN$.
\end{definition}

Due the inclusions of the models, we always have
\begin{align*} \chi_{loc}((\cS,\cM,M_n)) &\geq \chi_q((\cS,\cM,M_n)) \geq \chi_{qa}((\cS,\cM,M_n)) \geq \chi_{qc}((\cS,\cM,M_n)) \\
&\geq \chi_{C^*}((\cS,\cM,M_n)) \geq \chi_{hered}((\cS,\cM,M_n)) \geq \chi_{alg}((\cS,\cM,M_n)).\end{align*}

As a consequence of Corollary \ref{corollary: qgraphhom for classical graphs}, whenever $G$ is a classical graph, we have $\chi_t(G)=\chi_t((\cS_G,D_n,M_n))$.  This result is well known (see e.g., \cite{PT15}). As $\chi_{loc}(G)$ is the (classical) chromatic number of a classical graph $G$, we sometimes use the notation $\chi((\cS,\cM,M_n))$ for $\chi_{loc}((\cS,\cM,M_n))$.

\begin{example}
Let
$$\cS=\spn \{I, \, E_{ij}: i \neq j\} \subseteq M_n,$$
which is a quantum graph on $M_n$. It is known \cite{KM19} that $\chi((\cS,M_n, M_n))=n$. Here, we will show that $\chi_{qc}((\cS,M_n))=n$ as well, which shows that $\chi_t((\cS,M_n))=n$ for any $t \in \{loc,q,qa,qc\}$.

Evidently the basis $\cF=\{I, E_{ij}: i \neq j\}$ is a quantum edge basis for $(\cS,M_n,M_n)$. Now, suppose that $P_1,...,P_c$ are projections in $M_n(\bofh)$ with $P_a(E_{k\ell} \otimes I)P_a=0$ for all $1 \leq a \leq c$ and $1 \leq k \neq \ell \leq n$. A winning strategy in the $qc$-model for coloring $(\cS,M_n)$ with $c$ colors would mean that there is a trace $\tau$ on the algebra generated by the $P_{a,ij}$'s and that
$$p(a,a|e_i \otimes e_j)=0 \text{ if } i \neq j.$$
This implies that
$$\tau(P_{a,ii}P_{a,jj}^*)=0 \text{ for all } i \neq j.$$
By taking a quotient by the kernel of the GNS representation of the trace, we may assume that $\tau$ is faithful. Then by faithfulness of $\tau$ and positivity of $P_{a,jj}$, we have $P_{a,ii}P_{a,jj}=0$ for all $i \neq j$. Now, choose $i \neq j$. Notice that, for each $i$, the set $\{P_{a,ii}\}$ is a POVM on $\cH$. Moreover, for any $a,b \in \{1,...,c\}$,
$$p(a,b|e_i \otimes e_j)=\tau(P_{a,ii}P_{b,jj}^*)=\tau(P_{a,ii}P_{b,jj}).$$
Thus, the only information relevant to Alice and Bob winning the game is the correlation $(\tau(P_{a,ii}P_{b,jj}))_{a,b,i,j} \in C_{qc}^s(n,c)$. By faithfulness, this forces each $P_{a,ii}$ to be a projection. By the synchronous condition, the previous equation and faithfulness of the trace, we obtain
$$P_{a,ii}P_{a,jj}=0=P_{a,ii}P_{b,ii}$$
whenever $a \neq b$ and $i \neq j$. Therefore, $(\tau(P_{a,ii}P_{b,jj}))_{a,b,i,j} \in C_{qc}^{bs}(n,c)$; that is, the correlation is \textit{bisynchronous} in the sense of \cite{PR19}. By \cite{PR19}, we must have $c \geq n$. Therefore, $\chi_{qc}(\cS,M_n,M_n) \geq n$. It follows that $\chi_t(\cS,M_n,M_n)=n$ for every $t \in \{loc,q,qa,qc\}$.
\end{example}

\subsection{Quantum Complete Graphs and Algebraic Colorings} 

In this section, we consider \textbf{quantum complete graphs}; that is, graphs of the form $(M_n,\cM,M_n)$, where $\cM \subseteq M_n$ is a non-degenerate von Neumann algebra. We show that $\chi_t((M_n,\cM,M_n))=\dim(\cM)$ for all $t \in \{q,qa,qc,C^*,hered\}$. In contrast, we will see that $\chi_{loc}((M_n,\cM,M_n))$ is finite if and only if $\cM$ is abelian; in the case when $\cM$ is abelian, we recover known results on colorings for the (classical) complete graph on $\dim(\cM)$ vertices. The algebraic model for colorings is known to be very wild. At the end of this section, we will extend a surprising result of \cite{HMPS19}: in the algebraic model, \textit{any} quantum graph can be $4$-colored.

We start with a simple proposition on unitary equivalence that we will use throughout this section.

\begin{proposition}
Let $\cM \subseteq M_n$ be a non-degenerate von Neumann algebra. Then there is a unitary $U \in M_n$ such that $U^*\cM U=\bigoplus_{r=1}^m \bC I_{n_r} \otimes M_{k_r}$. Moreover, for any $t \in \{loc,q,qa,qs,qc,C^*,hered,alg\}$, we have
\[ \chi_t((M_n,\cM,M_n))=\chi_t\left( M_n,\bigoplus_{r=1}^m \bC I_{n_r} \otimes M_{k_r}, M_n \right).\]
\end{proposition}

\begin{proof}
The existence of the unitary $U$ is a consequence of the theory of finite-dimensional $C^*$-algebras. It is not hard to see that $(U^*\cM U)'=U^*\cM' U$. Now, an element $X \in M_n$ belongs to $\cM'$ if and only if $\Tr(XY)=0$ for all $Y \in \cM'$. This statement is equivalent to having $\Tr((U^*XU)(U^*YU))=0$ for all $Y \in \cM'$, since $U$ is unitary. It follows that $U^*(\cM')^{\perp}U=(U^*\cM'U)^{\perp}$.

Now, suppose that $\{P_a\}_{a=1}^c \subseteq M_n \otimes \cA$ is a collection of self-adjoint idempotents summing to $I_n \otimes 1_{\cA}$, where $\cA$ is a unital $*$-algebra. Then it is evident that $P_a((\cM')^{\perp} \otimes 1_{\cA})P_a=0$ if and only if $\widetilde{P}_a ((U^*\cM'U)^{\perp} \otimes 1_{\cA})\widetilde{P}_a=0$, where $\widetilde{P}_a=(U^* \otimes 1_{\cA})P_a(U \otimes 1_{\cA})$. Similarly, if $a \neq b$, then $P_a(\cM' \otimes 1_{\cA})P_b=0$ if and only if $\widetilde{P}_a((U^*\cM'U) \otimes 1_{\cA})\widetilde{P}_b=0$. Thus, there is a bijective correspondence between algebraic $c$-colorings of $(M_n,\cM,M_n)$ and algebraic $c$-colorings of $\left( M_n,\bigoplus_{r=1}^m \bC I_{n_r} \otimes M_{k_r},M_n \right)$. This yields the equality of chromatic numbers for $t=alg$; the other cases are similar.
\end{proof}

The different chromatic numbers satisfy a certain monotonicity as well.

\begin{proposition}
\label{proposition: monotonicity of colorings}
If $(\cS,\cM,M_n)$ and $(\cT,\cM,M_n)$ are quantum graphs with $\cS \subseteq \cT$, then
$$\chi_t((\cS,\cM,M_n)) \leq \chi_t((\cT,\cM,M_n)).$$
\end{proposition}

\begin{proof}
We deal with the $t=alg$ case; all the other cases are similar. If $(\cT,\cM,M_n)$ has no algebraic coloring, then $\chi_{alg}((\cT,\cM,M_n))=\infty$, so the desired result holds. Otherwise, let $\cA$ be a (non-zero) unital $*$-algebra. Suppose that $\{ P_a\}_{a=1}^c$ are self-adjoint idempotents in $M_n(\cA)$ such that $\sum_{a=1}^c P_a=I_n$, $P_a((\cT \cap (\cM')^{\perp}) \otimes 1_{\cA})P_a=0$ for all $a$, and $P_a(\cM' \otimes 1_{\cA})P_b=0$ for all $a \neq b$. Then evidently $P_a((\cS \cap (\cM')^{\perp}) \otimes 1_{\cA})P_a=0$ as well, so the self-adjoint idempotents form an algebraic $c$-coloring of $(\cS,\cM,M_n)$. This shows that $\chi_{alg}((\cS,\cM,M_n)) \leq \chi_{alg}((\cT,\cM,M_n))$. The proof for the other models is the same.
\end{proof}

By Proposition \ref{proposition: monotonicity of colorings}, to establish that every quantum graph has a finite quantum coloring, it suffices to consider quantum complete graphs. First, we look at $(M_n,M_n,M_n)$, the quantum complete graph. While we will have an alternative quantum coloring of this quantum graph from Theorem \ref{theorem: quantum coloring of quantum complete graph}, the protocol given in Theorem \ref{theorem: quantum teleportation protocol} is minimal for $(M_n,M_n,M_n)$ in terms of the dimension of the ancillary algebra. Moreover, it gives a foretaste of the protocol that we use for the quantum complete graph $(M_n,\cM,M_n)$ when $\cM$ is not isomorphic to a matrix algebra.

\begin{theorem}
\label{theorem: quantum teleportation protocol}
Let $d,k \in \bN$, and let $n=dk$. Let $\cM=\bC I_d \otimes M_k$. Then $\chi_q((M_n,\cM,M_n)) \leq k^2$.
\end{theorem}

\begin{proof}
We construct our projections from the canonical orthonormal basis for $\bC^k \otimes \bC^k$ that consists of maximally entangled vectors; that is, the basis of the form
\[ \varphi_{a,b}=\frac{1}{\sqrt{k}} \sum_{p=0}^{k-1} \exp\left( \frac{2\pi i a (p+b)}{k}\right) e_{b+p} \otimes e_p,\]
where addition in the indices of the vectors is done modulo $k$. (See \cite{GHJ19} for example.) We define projections in $\cM \otimes \cM$, for all $1 \leq a,b \leq n$, by
\[ P_{(a,b)}=\frac{1}{k} \sum_{p,q=0}^{k-1} \exp\left( \frac{2\pi i a(p-q)}{k}\right) I_d \otimes E_{b+p,b+q} \otimes I_d \otimes E_{pq}.\]
Since the set $\{ \varphi_{(a,b)} \}_{a,b=1}^n$ is orthonormal, it is not hard to see that $\{ P_{(a,b)}\}_{a,b=1}^n$ is a family of mutually orthogonal projections. Moreover, $\sum_{a,b=1}^n P_{(a,b)}=I_d \otimes I_k \otimes I_d \otimes I_k$. With respect to $M_n$, $(\cM')^{\perp}$ is spanned by elements of the form $E_{xy} \otimes E_{vw}$ and $E_{xy} \otimes (E_{vv}-E_{ww})$ for $1 \leq x,y \leq d$ and $1 \leq v,w \leq k$ with $v \neq w$. For $Y=E_{xy} \otimes E_{vw} \otimes (I_d \otimes I_k)$, one computes $P_{(a,b)}YP_{(a,b)}$ and obtains
\[\frac{1}{k^2} \sum_{p,q,p',q'=0}^{k-1} \exp\left(\frac{2\pi i a (p+p'-q-q')}{k}\right) E_{xy} \otimes E_{b+p,b+q}E_{vw}E_{b+p',b+q'} \otimes I_d \otimes E_{pq}E_{p'q'}.\]
For a term in the above sum to be non-zero, one requires that $b+q=v$, $w=b+p'$, and $q=p'$. Equivalently, a term in the sum is non-zero only when $q=p'$ and $b+q=v=w$. Hence, if $v \neq w$, then the above sum is $0$. In the case when $v=w$, one obtains
\[ \frac{1}{k^2} \sum_{p,q'=0}^{k-1} \exp \left( \frac{2\pi i a (p-q')}{k} \right) E_{xy} \otimes E_{b+p,b+q'} \otimes I_d \otimes E_{pq'}. \]
The above expression does not depend on $v$, so we conclude that, for all $1 \leq v,w \leq k$, 
\[ P_{(a,b)}(E_{xy} \otimes E_{vv} \otimes I_d \otimes I_k)P_{(a,b)}=P_{(a,b)}(E_{xy} \otimes E_{ww} \otimes I_d \otimes I_k)P_{(a,b)}. \]
This shows that $P_{(a,b)}(X \otimes I_d \otimes I_k)P_{(a,b)}=0$ whenever $X=E_{xy} \otimes E_{vw}$ or $X=E_{xy} \otimes (E_{vv}-E_{ww})$ for $v \neq w$. As such elements span $(\cM')^{\perp}$, we see that
\[ P_{(a,b)}(X \otimes I_d \otimes I_k)P_{(a,b)}=0 \, \forall X \in (\cM')^{\perp}.\]
Finally, we show that $P_{(a,b)}(\cM' \otimes I_d \otimes I_k)P_{(a',b')}=0$ whenever $(a,b) \neq (a',b')$. If $Y \in \cM'$, then $Y \otimes (I_d \otimes I_k)$ commutes with each $P_{(a,b)}$, since $P_{(a,b)} \in \cM \otimes (I_d \otimes M_k)$. Therefore, if $(a,b) \neq (a',b')$, we have
\[ P_{(a,b)}(Y \otimes (I_d \otimes I_k))P_{(a',b')}=P_{(a,b)}P_{(a',b')}(Y \otimes (I_d \otimes I_k))=0. \]
Putting all of these equations together, we see that there is a representation of the game algebra $\pi: \cA(\text{Hom}((M_n,\cM,M_n),K_{k^2})) \to \bC I_d \otimes M_k \otimes M_k$. Therefore, $\chi_q((M_n,\cM,M_n)) \leq k^2$, which yields the claimed result.
\end{proof}

For a general complete quantum graph $(M_n,\cM,M_n)$, we require a slightly different approach. The protocol in the previous proof is used in the context of quantum teleportation, and essentially arises from the use of a ``shift and multiply" unitary error basis for $M_n$ \cite{GHJ19, werner01}. To give a $\dim(\cM)$-coloring for $(M_n,\cM,M_n)$ in the $q$-model, we will use what we refer to as a ``global shift and local multiply" framework.

\begin{theorem}
\label{theorem: quantum coloring of quantum complete graph}
Let $\cM$ be a non-degenerate von Neumann algebra in $M_n$. For the quantum complete graph $(M_n,\cM,M_n)$, we have $\chi_q((M_n,\cM,M_n)) \leq \dim(\cM)$.
\end{theorem}

\begin{proof}
Up to unitary equivalence in $M_n$, we may write $\cM=\bigoplus_{r=1}^m (\bC I_{n_r} \otimes M_{k_r})$, where $n=\sum_{r=1}^m n_rk_r$. We will exhibit a PVM in $\cM \otimes M_d$, with $d=\text{lcm}(k_1,...,k_m)$, satisfying the properties of a quantum coloring for $(M_n,\cM,M_n)$. For notational convenience, we index our set of $\dim(\cM)$ colors by the triples $(s,a,b)$, where $1 \leq s \leq m$ and $0 \leq a,b \leq k_s-1$. For $1 \leq r \leq m$ and $1 \leq i \leq k_r$, we define $P_{(s,a,b)}=\bigoplus_{r=1}^m I_{n_r} \otimes P_{(a,b)}^{(r,s)}$, where $P_{(a,b)}^{(r,s)}=(P_{(a,b),(i,j)}^{(r,s)})_{i,j=0}^{k_r-1} \in M_{k_r}(M_d)$ is given by
\[ P_{(a,b),(i,j)}^{(r,s)}=\frac{\delta_{rs}}{k_r} \omega_{k_r}^{(i-j)a} I_{d_r} \otimes E_{i+b,j+b},\]
where $\omega_{k_r}$ is a primitive $k_r$-th root of unity and $d_r=\frac{d}{k_r}$. (Note that indices are computed modulo $k_r$.) By our choice of the operators $P_{(a,b)}^{(r,s)}$, we see that each $P_{(s,a,b)}$ belongs to $\cM \otimes M_d$.

First, we show that $\sum_{s=1}^m \sum_{a,b=0}^{k_s-1} P_{(s,a,b)}=I_n \otimes I_d$. For each $1 \leq r \leq m$ and $0 \leq i,j \leq k_r-1$,
\[ \sum_{a,b=0}^{n_r-1} P_{(a,b),(i,j)}^{(r,r)}=\frac{1}{k_r} \sum_{a,b=0}^{n_r-1} \omega_{k_r}^{(i-j)a} I_{d_r} \otimes E_{i+b,j+b}.\]
If $i \neq j$, then the above sum over $a$ is $0$, for each value of $b$. If $i=j$, then the above sum simply becomes
\[ \sum_{b=0}^{n_r-1} I_{d_r} \otimes E_{i+b,i+b}=I_{d_r} \otimes I_{n_r}=I_d.\]
Thus, $\sum_{a,b=0}^{n_r-1} P_{(a,b)}^{(r,r)}=I_{k_r} \otimes I_d$. Since $P_{(a,b)}^{(r,s)}=0$ if $s \neq r$, it follows that $\sum_{s=1}^m \sum_{a,b=0}^{n_s-1} P_{(a,b)}^{(r,s)}=I_{k_r} \otimes I_d$ for each $1 \leq r \leq m$. As $P_{(s,a,b)}=\bigoplus_{r=1}^m I_{n_r} \otimes P_{(a,b)}^{(r,s)}$, we must have $\sum_{s=1}^m \sum_{a,b=0}^{n_s-1} P_{(s,a,b)}=I_n \otimes I_d$.

Next, we check that each $P_{(s,a,b)}$ is an orthogonal projection. By definition, it is easy to see that $P_{(s,a,b)}^*=P_{(s,a,b)}$ for all $s,a,b$. To compute $P_{(s,a,b)}^2$, we note that
\[ P_{(s,a,b)}^2=\bigoplus_{r=1}^m I_{n_r} \otimes (P_{(a,b)}^{(r,s)})^2,\]
so it suffices to show that each $P_{(a,b)}^{(r,s)}$ is an idempotent in $M_{k_r} \otimes M_d$. If $r \neq s$, then this is immediate. In the other case, we have
\begin{align*}
P_{(a,b),(v,j)}^{(r,r)} P_{(a,b),(j,w)}^{(r,r)}&=\frac{1}{k_r^2} \omega_{k_r}^{(v-w)a} I_{d_r} \otimes E_{v+b,j+b} E_{j+b,w+b} \\
&=\frac{1}{k_r^2} \omega_{k_r}^{(v-w)a} I_{d_r} \otimes E_{v+b,w+b} \\
&=\frac{1}{k_r} P_{(a,b),(v,w)}^{(r,r)}.
\end{align*}
Since this happens for all $0 \leq v,w \leq k_r-1$, it follows that $P_{(a,b),(v,w)}^{(r,r)}=\sum_{j=0}^{k_r-1} P_{(a,b),(v,j)}^{(r,r)} P_{(a,b),(j,w)}^{(r,r)}$. Therefore, $P_{(a,b)}^{(r,r)}$ is idempotent in $M_{k_r} \otimes M_d$, so $P_{(s,a,b)}$ is an orthogonal projection.

Now, we show that $P_{(s,a,b)}((\cM')^{\perp} \otimes I_d)P_{(s,a,b)}=0$ for all $a$. We note that $\cM'=\bigoplus_{r=1}^m M_{n_r} \otimes \bC I_{k_r}$. Hence, $(\cM')^{\perp}$ is spanned by the canonical matrix units that do not reside in $\cM'$, and elements from each $M_{n_r} \otimes M_{k_r}$ of the form $E_{ij} \otimes E_{vw}$ and $E_{ij} \otimes (E_{vv}-E_{ww})$, where $1 \leq i,j \leq n_r$, $0 \leq v,w \leq k_r-1$, and $v \neq w$. By a consideration of blocks, if a matrix unit $E_{xy}$ does not belong to $\bigoplus_{r=1}^m M_{n_r} \otimes M_{k_r}$, then in $M_n \otimes M_d$, the element $P_{(s,a,b)}(E_{xy} \otimes I_d)P_{(s,a,b)}$ is a product of two entries from $P_{(s,a,b)}$, at least one of which will be $0$.

Next, we suppose that $0 \leq v,w \leq k_r-1$ with $v \neq w$ and $1 \leq i,j \leq n_r$, and consider the matrix unit $E_{ij} \otimes E_{vw} \in M_{n_s} \otimes M_{k_s} \subset \bigoplus_{r=1}^m M_{n_r} \otimes M_{k_r}$. Since $P_{(s,a,b)}=\bigoplus_{r=1}^m I_{n_r} \otimes P_{(a,b)}^{(r,s)}$,
\[ P_{(s,a,b)}(E_{ij} \otimes E_{vw} \otimes I_d)P_{(s,a,b)}=E_{ij} \otimes \sum_{k,\ell=0}^{k_r-1} P_{(a,b),(k,v)}^{(s,s)} P_{(a,b),(w,\ell)}^{(s,s)}=0,\]
since $P_{(a,b),(k,v)}^{(s,s)}P_{(a,b),(w,\ell)}^{(s,s)}=0$ for all $v \neq w$ in $\{0,...,n_s-1\}$. For the last case, we look at the element $E_{ij} \otimes (E_{vv}-E_{ww})$ in $M_{n_s} \otimes M_{k_s} \subset (\cM')^{\perp}$, where $v \neq w$. Multiplying on the left and right by $P_{(s,a,b)}$ yields
\[P_{(s,a,b)}(E_{ij} \otimes (E_{vv}-E_{ww}) \otimes I_d)P_{(s,a,b)}=E_{ij} \otimes \sum_{k,\ell=0}^{k_r-1} (P_{(s,a,b),(k,v)}^{(s)}P_{(s,a,b),(v,\ell)}^{(s)}-P_{(s,a,b),(k,w)}^{(s)}P_{(s,a,b),(w,\ell)}^{(s)})=0,\]
since $P_{(s,a,b),(k,v)}^{(s)}P_{(s,a,b),(v,\ell)}^{(s)}=\frac{1}{k_r} P_{(s,a,b),(k,\ell)}^{(s)}$ for any $1 \leq s \leq m$ and $0 \leq k,v,\ell \leq k_r-1$. 
Putting all of these facts together, we conclude that $\{ P_{(s,a,b)}: 1 \leq s \leq m, \, 0 \leq a,b \leq n_s-1\} \subset \cM \otimes M_d$ is a quantum $\dim(\cM)$-coloring of $(M_n,\cM,M_n)$, as desired.
\end{proof}

\begin{remark}
We suspect that the ancillary algebra in the previous proof is the minimal choice, but are unable to prove this. In the case when $\cM=M_n$, this is immediate, since having a PVM with $n^2$ outputs in $M_n \otimes M_f$, and with each projection non-zero, requires $f \geq n$.
\end{remark}

Next, we will show that $\chi_{hered}((M_n,\cM,M_n)) \geq \dim(\cM)$, which will show that, for every $t \in \{q,qa,qc,C^*,hered\}$, we have $\chi_t((M_n,\cM,M_n))=\dim(\cM)$. Moreover, we will show that $\dim(\cM)$-colorings of $(M_n,\cM,M_n)$ in the hereditary model must arise from {\it trace-preserving} $*$-homomorphisms $\Psi: D_{\dim(\cM)} \to \mathcal M \otimes  \mathcal A$.  More precisely, we equip $D_{\dim(\cM)}$ with its canonical uniform trace $\psi_{D_{\dim(\cM)}}$ satisfying $\psi_{D_{\dim(\cM)}}(e_a)=\frac{1}{\dim(\cM)}$ for all $1 \leq a \leq \dim(\cM)$.  We also equip the von Neumann algebra $\cM \simeq \bigoplus_{r=1}^m \bC I_{n_r} \otimes M_{k_r}$ with its canonical ``Plancherel'' trace given by
\[ \psi_{\cM}=\bigoplus_{r=1}^m \frac{k_r}{n_r\dim(\cM)} \Tr_{n_rk_r}(\cdot).\]  Then we will show that the $\ast$-homomorphism $\Psi$ satisfies the following trace covariance condition:
\[
(\psi_{\cM} \otimes \id)\Psi(x) = \psi_{D_{\dim(\cM)}}(x)1_\mathcal A \qquad (x \in D_{\dim(\cM)}).
\]
We thus establish that the hereditary coloring number for any complete quantum graph $(M_n,\cM,M_n)$ is $\dim (\cM)$, and moreover, the above trace-preserving condition shows that any minimal hereditary coloring induces a quantum version of isomorphism between $(M_n,\cM,M_n)$ and the complete graph $K_{\dim(\cM)}$ on $\dim(\cM)$ vertices. Here, the notion of a ``quantum isomorphism" means a quantum  isomorphism between quantum graphs in the sense of \cite{Br+20}, when using an ancillary hereditary unital $*$-algebra $\cA$.  This result can be interpreted as a quantum analogue of the (classically obvious) fact that any minimal coloring of a complete graph $K_c$ is automatically a graph {\it isomorphism} $K_c \to K_c$. 

We consider the case when $\cM \simeq \bC I_d \otimes M_k$ first.

\begin{lemma}
\label{lemma: rigidity of complete colorings for one block}
Let $d,k \in \bN$ and let $n=dk$. Consider the quantum graph $(M_n,\cM,M_n)$ with $\cM=\bC I_d \otimes M_k$. Let $\cA$ be a unital $*$-algebra, and let $\{P_1,...,P_c\} \in \cM \otimes \cA$ be a family of mutually orthogonal projections such that $\sum_{a=1}^c P_a=I_{dk} \otimes 1_{\cA}$ and
\[ P_a (X \otimes 1_{\cA})P_a=0 \text{ for all } X \in (\cM')^{\perp}. \]
Then for each $a$, the element $R_a=\frac{k}{d \dim(\cM)} (\Tr_{dk} \otimes \id_{\cA})(P_a)$ is a self-adjoint idempotent in $\cA$, and $\sum_{a=1}^c R_a=k^2 1_{\cA}$.
\end{lemma}

\begin{proof}
Since $\cM=I_d \otimes M_k$, we have $\cM'=M_d \otimes I_k$ and $n=dk$. 
Now, let $1 \leq v,w \leq k$ with $x \neq y$, and let $1 \leq i,j \leq d$. Then $E_{ij} \otimes (E_{vv}-E_{ww})$ belongs to $(\cM')^{\perp}$, so we must have
$$P_a(E_{ij} \otimes (E_{vv}-E_{ww}) \otimes 1_{\cA})P_a=0 \, \forall 1 \leq a \leq c.$$
 Similarly, $E_{ij} \otimes E_{vw}$ is in $(\cM')^{\perp}$, so
 $$P_a(E_{ij} \otimes E_{vw} \otimes 1_{\cA})P_a=0 \, \forall 1 \leq a \leq c.$$
Note that $P_a \in \cM \otimes \cA=I_d \otimes M_k \otimes \cA$, so $P_a=\sum_{p,q=1}^k \sum_{x=1}^d E_{xx} \otimes E_{pq} \otimes P_{a,x,pq},$ with the property that $P_{a,x,pq}=P_{a,y,pq}$ for any $1 \leq x,y \leq d$. For simplicity, we set $P_{a,pq}=P_{a,x,pq}$ for any $1 \leq x \leq d$. The quantity on the left of the above is exactly
$$\sum_{p,q=1}^k E_{ij} \otimes E_{pq} \otimes P_{a,pv}P_{a,wq}$$
so this says that $P_{a,pv}P_{a,wq}=0$ and $P_{a,pv}P_{a,vq}=P_{a,pw}P_{a,wq}$. Now, since $P_a$ is a projection, we have $P_{a,pq}=\sum_{v=1}^k P_{a,pv}P_{a,vq}=kP_{a,pv}P_{a,vq}$ for all $p,q$. In particular, $P_{a,vv}=kP_{a,vv}^2$. By scaling, we see that $kP_{a,vv}$ is a self-adjoint idempotent. Similarly, since $P_{a,pv}P_{a,wq}=0$ if $v \neq w$, we see that $P_{a,vv}P_{a,ww}=0$. Therefore, $\{ kP_{a,vv}\}_{v=1}^n$ is a collection of mutually orthogonal projections in $\cA$. 

Next, we set $R_a=\sum_{v=1}^k kP_{a,vv}$ for each $1 \leq a \leq c$. Then $R_a$ is a self-adjoint idempotent. We see that
\[ \sum_{a=1}^c R_a=\sum_{a=1}^c \sum_{v=1}^k kP_{a,vv}=\sum_{v=1}^k k1_{\cA}=k^2 1_{\cA}, \]
which completes the proof.
\end{proof}

Now, we deal with the case of a general quantum complete graph.

\begin{theorem}
\label{theorem: rigidity of colorings}
Let $(M_n,\cM,M_n)$ be a quantum complete graph. Let $\cA$ be a hereditary $*$-algebra, and let $\{ P_a\}_{a=1}^c \subseteq \cM \otimes \cA$ be a hereditary $c$-coloring of $(M_n,\cM,M_n)$. Then $c \geq \dim(\cM)$. Moreover, if $c=\dim(\cM)$, then for each $1 \leq a \leq \dim(\cM)$ we have
\[ (\psi_{\cM} \otimes \id_{\cA})(P_a)=\frac{1}{\dim(\cM)} 1_{\cA}.\]
\end{theorem}

\begin{proof}
Up to unitary equivalence, we may write $\cM=\bigoplus_{r=1}^m \bC I_{n_r} \otimes M_{k_r}$. Then
\[ \cM'=\bigoplus_{r=1}^m M_{n_r} \otimes \bC I_{k_r}.\]
Define $\cE_r=0 \oplus \cdots \oplus I_{n_r} \otimes I_{k_r} \oplus 0 \oplus \cdots \oplus 0$, which belongs to $\cM' \cap \cM$. Then defining $\widetilde{P}_a=(\cE_r \otimes 1_{\cA})P_a(\cE_r \otimes 1_{\cA}) \in (\cE_r\cM \cE_r) \otimes \cA$, we obtain a family of mutually orthogonal projections whose sum is $\cE_r$. Since $\cE_r$ is central in $\cM$, we see that $(\cE_r\cM \cE_r)'=\cE_r \cM' \cE_r$, while $\cE_rM_n\cE_r=M_{n_rk_r}$. It is evident that $X \in \cB(\cE_r \bC^n) \cap (\cE_r \cM' \cE_r)^{\perp}$ if and only if $X=\cE_r X \cE_r$ and $X \perp \cM'$ in $M_n$.
Therefore, for $X \in \cB(\cE_r \bC^n) \cap (\cE_r \cM' \cE_r)^{\perp}$ and $1 \leq a \leq c$, one has
\[ \widetilde{P}_a (X \otimes 1_{\cA}) \widetilde{P}_a=(\cE_r \otimes 1_{\cA})P_a(\cE_r X \cE_r \otimes 1_{\cA})P_a (\cE_r \otimes 1_{\cA})=0,\]
using the fact that $\cE_r X \cE_r=X$ and $X$ belongs to $\cM'$. Therefore, $\{ \widetilde{P}_a\}_{a=1}^c$ is a hereditary coloring of the quantum complete graph $(M_{n_rk_r},\cE_r \cM \cE_r, M_{n_rk_r})$. 

Since $\cE_r \cM \cE_r=\bC I_{n_r} \otimes M_{k_r}$, by Lemma \ref{lemma: rigidity of complete colorings for one block}, we see that $R_a^{(r)}:=\frac{k_r}{n_r}(\Tr_{n_rk_r} \otimes \id_{\cA})(\widetilde{P}_a)$ is a self-adjoint idempotent in $\cA$ for each $1 \leq a \leq c$ and $1 \leq r \leq m$. Moreover, $\sum_{a=1}^c R_a^{(r)}=k_r^2 1_{\cA}$.

Next, we claim that $R_a^{(r)}R_a^{(s)}=0$ if $r \neq s$. To show this orthogonality relation, it suffices to show that $P_{a,xx}P_{a,yy}=0$ whenever $P_{a,xx}$ is a block from $(\cE_r \cM \cE_r) \otimes \cA$ and $P_{a,yy}$ is a block from $(\cE_s \cM \cE_s) \otimes \cA$. If $x$ and $y$ are chosen in this way, then the matrix unit $E_{xy}$ in $M_n$ satisfies $\cE_r (E_{xy})\cE_s=E_{xy}$ and $\cE_p E_{xy} \cE_q=0$ for all other pairs $(p,q)$. It is not hard to see that $E_{xy}$ belongs to $(\cM')^{\perp}$, so that $P_a(E_{xy} \otimes 1_{\cA})P_a=0$. Considering the $(x,y)$-block of this equation gives $P_{a,xx}P_{a,yy}=0$. It follows that $R_a^{(r)}R_a^{(s)}=0$ for $r \neq s$.

Since $\{ R_a^{(r)}\}_{r=1}^m$ is a collection of mutually orthogonal projections in $\cA$, the element $R_a:=\sum_{r=1}^m R_a^{(r)}$ is a self-adjoint idempotent in $\cA$ for each $a$. Considering blocks, it is not hard to see that 
\[ \sum_{a=1}^c R_a=\sum_{a=1}^c \sum_{r=1}^m R_a^{(r)}=\sum_{r=1}^m k_r^2 1_{\cA}=\dim(\cM)1_{\cA}.\]
Since $R_a$ is a self-adjoint idempotent, so is $1_{\cA}-R_a$. Their sum is given by
\[ \sum_{a=1}^c (1_{\cA}-R_a)=c 1_{\cA} - \sum_{a=1}^c R_a=(c-\dim(\cM))1_{\cA}.\]
It follows that $c \geq \dim(\cM)$, since the sum above is a sum of positives and $\cA$ is hereditary.

Now, if $c=\dim(\cM)$, then the above sum of positives in $\cA$ is $0$, which forces $1_{\cA}-R_a=0$ for all $a$. Hence, $R_a=1_{\cA}$. Since $R_a=\sum_{r=1}^m R_a^{(r)}$ and $R_a^{(r)}=\frac{k_r}{n_r} (\text{Tr}_{n_rk_r} \otimes \id_{\cA})(P_a)$, we see that
\[ \sum_{r=1}^m \frac{k_r}{n_r} (\text{Tr}_{n_rk_r} \otimes \id_{\cA})(P_a)=1_{\cA}. \]
Therefore,
\[ (\psi_{\cM} \otimes \id_{\cA})(P_a)=\sum_{r=1}^m \frac{k_r}{\dim(\cM)n_r} (\Tr_{n_rk_r} \otimes \id_{\cA})(P_a)=\frac{1}{\dim(\cM)}1_{\cA}.\]
\end{proof}

\begin{remark}
In essence, Theorem \ref{theorem: rigidity of colorings} proves that any $q$-coloring of $(M_n,\cM,M_n)$ with $\dim(\cM)$ colors induces a quantum isomorphism between the quantum graph $(M_n,\cM,M_n)$ and the classical graph $K_{\dim(\cM)}$. This isomorphism occurs because any such coloring with ancillary algebra $\cA$ yields a (necessarily injective) unital $*$-isomorphism $\pi:D_{\dim(\cM)} \to \cM \otimes \cA$ satisfying the properties of a quantum graph homomorphism, with the additional property that $(\psi_{\cM} \otimes \id_{\cA}) \circ \pi=\pi \circ \psi_{D_{\dim(\cM)}}$.
\end{remark}

In contrast to the case of $q$-colorings, the existence of a $loc$-coloring for a complete quantum graph is equivalent to the von Neumann algebra being abelian.
 
\begin{theorem}
\label{theorem: only abelian complete graphs can be colored}
Let $\cM \subseteq M_n$ be a non-degenerate von Neumann algebra. Then $\chi_{loc}((M_n,\cM,M_n))$ is finite if and only if $\cM$ is abelian. In particular, if $\cM$ is non-abelian, then $\chi((M_n,\cM,M_n)) \neq \chi_q((M_n,\cM,M_n))$.
\end{theorem}

\begin{proof}
Suppose that there is a  $c$-coloring of $(M_n,\cM,M_n)$ in the $loc$-model. Up to unitary equivalence, we write $\cM=\bigoplus_{r=1}^m \bC I_{n_r} \otimes M_{k_r}$. We may choose projections $P_a \in \cM$ such that $\sum_{a=1}^c P_a=I_n$ and $P_a((\cM')^{\perp})P_a=0$ for all $a$. Let $R_a=\sum_{r=1}^m \frac{k_r}{n_r} \Tr_{n_rk_r}(P_a)$ as in the proof of the last theorem. Each $R_a$ is an idempotent in $\bC$; hence, either $R_a=0$ or $R_a=1$. We know that $\sum_{a=1}^c R_a=\dim(\cM)$, so exactly $\dim(\cM)$ of the $R_a$'s are non-zero. Since $R_a$ is given by a trace on $\cM$ which is faithful, having $R_a=0$ implies that $P_a=0$. Hence, by discarding any projections $P_a$ for which $R_a=0$, we may assume without loss of generality that $R_a=1$ for all $a$, and that $c=\dim(\cM)$.

Let $\cE_r$ be the orthogonal projection onto the copy of $\bC I_{n_r} \otimes M_{k_r}$ inside of $\cM=\bigoplus_{r=1}^m \bC I_{n_r} \otimes M_{k_r}$. Then, as before, the PVM $\{ \cE_r P_a \cE_r \}_{a=1}^{\dim(\cM)}$ yields a classical $\dim(\cM)$-coloring for $(M_{n_rk_r},\bC I_{n_r} \otimes M_{k_r},M_{n_rk_r})$. We will show that $k_r=1$. By the same argument as above, by discarding values of $a$ for which $\frac{k_r}{n_r} \Tr_{n_rk_r}(\cE_r P_a \cE_r)=0$, we may assume that there are exactly $k_r^2$ non-zero projections $\cE_r P_a \cE_r$ that yield a $k_r^2$-classical coloring for $(M_{n_rk_r},\bC I_{n_r} \otimes M_{k_r},M_{n_rk_r})$. Set $\widetilde{P}_a=\cE_r P_a \cE_r$. By Theorem \ref{theorem: rigidity of colorings}, for each $a$, we have $\frac{k_r}{n_r} \Tr_{n_rk_r}(\widetilde{P}_a)=1$. Notice that $k_r\widetilde{P}_a=I_{n_r} \otimes k_rQ_a$ for some projection $k_rQ_a \in M_{k_r}$. Hence, $\Tr_{k_r}(k_r Q_a)=1$. Let $\lambda_1,...,\lambda_{k_r}$ be the eigenvalues of $k_r Q_a$ in $M_{k_r}$. Since each $\lambda_i \in \{0,1\}$ and $\sum_{i=1}^{k_r}\lambda_i=\Tr_{k_r}(k_r Q_a)=1$, there is exactly one $\lambda_i$ that is non-zero. Hence, $Q_a$ is rank one. The sum over all non-zero $Q_a$ gives $I_{k_r}$, and each $Q_a$ is rank one. Hence, the number of $a$ for which $Q_a$ is non-zero must be $k_r$. Since we assumed that this number is $k_r^2$, we must have $k_r=k_r^2$. Since $k_r>0$, we have $k_r=1$. Since $r$ was arbitrary, we see that $\cM=\bigoplus_{r=1}^m \bC I_{n_r} \otimes M_{k_r}=\bigoplus_{r=1}^m \bC I_{n_r}$ is abelian.

Conversely, suppose that $\cM$ is abelian. Then the proof of Theorem \ref{theorem: rigidity of colorings} yields projections $P_a \in \cM \otimes M_d$, where $d=\dim(\cM)$, such that $\sum_{a=1}^d P_a=I_n \otimes I_d$ and $P_a(X \otimes I_d)P_a=0$ whenever $X \in (\cM')^{\perp}$. Moreover, the projections obtained in this case satisfy $P_{a,ij}P_{b,k\ell}=P_{b,k\ell}P_{a,ij}$ for all $1 \leq a,b \leq d$ and $1 \leq i,j,k,\ell \leq n$. Thus, the entries of the projections $P_a$ must $*$-commute with each other, so the $C^*$-algebra they generate is abelian. Since there is a $d$-coloring for $(M_n,\cM,M_n)$ with an abelian ancilla, this implies that $\chi_{loc}((M_n,\cM,M_n)) \leq d$.
\end{proof}

Using the monotonicity of colorings and the results above on quantum complete graphs, we see that every quantum graph has a finite quantum coloring. As a result, we obtain the following generalization of a theorem from \cite{HMPS19}.

\begin{theorem}
Let $(\cS,\cM,M_n)$ be any quantum graph. Then $\chi_{alg}((\cS,\cM,M_n)) \leq 4$.
\end{theorem}

\begin{proof} Suppose that $\chi_{alg}(\cS,\cM,M_n) \leq c$ for some $c<\infty$. Then $\mathcal{A}(\text{Hom}((\cS,\cM,M_n),K_c))$ exists. We will let $p_1,...,p_c$ be the canonical self-adjoint idempotents in the matrix algebra $M_n(\mathcal{A}(\text{Hom}((\cS,\cM,M_n),K_c))$. By \cite{HMPS19}, there is an algebraic homomorphism $K_c \to K_4$. Thus, there are self-adjoint idempotents $f_{a,v}$ in $\mathcal{A}(\text{Hom}(K_c,K_4))$ for $1 \leq a \leq c$ and $1 \leq v \leq 4$ such that $\sum_{v=1}^4 f_{a,v}=1$ for all $a$ and $f_{a,v}f_{b,v}=0$ whenever $a \neq b$. Define
$$q_{v,ij}=\sum_{a=1}^c p_{a,ij} \otimes f_{a,v} \in \mathcal{A}(\text{Hom}((\cS,\cM,M_n),K_c)) \otimes \mathcal{A}(\text{Hom}(K_c,K_4)).$$
Then
\begin{align*}
\sum_{k=1}^n q_{v,ik}q_{v,kj}&=\sum_{k=1}^n \left( \sum_{a=1}^c p_{a,ik} \otimes f_{a,v}\right)\left(\sum_{b=1}^c p_{b,kj} \otimes f_{b,v}\right) \\
&=\sum_{k=1}^n \sum_{a,b=1}^c p_{a,ik}p_{b,kj} \otimes f_{a,v}f_{b,v} \\
&=\sum_{a=1}^c \sum_{k=1}^n p_{a,ik}p_{a,kj} \otimes f_{a,v}^2 \\
&=\sum_{a=1}^c p_{a,ij} \otimes f_{a,v}=q_{v,ij}.
\end{align*}
Therefore, $q_v=(q_{v,ij})$ is an idempotent for each $v$. Similarly, one can see that $q_v^*=q_v$ (that is, $q_{v,ij}^*=q_{v,ji}$) and $\sum_{v=1}^4 q_{v,ij}$ is $0$ if $i \neq j$ and $1$ if $i=j$. Let $X=(x_{ij}) \in M_n$. Letting $1 \otimes 1$ denote the unit in the tensor product of the game algebras,
\begin{equation}
q_v(X \otimes 1 \otimes 1)q_w=\left(\sum_{k,\ell=1}^n q_{v,ik} x_{k\ell} q_{w,\ell j} \right)_{i,j}=\left( \sum_{k,\ell=1}^n \sum_{a,b=1}^c p_{a,ik}x_{k\ell}p_{b,\ell j} \otimes f_{a,v}f_{b,w}\right)_{i,j}. \label{equation: quantum composition}
\end{equation}
If $X \in (\cM')^{\perp}$ and $a=b$, then the above sum becomes
$$q_v(X \otimes 1 \otimes 1)q_v=\left( \sum_{k,\ell=1}^n \sum_{a=1}^c p_{a,ik}x_{k\ell}p_{a,\ell j} \otimes f_{a,v}\right)=\sum_{a=1}^c p_a(X \otimes 1)p_a \otimes f_{a,v}=0,$$
by definition of $\cA(\text{Hom}((\cS,\cM,M_n),K_c))$. If $X \in \cM'$ and $a \neq b$, then $\sum_{k,\ell=1}^np_{a,ik}x_{k\ell}p_{b,\ell j}$ is the $(i,j)$ entry of $p_a(X \otimes 1)p_b=0$. Thus, if $v \neq w$, then Equation (\ref{equation: quantum composition}) reduces to
$$q_v(X \otimes 1 \otimes 1)q_w=\left( \sum_{k,\ell=1}^n \sum_{a=1}^c p_{a,ik}x_{k\ell}p_{a,\ell j} \otimes f_{a,v}f_{a,w} \right)_{i,j}=0,$$
since $f_{a,v}f_{a,w}=0$ for $v \neq w$. Therefore, letting $r_{v,ij}$ be the canonical generators of $\cA(\text{Hom}((\cS,\cM,M_n),K_4))$, we obtain a unital $*$-homomorphism
\begin{align*}
\pi:\mathcal{A}(\text{Hom}((\cS,\cM,M_n),K_4)) &\to \mathcal{A}(\text{Hom}((\cS,\cM,M_n),K_c)) \otimes \mathcal{A}(\text{Hom}(K_c,K_r)), \\
r_{v,ij} &\mapsto q_{v,ij}.
\end{align*}
The latter algebra is non-zero, so $\cA(\text{Hom}((\cS,\cM,M_n),K_4)) \neq \{0\}$. Thus, $\chi_{alg}((\cS,\cM,M_n)) \leq 4$.
\end{proof}

\end{document}